\documentclass[a4paper]{amsart}

\usepackage{amsfonts}
\usepackage{amssymb}
\usepackage{amsmath,mathtools}
\usepackage[T1]{fontenc}

\usepackage{color}
\usepackage{todonotes}

\usepackage{constants}

\usepackage{palatino}

\usepackage{array}
\usepackage{esint}
\usepackage{microtype}
\usepackage[alwaysadjust]{paralist}
\usepackage[
pdftitle={Density of bounded maps in Sobolev spaces into complete manifolds},%
pdfauthor={Pierre Bousquet, Augusto Ponce and Jean Van Schaftingen}]{hyperref}

\usepackage[abbrev]{amsrefs}

\usepackage[utf8]{inputenc}

\usepackage{tikz}

\newcommand{\abs}[1]{\mathopen\lvert#1\mathclose\rvert}
\newcommand{\bigabs}[1]{\bigl\lvert#1\bigr\rvert}
\newcommand{\Bigabs}[1]{\Bigl\lvert#1\Bigr\rvert}

\newcommand{\norm}[1]{\mathopen\lVert#1\mathclose\rVert}
\newcommand{\bignorm}[1]{\mathopen\big\lVert#1\mathclose\big\rVert}
\newcommand{\floor}[1]{\lfloor#1\rfloor}
\newcommand{\N}{{\mathbb N}}
\newcommand{\R}{{\mathbb R}}

\newcommand{\Sphere}{{\mathbb S}}
\newcommand{\st}{:}
\newcommand{\cE}{\mathcal{E}}

\newcommand{\cG}{\mathcal{G}}
\newcommand{\cK}{\mathcal{K}}
\newcommand{\cS}{\mathcal{S}}

\DeclareMathOperator{\supp}{supp}

\DeclareMathOperator{\jac}{jac}
\DeclareMathOperator{\dist}{dist}
\DeclareMathOperator{\Dist}{Dist}

\newcommand{\Id}{\mathrm{Id}}

\newcommand{\dif}{\,\mathrm{d}}

\theoremstyle{plain}
\newtheorem{proposition}{Proposition}[section]
\newtheorem{lemma}[proposition]{Lemma}
\newtheorem{theorem}{Theorem}
\newtheorem{corollary}[proposition]{Corollary}

\newtheorem{claim}{Claim}
\newtheoremstyle{addendumstyle}{\topsep}{\topsep}{\itshape}{}{\bfseries}{.}{.5em plus 1pt minus 1pt}{#1 #2 to #3}
\theoremstyle{addendumstyle}

\newenvironment{proofclaim}[1][Proof of the claim]{\begin{proof}[#1]}{\end{proof}}

\theoremstyle{definition}
\newtheorem{definition}{Definition}[section]
\theoremstyle{remark}

\newtheoremstyle{claimstyle}{\topsep}{\topsep}{}{}{\bfseries}{.}{.5em plus 1pt minus 1pt}{#1}
\theoremstyle{claimstyle}

\numberwithin{equation}{section}

\date{\today}

\title[Density of bounded maps in Sobolev spaces]{Density of bounded maps in Sobolev spaces into complete manifolds}

\author[P. Bousquet]{Pierre Bousquet}

\address{
Pierre Bousquet\hfill\break\indent
Universit{\'e} de Toulouse \hfill\break\indent
Institut de Math\'ematiques de Toulouse, UMR CNRS 5219\hfill\break\indent
Universit\'e Paul Sabatier Toulouse 3\hfill\break\indent 
118 Route de Narbonne\hfill\break\indent
31062 Toulouse Cedex 9\hfill\break\indent
France}

\author[A. C. Ponce]{Augusto C. Ponce}

\address{
Augusto C. Ponce\hfill\break\indent
Universit{\'e} catholique de Louvain\hfill\break\indent
Institut de Recherche en Math{\'e}matique et Physique\hfill\break\indent
Chemin du cyclotron 2, bte L7.01.02\hfill\break\indent
1348 Louvain-la-Neuve\hfill\break\indent
Belgium}

\author[J. Van Schaftingen]{Jean Van Schaftingen}

\address{
Jean Van Schaftingen\hfill\break\indent
Universit{\'e} catholique de Louvain\hfill\break\indent
Institut de Recherche en Math{\'e}matique et Physique\hfill\break\indent
Chemin du cyclotron 2, bte L7.01.02\hfill\break\indent
1348 Louvain-la-Neuve\hfill\break\indent
Belgium}

\begin{document}

\begin{abstract}
Given a complete noncompact Riemannian manifold \(N^n\), we investigate whether the set of bounded Sobolev maps \((W^{1, p} \cap L^\infty) (Q^m; N^n)\) on the cube \(Q^{m}\) is strongly dense in the Sobolev space \(W^{1, p} (Q^m; N^n)\)
for \(1 \le p \le m\). The density always holds when \(p\) is not an integer.
When \(p\) is an integer, the density can fail, and we prove that a quantitative \textit{trimming property} is equivalent with the density. 
This new condition is ensured for example by a uniform Lipschitz geometry of \(N^{n}\). 
As a byproduct, we give necessary and sufficient conditions for the strong density of the set of smooth maps \(C^\infty (\overline{Q^m}; N^n)\) in \(W^{1, p} (Q^m; N^n)\).
\end{abstract}



\maketitle


\section{Introduction}

Bounded maps from the unit cube \(Q^m =(-1, 1)^m \subset \R^m\) with \(m \in \N_* = \{1, 2, \dotsc\}\) are dense in the class of Sobolev maps \(W^{1,p}(Q^m ; \R^\nu)\), and this follows from a straightforward truncation argument. 
In the setting of Sobolev maps with values into manifolds, this elementary approach is unable to handle additional constraints on the target.

More precisely, given a closed smooth submanifold \(N^n  \subset \R^\nu\), we define
the class of Sobolev maps with values into \(N^n\) as
\[
  W^{1, p}(Q^m ; N^n)
  = \Bigl\{ u \in  W^{1, p}(Q^m ; \R^\nu) : u \in N^n  \ \text{almost everywhere} \Bigr\};
\]
the space \(L^\infty (Q^m; N^n)\) of essentially bounded maps is defined similarly.
The question addressed in the present work is whether the set \((W^{1, p} \cap L^\infty)(Q^m ; N^n)\) is dense in \(W^{1, p}(Q^m ; N^n)\) with respect to the strong \(W^{1, p}\) topology.

When \(N^{n}\) is an abstract complete smooth Riemannian manifold, there exists an isometric embedding \(\iota : N^n \to \R^\nu\) such that \(\iota(N^{n})\) is closed~\cites{Nash-1956,Mueller-2009}.
This allows one to define the functional spaces \(W^{1, p} (Q^m; N^n)\) and \(L^\infty (Q^m; N^n)\), and
different embeddings of \(N^n\) yield homeomorphic spaces; see Section~\ref{sectionMapsCompleteManifold} below.
We thus consider indifferently the case where \(N^{n}\) is an embedded closed smooth submanifold of \(\R^{\nu}\) or an abstract complete smooth Riemannian manifold.

One of our motivations to the question above comes from the density problem of the set \(C^\infty(\overline{Q^m}; N^n)\) of smooth maps  in Sobolev spaces \(W^{1, p}(Q^m ; N^n)\) with values into complete manifolds.
 Even when \(N^{n}\) is compact, this is a delicate problem that has been studied by many authors.
Schoen and Uhlenbeck~\cite{Schoen-Uhlenbeck} established the density when \(p \ge m\), and Bethuel~\cite{Bethuel}  (see also Hang and Lin~\cite{Hang-Lin}) proved  in the case \(1 \le p < m\) that density holds if and only if the homotopy group \(\pi_{\floor{p}}(N^n)\) is trivial.
The latter condition means that every continuous map \(f : \Sphere^{\floor{p}} \to N^n\) on the \(\floor{p}\)-dimensional sphere is homotopic to a constant map, where \(\floor{p}\) is the largest integer less than or equal to \(p\).
For complete manifolds, the same conclusions hold provided that \(W^{1, p}(Q^m ; N^n)\) is replaced by the smaller space \((W^{1, p} \cap L^\infty)(Q^m; N^n)\); see Section~\ref{sectionDensitySmoothMaps} below.
The strong density of smooth maps in \(W^{1, p}(Q^m ; N^n)\) is thus equivalent to the density of \((W^{1, p} \cap L^\infty)(Q^m; N^n)\) in the latter space. 

When \(p > m\), Sobolev maps on the cube \(Q^m\) are bounded, and even H\"older continuous, whence \(W^{1, p} \cap L^{\infty} = W^{1, p}\). 
When \(p \leq m\), we establish that the density of bounded Sobolev maps depends on whether \(p\) is an integer or not.

\begin{theorem}\label{theoremMainNonInteger}
For every \(1 \le p \leq m\) such that \(p \not\in \N\), the set \((W^{1, p}\cap L^{\infty})(Q^m ; N^n)\)  is dense in \(W^{1,p}(Q^m ; N^n)\).
\end{theorem}
The case where \(p\) is an integer is more subtle and the answer involves analytical properties of the manifold \(N^n\).
This surprising phenomenon arises even in the case \(p = m\).
In the related problem of density of smooth maps in \(W^{1, m}(Q^m; N^n)\) when \(N^n\) is a compact manifold, this critical case always has an affirmative answer, regardless of \(\pi_{m}(N^{n})\), and is a straightforward consequence of the fact that \(W^{1, m}\) maps embed into the class of vanishing mean oscillation (VMO) maps \citelist{\cite{Brezis-Nirenberg}\cite{Schoen-Uhlenbeck}}.
For noncompact manifolds, this VMO property is not sufficient to imply the density of bounded maps in \(W^{1, m}\), even if \(N^n\) is diffeomorphic to the Euclidean space \(\R^n\); see Section~\ref{sectionCounterexample} below.
In fact, for integer exponents \(p\) this density problem is equivalent to the following analytical assumption on the target \(N^{n}\):

\begin{definition}
\label{definitionExtensionProperty}
Given \(p \in \N_{*}\), the manifold \(N^n\) satisfies the \emph{trimming property of dimension \(p\)} whenever
there exists a constant \(C > 0\) such that each map \(f \in C^\infty(\partial Q^p; N^n)\) with a Sobolev extension \(u \in W^{1, p}(Q^{p}; N^n)\) also has a smooth extension \(v \in C^\infty(\overline{Q^{p}} ; N^n)\) such that
\[
\norm{Dv}_{L^{p}(Q^{p})}\leq C \norm{Du}_{L^{p}(Q^{p})}.
\]
\end{definition}

The use of \(C^{\infty}\) maps is not essential, and other classes like Lipschitz maps or continuous Sobolev maps (\(W^{1, p} \cap C^{0}\)) yield equivalent definitions of the trimming property; see e.g.~Proposition~\ref{proposition_continuous_trimming_property}.
The trimming property is satisfied by any manifold \(N^n\) with uniform Lipschitz geometry in the sense of Definition~\ref{definition_bounded_geometry} below, for example when \(N^n\) is the covering space of a compact manifold~\cites{Bethuel-Chiron, Pakzad-Riviere} or when \(N^n\) is a Lie group or a homogeneous space~\cites{DaiShojiUrakawa1997,Uhlenbeck1989}.
We also observe that every complete manifold satisfies the trimming property of dimension \(1\):
it suffices to take as \(v\) any shortest geodesic connecting the points \(f(-1)\) and \(f(1)\).   

The answer to the density problem for integer exponents can now be stated as follows:

\begin{theorem}\label{theorem_Ap_CNS}
For every \(p \in \{1, \dotsc, m\}\), the set \((W^{1, p} \cap L^{\infty})(Q^m ; N^n)\) is  dense in \(W^{1, p}(Q^m ; N^n)\)  if and only if \(N^n\) satisfies the trimming property of dimension \(p\).
\end{theorem}

As a consequence of Theorems~\ref{theoremMainNonInteger} and~\ref{theorem_Ap_CNS}, we characterize the class of complete manifolds \(N^n\) for which smooth maps are dense in the space \(W^{1,p}(Q^m ; N^n)\).
For non-integer values of the exponent \(p\) we have:

\begin{corollary}\label{corollaryDensitySmoothNonInteger}
For every \(1 \le p \leq m\) such that \(p \not\in \N\), the set \(C^{\infty}(\overline{Q^m} ; N^n)\) is dense in \(W^{1,p}(Q^m ; N^n)\) if and only if \(\pi_{\floor{p}}(N^n)\) is trivial.
\end{corollary}

For integer values of \(p\), the characterization becomes:
\begin{corollary}\label{corollaryDensitySmoothInteger}
\begin{itemize}
\item[\textbf{Case \boldmath$p=1$:}] The set \(C^{\infty}(\overline{Q^m} ; N^n)\) is  dense in \(W^{1, 1}(Q^m ; N^n)\)  if and only if \(\pi_{1}(N^n)\) is trivial.
\item[\textbf{Case \boldmath$p \in \{2, \dotsc, m-1\}$:}] The set \(C^{\infty}(\overline{Q^m} ; N^n)\) is  dense in \(W^{1, p}(Q^m ; N^n)\)  if and only if \(\pi_{p}(N^n)\) is trivial and \(N^n\) satisfies the trimming property of dimension \(p\).
\item[\textbf{Case \boldmath$p=m$:}] The set \(C^{\infty}(\overline{Q^m} ; N^n)\) is dense in \(W^{1, m}(Q^m ; N^n)\)  if and only if  \(N^n\) satisfies the trimming property of dimension \(m\).
\end{itemize}
\end{corollary}

In the more general setting of Sobolev maps into noncompact metric spaces~\cites{HKST, Hajlasz}, Haj\l asz and Schikorra~\cite{Hajlasz-Schikorra} have recently given a necessary condition for the density of Lipschitz maps in terms of an \((m-1)\)-Lipschitz connectedness property.

The strategy of the proofs of Theorems~\ref{theoremMainNonInteger} and~\ref{theorem_Ap_CNS} above is based on the \emph{good} and \emph{bad cube} method introduced by Bethuel~\cite{Bethuel} for a compact manifold \(N^n\).
More precisely, we divide the domain \(Q^m\) as a union of small cubes and we approximate a  map \(u \in W^{1,p}(Q^m ; N^n)\) in two different ways, depending on the properties of \(u\)  on each cube.

On the \emph{good cubes},  we  approximate \(u\)  by convolution with a smooth kernel. 
In general, such an approximation does not take its values in \(N^n\), so that we must project it back on the target manifold using a retraction \(\Pi\) on \(N^n\). Such a strategy works when 
\begin{compactitem}
\item the retraction \(\Pi\) is well-defined on a tubular neighborhood of positive \emph{uniform} width around \(N^n\),
\item the convolution of \(u\) with a smooth kernel takes its values in this tubular neighborhood.
\end{compactitem}

The first condition is automatically satisfied when \(N^n\) is compact, while the second one holds true when \(u\) has small mean oscillation.
In Bethuel's and Hang--Lin's works, the latter condition on \(u\) is used to define good cubes (see also \cite{Bousquet-Ponce-VanSchaftingen}*{p.~797}) in the sense that a cube \(\sigma^{m}_{\eta}\) of inradius \(\eta\) is good if the rescaled energy satisfies
\[{}
\frac{1}{\eta^{m - p}} \int_{\sigma_{\eta}^{m}}{\abs{Du}^{p}}
\lesssim \delta,
\]
for some small parameter \(\delta > 0\) depending on the width of the tubular neighborhood; the inradius is half the edge-length of the cube.
The connection between such a condition and the oscillation of \(u\) on \(\sigma_{\eta}^{m}\) can be made using the Poincaré--Wirtinger inequality:
\[{}
\fint_{\sigma_{\eta}^{m}}\fint_{\sigma_{\eta}^{m}}{\abs{u(x) - u(y)}} \dif x \dif y
\le
\frac{C}{\eta^{m - p}} \int_{\sigma_{\eta}^{m}}{\abs{Du}^{p}}.
\]

Noncompact submanifolds, however, need not have a global tubular neighborhood with positive uniform width. 
For instance, it is impossible to find such a tubular neighborhood for any isometric embedding of the hyperbolic space in a Euclidean space because the volumes of hyperbolic balls grow exponentially with respect to the radius. 
Since geodesic balls in \(N^{n}\) do have a tubular neighborhood with uniform width, we thus modify the classical definition of a good cube by further requiring that most of the points of \(u\) lie in a fixed geodesic ball.
It is then  possible to  perform the approximation by convolution and projection.
The parameter scale of this convolution is not constant but depends   on the distance to the bad cubes. Such an \emph{adaptive smoothing} (Section~\ref{sectionAdaptiveSmoothing}) is used to smoothen the transition from good cubes to bad cubes.

On the bad cubes, we modify \(u\) using the \emph{opening technique} (Section~\ref{sectionOpening}). 
This operation was introduced by Brezis and Li~\cite{Brezis-Li}, and was then pursued by the authors in~\cite{Bousquet-Ponce-VanSchaftingen} in the framework of higher order Sobolev spaces.
More precisely, given an \(\ell\)-dimensional grid in the bad cubes with \(\ell \in \{0, \dotsc, m - 1\}\), we use the opening technique to slightly modify \(u\) near the grid to obtain a new map \(u^{\mathrm{op}}\) that locally depends on \(\ell\)~variables and whose restriction to the grid belongs to \(W^{1, p}\). 

Taking \(\ell = \floor{p}\) for \(p < m\), we next perform a  \emph{zero-degree homogenization} (Section~\ref{sectionHomogenization}) which consists in propagating the values of \(u^{\mathrm{op}}\) on the \(\floor{p}\)-dimensional grid to all \(m\)-dimensional bad cubes~\cites{Bethuel-Zheng, Bethuel, Hang-Lin}. 
When \(p \not\in \N_{*}\), the Morrey--Sobolev embedding implies that \(u^{\mathrm{op}}\) is bounded on the \(\floor{p}\)-dimensional grid, and we end up with a bounded Sobolev map on the bad cubes.
When \(p \in \N_{*}\), the resulting map need not be bounded, so before applying the zero-degree homogenization, we need to do some preliminary work that relies on the \emph{trimming property of dimension \(p\)}. 
Indeed, since \(u^{\mathrm{op}}\) is continuous on the \((p - 1)\)-dimensional grid, the trimming property allows one to replace \(u^{\mathrm{op}}\) on the \(p\)-dimensional grid by a bounded map which coincides with \(u^{\mathrm{op}}\) on the \((p - 1)\)-dimensional grid; see Section~\ref{sectionExtensionProperty}.
The resulting map obtained by zero-degree homogenization is now bounded as in the non-integer case.
Although we obtain a function which can be quite different from \(u\) on the bad cubes, we can conclude using the fact that most of the cubes are good.

We now describe the plan of the paper. 
In Section~\ref{sectionMapsCompleteManifold}, we explain why the definition of the Sobolev space  \(W^{1,p}(Q^m ; N^n)\) is independent of the specific isometric embedding of \(N^n\). 
In Section~\ref{sectionDensitySmoothMaps}, we investigate the density of smooth maps \(C^{\infty}(\overline{Q^m}; N^n)\) in \((W^{1,p}\cap L^{\infty})(Q^m;N^n)\) using results from the case of compact-target manifolds. 
In Section~\ref{sectionCounterexample}, we present a counterexample to the density of bounded maps in the class \(W^{1,m}(Q^m ; N^n)\), where \(N^n\) is a suitable embedding of \(\R^n\) in \(\R^{n+1}\). In Section~\ref{sectionMainTools}, we have collected the main tools that will be used in the proofs of Theorems~\ref{theoremMainNonInteger} and~\ref{theorem_Ap_CNS}: the opening technique, the adaptive smoothing and the zero-degree homogenization. 
In Section~\ref{sectionExtensionProperty}, we give equivalent formulations of the trimming property and we establish the necessity of this condition for the density of bounded maps when \(p \in \N_{*}\). 
Finally, we present the proofs of Theorems~\ref{theoremMainNonInteger} and~\ref{theorem_Ap_CNS} in Section~\ref{sectionProofs}.


\section{Maps into a complete manifold}%
\label{sectionMapsCompleteManifold}%
\subsection{Sobolev maps}%
\label{sectionIntrinsicSobolevMaps}%
In the definition of the Sobolev spaces  introduced above, we have used an isometric embedding \(\iota : N^n \to \R^{\nu}\) of the target smooth Riemannian manifold \(N^n\).{}
We claim that such a definition does not depend on the embedding, in the following sense: if \(\iota_1\) and \(\iota_2\) are two isometric embeddings of \(N^n\) into \(\R^{\nu_1}\) and \(\R^{\nu_2}\) respectively, then the two resulting Sobolev spaces are homeomorphic.
For the convenience of the reader, we provide below a self-contained proof of this fact.
Alternatively, it is possible to define intrinsically  Sobolev spaces of maps into complete Riemannian manifolds without reference to any isometric embedding of the target manifold \(N^n\). 
Such an approach turns out to be equivalent to the definition that relies on an embedding of \(N^n\), see  \cite{ConventVanSchaftingen2016}*{Propositions~2.7 and~4.4} and also~\cite{FocardiSpadaro2013}.

\begin{proposition}
\label{propositionIntrinsicSobolev}
Let \(\iota_k : N^n \to \R^{\nu_k}\), with \(k \in \{1, 2\}\), be two isometric embeddings of a complete Riemannian manifold \(N^n\) and let \(1 \le p < \infty\).
\begin{enumerate}[\((i)\)]
\item Given a measurable map \(u : Q^m \to N^n\), \(\iota_1 \circ u \in W^{1, p} (Q^m; \R^{\nu_1})\) if and only if \(\iota_2 \circ u \in W^{1, p} (Q^m; \R^{\nu_2})\).
\item Given a sequence of measurable maps \(u_j : Q^m \to N^n\), with \(j \in \N\), we have that 
\((\iota_1 \circ u_j)_{j \in \N}\) converges to \(\iota_1 \circ u\) in \(W^{1, p} (Q^m; \R^{\nu_1})\) if and only if \((\iota_2 \circ u_j)_{j \in \N}\) converges to \(\iota_2 \circ u\) in \(W^{1, p} (Q^m; \R^{\nu_2})\).{}
\end{enumerate}
\end{proposition}

In the proof of Proposition~\ref{propositionIntrinsicSobolev}, we need a specific version of the chain rule for functions of the form \(\Phi\circ w\), where \(w\) is a Sobolev map and \(\Phi\) is a \(C^1\) map such that \(D\Phi\) is bounded on the range of \(w\), in the spirit of the composition formula in~\cite{Marcus-Mizel}*{Theorem~2.1}:

\begin{lemma}\label{lemma_chain_rule_C1}
Let \(M^{n} \subset \R^{\nu_{1}}\) be an embedded complete Riemannian submanifold and \(\Phi \in C^{1}(M^{n}; \R^{\nu_{2}})\). 
Then, for every \(w \in W^{1, p}(Q^m; M^{n})\) such that \(D\Phi \circ w \in L^{\infty}(Q^{m})\),  we have \(\Phi \circ w \in W^{1,p}(Q^m ; \R^{\nu_{2}})\) and, for almost every \(x \in Q^m\), 
\[{}
D(\Phi\circ w)(x) = D\Phi(w(x))\circ Dw(x).
\] 
\end{lemma}

The proof is based on an adaptation of the argument in~\cite{Marcus-Mizel} and relies on Morrey's characterization of Sobolev maps.
To deduce Proposition~\ref{propositionIntrinsicSobolev}, we apply this lemma to the map \(\Phi = \iota_{2}\circ \iota_{1}^{-1}\), which need not be globally Lipschitz-continuous as in usual versions of the composition formula.

\begin{proof}[Proof of Lemma~\ref{lemma_chain_rule_C1}]
The fact that \(\Phi\circ w \in L^{p}(Q^m ; \R^{\nu_{2}})\) is a consequence of the following Poincar\'e--Wirtinger inequality:
\begin{equation}\label{eq_Poincare_Pi}
\int_{Q^m}\int_{Q^m}\abs{\Phi\circ w(x)-\Phi\circ w(y)}^p\dif x \dif y \leq C
\int_{Q^m}\abs{Dw(x)}^p\dif x,
\end{equation}
where \(C > 0\) depends on \(\norm{D\Phi \circ w}_{L^{\infty}(Q^{m})} = \norm{D\Phi}_{L^{\infty}(F)}\), and \(F\) is the essential range of \(w\).{}
We recall that the essential range is the smallest closed subset \(F \subset M^{n}\) such that \(w(x) \in F\) for almost every \(x \in Q^{m}\), see~\cite{Brezis-Nirenberg}*{Section~I.4}.

To prove \eqref{eq_Poincare_Pi}, we observe that, for almost every \(x'\in Q^{m-1}\),  we have \(w(\cdot, x')\in W^{1,p}(Q^1 ; M^{n})\). 
By the Morrey--Sobolev embedding, \(w(\cdot, x')\) can be identified to a continuous (and thus bounded) map on \(Q^1\) with values into \(M^{n}\). 
Hence, we can  apply the classical chain rule to \(\Phi\circ w(\cdot, x')\), because the restriction of \(\Phi\) to \(w(\cdot, x')(Q^1)\) is globally Lipschitz-continuous. 
Hence, \(\Phi\circ w(\cdot, x')\in W^{1,p}(Q^1;\R^{\nu_{2}})\) and 
\begin{equation}
	\label{eqMorreyOneDimensional}
\frac{\mathrm{d}}{\mathrm{d} x_{1}}(\Phi\circ w)(x_{1}, x')
= D\Phi(w(x_1, x'))[\partial_1 w(x_1, x')].
\end{equation}
Since \(D\Phi\) is bounded on the essential range \(F\), this implies
\begin{equation}\label{eq_Sobolev_line}
\resetconstant
\Bigabs{\frac{\mathrm{d}}{\mathrm{d} x_{1}}(\Phi\circ w)(x_{1}, x') }
\leq \Cl{cte-2201} \abs{\partial_1 w(x_1, x')},
\end{equation}
for some constant \(\Cr{cte-2201} > 0\) independent of \(x'\).
The standard one-dimensional Poincaré--Wirtinger inequality applied to \(\Phi\circ w(\cdot, x')\) thus yields
\[
\int_{Q^1} \int_{Q^1} \abs{\Phi\circ w(t,x')-\Phi\circ w(s,x')}^p\dif t \dif s
\leq \Cl{cte-2202} \int_{Q^1}\abs{\partial_1 w(t,x')}^p\dif t.
\]
By integrating over \(x'\in Q^{m-1}\), one gets
\[
\int_{Q^{m-1}}\int_{Q^1} \int_{Q^1} \abs{\Phi\circ w(t,x')-\Phi\circ w(s,x')}^p\dif t \dif s \dif x'
\leq \Cr{cte-2202} \int_{Q^m}\abs{\partial_1 w(t,x')}^p\dif t\dif x'.
\]
The same calculation can be performed for every coordinate. Using the triangle inequality, one has the estimate
\begin{multline*}
\abs{\Phi\circ w(x)-\Phi\circ w(y)}^p \leq \C
\sum_{i=1}^{m}\abs{\Phi\circ w(y_1, \dots, y_{i-1}, x_i, x_{i+1}, \dots, x_m)\\-\Phi\circ w (y_1, \dots, y_{i-1}, y_i, x_{i+1}, \dots, x_m)}^p.
\end{multline*}
By integration, one obtains the Poincaré--Wirtinger inequality~\eqref{eq_Poincare_Pi}. 

We now prove that \(\Phi\circ w \in W^{1,p}(Q^m ; \R^{\nu_{2}})\).  
For almost every \(x'\in Q^{m-1}\), by estimate~\eqref{eq_Sobolev_line} we have that 
\[
\int_{Q^m}\Bigabs{\frac{\mathrm{d}}{\mathrm{d} x_{1}}(\Phi\circ w)(x_{1}, x') }^p\dif x \leq 
(\Cr{cte-2201})^{p} \int_{Q^m} \abs{\partial_1 w(x)}^p\dif x.
\]
Since this is true for every coordinate, Morrey's characterization of Sobolev maps~\cite{Evans-Gariepy}*{Theorem~4.21} implies that \(\Phi\circ w \in W^{1,p}(Q^m ; \R^{\nu_{2}})\) and, for almost every \(x\in Q^m\), it follows from the counterpart of identity~\eqref{eqMorreyOneDimensional} for each coordinate that
\[{}
D(\Phi\circ w)(x)=D\Phi(w(x))\circ Dw(x).
\qedhere
\]
\end{proof}

\begin{proof}[Proof of Proposition~\ref{propositionIntrinsicSobolev}]
Let \(u:Q^m \to N^n\) and let us assume that \(\iota_1\circ u \in W^{1,p}(Q^m;\R^{\nu_1})\). 
The smooth map \(\Phi:=\iota_2\circ \iota_{1}^{-1}\) is defined on the embedded complete submanifold \(\iota_1(N^n) \subset \R^{\nu_{1}}\) with values into \(\iota_2(N^n) \subset \R^{\nu_{2}}\). 
Since \(\iota_2\circ \iota_{1}^{-1}\) is an isometry, \(D\Phi\) is bounded on the tangent bundle \(T(\iota_{1}(N^n))\). 
We can thus apply Lemma~\ref{lemma_chain_rule_C1} to \(w = \iota_1\circ u\). 
This implies that \(\iota_2\circ u \in W^{1,p}(Q^m ; \R^{\nu_2})\) and, for almost every \(x\in Q^m\), we have 
\[{}
D(\iota_2\circ u)(x)=D\Phi((\iota_1\circ u)(x))\circ D(\iota_1\circ u)(x).
\]

Let  \(u_j : Q^m \to N^n\) be a sequence of measurable maps such that
\((\iota_1 \circ u_j)_{j \in \N}\)  converges to \(\iota_1 \circ u\) in \(W^{1, p} (Q^m; \R^{\nu_1})\). 
This implies the convergence in measure of the functions \(\iota_1\circ u_j\) and their derivatives. 
Since \(\Phi\) is \(C^1\), one deduces the same convergence in measure for \(\iota_2\circ u_j=\Phi(\iota_1\circ u_j)\).

For every \(j\geq 1\) and almost every \(x\in Q^m\), the quantity \(\abs{D(\iota_2\circ u_j)(x) - D(\iota_2\circ u)(x)}^p\), which is equal to \(\abs{D(\Phi\circ \iota_1\circ u_j)(x)-D(\Phi\circ \iota_{1}\circ u)(x)}^p\), is dominated by 
\[{}
\resetconstant
\C \big( \abs{D(\Phi\circ \iota_1\circ u_j)(x)}^p  + \abs{D(\Phi\circ \iota_1\circ u)(x)}^p\big).
\]
Lemma~\ref{lemma_chain_rule_C1} and the boundedness of \(D\Phi\) on \(T(\iota_{1}(N^n))\) yield 
\[
\abs{D(\iota_2\circ u_j)(x) - D(\iota_2\circ u)(x)}^p \leq \C \big( \abs{D(\iota_1\circ u_j)(x)}^p + \abs{D(\iota_1\circ u )(x)}^p\big).
\]
Since \((D(\iota_2\circ u_j))_{j \in \N}\) converges to \(D(\iota_2\circ u )\) in measure and the right-hand side converges in \(L^{1}(Q^{m})\), the dominated convergence theorem implies that \(( D(\iota_2\circ u_j))_{j \in \N}\) converges to \(D(\iota_2\circ u)\) in \(L^{p}(Q^m ; \R^{m \times \nu_2})\). 
 
The convergence of \((\iota_2\circ u_j)_{j \in \N}\)  to \(\iota_2\circ u\) in \(L^{p}(Q^m;\R^{\nu_2})\) follows from the convergence in measure and the convergence of the derivatives in \(L^{p}\). 
Indeed, for every \(\varepsilon>0\) we have
\[
\int_{Q^m}\abs{\iota_2\circ u_j-\iota_2\circ u}^p 
\leq (2\varepsilon)^p\abs{Q_m} + \int\limits_{\{\abs{\iota_2\circ u_j-\iota_2\circ u}\geq 2\varepsilon\}} \abs{\iota_2\circ u_j-\iota_2\circ u}^p.
\]
By the convergence in measure, for every \(j \in \N\) large enough we have \(\abs{\{\abs{\iota_2\circ u_j-\iota_2\circ u}\geq \varepsilon\}}\leq \abs{Q^m}/2\).
Defining \(\theta_{\varepsilon} : [0, +\infty) \to \R\) by
\[
\theta_{\varepsilon}(t):=
\begin{cases}
0 & \text{ if \(t \le \varepsilon,\)}\\
2(t-\varepsilon) & \text{ if \(\varepsilon < t < 2\varepsilon,\)}\\
t & \text{ if  \(t \ge 2\varepsilon,\)}
\end{cases}
\]
it follows from the Poincaré inequality for functions vanishing on a set of positive measure that
\begin{align*}
\int\limits_{\{\abs{\iota_2\circ u_j-\iota_2\circ u}\geq 2\varepsilon\}} \abs{\iota_2\circ u_j-\iota_2\circ u}^p 
&\leq \int_{Q^m}\big(\theta_{\varepsilon}(\abs{\iota_2\circ u_j-\iota_2\circ u})\big)^p\\
&\leq \C \int_{Q^m}\abs{D(\iota_2\circ u_j-\iota_2\circ u)}^p.
\end{align*}
By the convergence of \((D(\iota_2\circ u_j))_{j \in \N}\) to \(D(\iota_2\circ u)\) in \(L^{p}(Q^m; \R^{m \times \nu_{2}})\), we then have
\[
\limsup_{j \to \infty}\int_{Q^m}\abs{\iota_2\circ u_j-\iota_2\circ u}^p \leq (2\varepsilon)^p\abs{Q_m}.
\]
Since \(\varepsilon\) is arbitrary, this proves the convergence of \((\iota_2\circ u_j)_{j \in \N}\) to \(\iota_2\circ u\) in \(L^{p}(Q^m; \R^{\nu_{2}})\). The proof is complete.
\end{proof}

\subsection{Bounded maps}%
\label{sectionIntrinsicBoundedMaps}
Let \(N^n\) be an (abstract) complete Riemannian manifold.
We say that a measurable map \(u : Q^m \to N^n\) is \emph{essentially bounded} if it is essentially bounded for the geodesic distance \(\dist_{N^{n}}\) induced by the Riemannian metric on \(N^n\): there exists \(C > 0\)
such that, for almost every \(x, y \in Q^m\),
\(\dist_{N^n} (u (x), u (y))\le C\).
Since \(N^n\) is complete, this is equivalent to the existence of a compact set \(K \subset N^n\) such that, for almost every \(x \in Q^m\), \(u (x) \in K\).

We now  consider an isometric embedding \(\iota : N^n \to \R^\nu\).
If the map \(u\) is essentially bounded, then \(\iota \circ u\in L^{\infty}(Q^m ; \R^\nu)\). 
In general, the converse is not true because there exists \(\nu \in \N_{*}\), depending on the dimension \(n\), such that the manifold \(N^n\) can  be isometrically embedded inside a ball of any radius \(r > 0\) in \(\R^{\nu}\), see \citelist{\cite{Nash-1954}*{Theorem~2}\cite{Nash-1956}*{Theorem~3}}.

We can discard this phenomenon under the additional assumption that the embedding \(\iota : N^n \to \R^\nu\) is a \emph{proper map} from \(N^n\) to \(\R^\nu\); that is,  for every compact set \(L \subset \R^\nu\), the set \(\iota^{-1} (L)\) is a compact subset of \(N^{n}\).
Since \(\iota\) is a homeomorphism from \(N^n\) to \(\iota (N^n)\), this amounts to the property that \(\iota (N^n)\) is a \emph{closed} subset of \(\R^\nu\).
Such proper isometric embeddings of complete Riemaniann manifolds have been constructed by a shrewd application of the classical Nash embedding theorem \cite{Mueller-2009}.

The next proposition summarizes the preceding discussion:

\begin{proposition}%
\label{propositionIntrinsicBounded}
If \(N^n\) is a complete Riemannian manifold, then there exists an isometric embedding \(\iota : N^n \to \R^\nu\) such that \(\iota (N^n)\) is closed.
For such an embedding, the map \(u : Q^m \to N^n\) is essentially bounded if and only if the map \(\iota \circ u : Q^m \to \R^\nu\) is essentially bounded.
\end{proposition}

In the sequel, \emph{we work exclusively with proper isometric embeddings \(\iota : N^n \to \R^\nu\), and we systematically identify \(N^n\) with \(\iota(N^n)\)}.{}

\subsection{The nearest point projection}
\label{section_nearest_point_projection}
An important tool for the approximation of Sobolev maps when \(N^n\) is a closed submanifold of \(\R^\nu\) is the fact that \(N^n\) is a  smooth retraction of a neighborhood of itself. 
More precisely,  the \emph{nearest point projection} \(\Pi\) is well-defined and smooth on an open neighborhood \(O\subset \R^\nu\) of \(N^n\). The map \(\Pi\in C^{\infty}(O ; N^n)\)  satisfies the following properties: 
\begin{enumerate}[\((a)\)]
\item for every \((y, z)\in O\times N^n\), \(\dist_{\R^{\nu}}(y, \Pi(y)) \leq \dist_{\R^{\nu}}(y, z)\);
\item in particular, \(\Pi(y)=y\) for every \(y\in N^n\).
\end{enumerate}
It follows that \(D\Pi\) is bounded on \(N^n\). 
For our purposes in this paper, the map \(\Pi\) could be replaced by any  retraction \(\widetilde{\Pi}\in C^1(O ; N^n)\) with bounded derivative. The existence of such a map only requires \(N^n\) to be \(C^1\), as a consequence of \cite{Whitney}*{Theorem~10A}. In contrast, the nearest point projection onto a \(C^1\) submanifold is merely continuous in general.

Reducing the size of \(O\) if necessary, we note that
\begin{enumerate}[(1)]
	\item{}
	\label{item-0639}
	 by continuity of \(D\Pi\) one can assume that \(D\Pi\) is bounded on \(O\), although this does not imply that the map \(\Pi\) is globally Lipschitz-continuous  on \(O\),{}
	\item{}
	\label{item-0640}
	by closedness of \(N^{n}\), the map \(\Pi : O \to N^{n}\) can be extended as a smooth map from \(\R^{\nu}\) to \(\R^{\nu}\), whence the chain rule of Marcus and Mizel~\cite{Marcus-Mizel}*{Theorem~2.1} or Lemma~\ref{lemma_chain_rule_C1} above with \(M^{m} = \R^{\nu}\) can be applied to \(\Pi\circ w\) when \(w\) is a Sobolev map with values into a closed subset \(F \subset O\). 
\end{enumerate}

We assume henceforth that \(O \supset N^{n}\) is chosen so that \emph{the smooth map \(\Pi : O \to N^{n}\) satisfies Properties~\eqref{item-0639} and~\eqref{item-0640} above}.

\section{Approximation of bounded maps by smooth maps}
\label{sectionDensitySmoothMaps}
\subsection{High-integrability case}

When \(p > m\), maps in \(W^{1, p} (Q^m; N^n)\)  are essentially bounded and continuous, and can be approximated uniformly via convolution by smooth maps with values in \(\R^{\nu}\).
The nearest point projection \(\Pi\) defined in Section~\ref{section_nearest_point_projection} allows one to project the sequence back to \(N^{n}\).

\begin{proposition}\label{proposition_high_smooth_bounded_density}
If \(p > m\), then, for every \(u \in W^{1, p}(Q^m; N^n)\), there exists a sequence in \(C^\infty(\overline{Q^m}; N^n)\) converging strongly to \(u\) in \(W^{1, p}(Q^m;\R^\nu)\).
\end{proposition}

\begin{proof}
We first extend the map \(u\) by reflection on the larger cube \(Q_{2}^{m}\) with inradius equal to \(2\). 
We still denote by \(u\) the resulting map, which now belongs to \(W^{1, p}(Q_2^m; N^n)\).

Given a family of mollifiers \((\varphi_\varepsilon)_{\varepsilon > 0}\) of the form \(\varphi_\varepsilon(z) =  \varphi({z}/{\varepsilon})/{\varepsilon^m}\) with \(\varphi \in C_{c}^{\infty}(B_{1}^{m})\), 
we have  
\[
  \lim_{\varepsilon \to 0}{\norm{\varphi_\varepsilon \ast u - u}_{W^{1, p} (Q^m)}} = 0,
\]
and thus, by the Morrey--Sobolev embedding,
\[
  \lim_{\varepsilon \to 0}{\norm{\varphi_\varepsilon \ast u - u}_{L^\infty (Q^m)}} = 0.
\]
Since the map \(u\) is essentially bounded on \(Q^m\), there exists a compact set \(K \subset N^n\) such that 
\(u (x) \in K\) for almost every \(x \in Q^m\). 
Take \(\iota_K > 0\) such that \(\overline{K + B_{\iota_K}^\nu} \subset O\),
where \(O\) is the tubular neighborhood on which the nearest point projection \(\Pi\) is defined and smooth, see Section~\ref{section_nearest_point_projection}.
For every \(\varepsilon>0\) sufficiently small, \((\varphi_{\varepsilon}\ast u) (Q^m)\subset K+B_{\iota_K}^\nu\). 
Then, the \(C^1\) regularity of \(\Pi\) implies that \(\Pi \circ (\varphi_{\varepsilon}\ast u)\in W^{1,p}(Q^m ; N^n)\) and such a family of maps converges to \(\Pi \circ u = u\) in \(W^{1,p}(Q^m ; \R^\nu)\). This completes the proof.
\end{proof}

\subsection{Critical-integrability case}

When \(p = m\), maps in \(W^{1, m} (Q^m; N^n)\) need not  be continuous nor even bounded. 
However, smooth maps taking their values into the manifold \(N^{n}\) are always dense in \((W^{1, m} \cap L^\infty)(Q^m; N^n)\), and this essentially follows from the seminal work of Schoen and Uhlenbeck for compact manifolds \cite{Schoen-Uhlenbeck}. 

\begin{proposition}\label{proposition_smooth_bounded_density}
For every \(u \in (W^{1, m} \cap L^\infty)(Q^m; N^n)\), there exists a sequence in \(C^\infty(\overline{Q^m}; N^n)\) converging strongly to \(u\) in \(W^{1, m}(Q^{m}; \R^{\nu})\).
\end{proposition}

Once again, the proof of Proposition~\ref{proposition_smooth_bounded_density} is based on an approximation by a family of  mollifiers. 
In contrast to the case \(p>m\) considered in Proposition~\ref{proposition_high_smooth_bounded_density}, the family \((\varphi_{\varepsilon}\ast u)_{\varepsilon > 0}\) need not converge uniformly to \(u\). 
However, as in the setting of compact manifolds \cite{Schoen-Uhlenbeck}, one can exploit  the vanishing mean oscillation property satisfied by   maps in the critical integrability space, as has been observed in \cite{Brezis-Nirenberg}. 
Indeed, this property guarantees that the approximating sequence takes its values in a small neighborhood of \(N^n\). 
One can then  project it back on \(N^n\) by using the nearest point projection \(\Pi\).

\begin{proof}[Proof of Proposition~\ref{proposition_smooth_bounded_density}]
Let \(u\in (W^{1,m}\cap L^{\infty})(Q^m ; N^n)\). 
As in Proposition~\ref{proposition_high_smooth_bounded_density},  we may assume that \(u \in (W^{1, m} \cap L^\infty)(Q_2^m; N^n)\) and consider the maps \(\varphi_{\varepsilon}\ast u\).
Let \(K \subset N^n\) be a compact subset such that \(u(x) \in K\) for almost every \(x \in Q^m\).
By the Poincar\'e--Wirtinger inequality, for every \(x \in Q^m\) and every \(0 < \varepsilon < 1\),
\[{}
\resetconstant
\big(\dist_{\R^\nu}{((\varphi_\varepsilon * u)(x), K )}\big)^m
\le \fint_{B_\varepsilon^m(x)} \bigabs{\varphi_\varepsilon * u(x) - u(y)}^m \dif y
\le \Cl{cte-1451} \int_{B_\varepsilon^m(x)} \abs{Du}^m,
\]
for some constant \(\Cr{cte-1451} > 0\) independent of \(\varepsilon\).
The quantity in the right-hand side converges uniformly to \(0\) with respect to \(x\) as \(\epsilon\) tends to \(0\).

Taking \(\iota_K > 0\) such that \(\overline{K + B_{\iota_K}^\nu} \subset O\), we deduce from the estimate above that there exists \(\Bar{\varepsilon} > 0\) such that, for every \(0 < \varepsilon \le \Bar{\varepsilon}\) and every \(x \in Q^m\), we have
\[
\dist_{\R^\nu}{((\varphi_\varepsilon * u)(x), K )} \le \iota_K.
\] 
We can then consider the family \(\big(\Pi \circ (\varphi_{\varepsilon} * u)\big)_{0 < \varepsilon \le \Bar{\varepsilon}}\) and conclude as in the proof of Proposition~\ref{proposition_high_smooth_bounded_density}. 
\end{proof}

\subsection{Low-integrability case}

The low integrability case \(p < m\) is the most delicate, but can be settled by the results and methods used to handle the density of smooth maps when the target manifold \(N^n\) is compact.
In general, smooth maps are not dense without an additional topological assumption on \(N^n\), but a larger class of maps admitting \((m - 1 - \floor{p})\)-dimensional singularities is dense.

\begin{proposition}\label{proposition_class_R}
Let \(1\leq p < m\). 
For every \(u \in (W^{1,p}\cap L^{\infty})(Q^m ; N^n)\), there exists a sequence in \((R_{m-\floor{p}-1}\cap L^{\infty})(Q^m ;N^n)\)  converging strongly to \(u\) in \(W^{1, m}(Q^{m}; \R^{\nu})\).{}
If moreover \(\pi_{\floor{p}}(N^n)\) is trivial, then such a sequence can be taken in \(C^{\infty}(\overline{Q^m} ; N^n)\).  
\end{proposition}

In the statement above we denote by \(R_i(Q^m ; N^n)\), for \(i \in \{0,\dotsc, m-1\}\), the set of maps \(u : \overline{Q^m} \to N^n\) which are smooth on \(\overline{Q^{m}}\setminus T\) and such that, for every \(x\in \overline{Q^m}\setminus T\),
\[
\abs{D u (x)}\leq \frac{C}{\dist (x,T)},
\]
where \(T\) is a finite union of \(i\)-dimensional hyperplanes.
Here, the set \(T\) and the constant \(C > 0\) depend on \(u\).

The proof of Proposition~\ref{proposition_class_R}  relies on the next lemma which allows one to identify a bounded map into a \emph{complete} manifold as a map into a \emph{compact} manifold.

\begin{lemma}\label{lemmaNonCompactToCompact}
Let \(N^n_0\) be a smooth compact submanifold of\/ \(\R^\nu\) with boundary. 
Then, for every compact set \(K\) in the relative interior of \(N^n_0\), there exists a smooth compact submanifold \(L^n\) without boundary of\/ \(\R^{\nu} \times \R\) such that 
\[{}
K \times \{0\} \subset L^n
\quad \text{and} \quad 
P(L^n) \subset N^n_0,
\]
where \(P : \R^{\nu} \times \R \to \R^{\nu}\) denotes the projection \(P(x, s) = x\).
\end{lemma}

The idea of the proof of Lemma~\ref{lemmaNonCompactToCompact} is to glue together two identical copies of \(N^n_0\) along the boundary. 
To avoid the creation of singularities,  the gluing is done along a tube diffeomorphic to \(\partial N_0 \times [0, 1]\).{}
To avoid the intersection between the two copies, they are placed in distinct \(\nu\)-dimensional affine hyperplanes of the space \(\R^{\nu+1}\).

\begin{proof}[Proof of Lemma~\ref{lemmaNonCompactToCompact}]
Let \(K\) be a compact subset in the interior of \(N^n_0\). By the collar neighborhood theorem (see for example~\cite{Kosinski}*{Theorem 1.7.3}), there exist a relative open neighborhood \(U\) of \(\partial N^{n}_0\) in \( N^n_0\) and a smooth diffeomorphism
\[
f : \partial N^n_0 \times [0,1] \to N^n_0 \cap \overline{U}
\]
such that \(f^{-1}(\partial N^n_0) = \partial N^n_0 \times \{1\}\) and \(f^{-1}(\partial U \cap N_{0}^{n}) = \partial N^n_0 \times \{0\}\). 
By reducing the size of \(U\) if necessary, we can assume that \(U\cap K=\emptyset\).

Let \(\alpha, \beta : [0,1] \to [0,1]\) be two smooth functions such that 
\[
\alpha(t)= 
\begin{cases}
t & \text{if } t<1/4,\\
1-t & \text{if } t>3/4,
\end{cases}
\quad 
\text{and}
\quad
\beta(t)=
\begin{cases}
0 & \text{if } t<1/8,\\
1 & \text{if } t>7/8.
\end{cases}
\]
We also require that \(\beta\) be nondecreasing and \(\beta'>0\) on the interval \([1/4, 3/4]\). We now define the set
\[
L^n=\big((N^n_0\setminus U)\times \{0,1\}\big)\cup \bigl\{\bigl(f(z, \alpha(t)), \beta(t)\bigr) \st z \in \partial N^n_0,\ t\in (0,1)\bigr\}.
\]
We observe  that
\begin{multline*}
(N^n_0\times \{0\}) \cap \bigl\{\bigl(f(z, \alpha(t)\bigr), \beta(t)\bigr) \st z \in \partial N^n_0,\ t\in (0,1)\bigr\} \\ 
 = \bigl\{(f(z, t),0) \st z\in \partial N^n_0,\ t \in (0,t_0]\bigr\}, 
\end{multline*}
where \(t_0=\max{\{t \st \beta(t)=0\}}\).
Similarly, 
\begin{multline*}
(N^n_0\times \{1\}) \cap \bigl\{\bigl(f(z, \alpha(t)\bigr), \beta(t)) \st z \in \partial N^n_0,\ t\in (0,1)\bigr\} \\ 
= \bigl\{ \bigl(f(z, 1-t), 1\bigr) \st z \in \partial N^n_0,\ t\in {[t_1,1)}\bigr\},
\end{multline*}
where  \(t_1 = \min{\{t \st \beta(t)=1\}}\). 
This implies that \(L^n\) is a smooth submanifold of \(\R^{\nu+1}\). By construction, \(L^n\) is compact and has no boundary. Moreover, \(K\times \{0\}\) is contained in \(L^n\). 
The inclusion \(P(L^{n}) \subset N_{0}^{n}\) follows from the fact that \(L^{n} \subset N^{n}_0 \times [0, 1]\).
\end{proof}

In the proof of Proposition~\ref{proposition_class_R} that we present below, we first observe that the range of a map \(u \in (W^{1,p}\cap L^{\infty})(Q^m ; N^n)\) is contained in a compact set. Hence, the map \(u\) can be identified to an element of \(W^{1,p}(Q^m ; L^n)\), where \(L^{n}\) is the compact manifold given by Lemma~\ref{lemmaNonCompactToCompact}. 
For a compact target manifold, the density of the class \(R_{m-\floor{p}-1}(Q^m ; L^n)\) in \(W^{1,p}(Q^m ; L^n)\) has been proved in \citelist{\cite{Bethuel}\cite{Hang-Lin}\cite{Bousquet-Ponce-VanSchaftingen}}. 
The retraction \(P\) from Lemma~\ref{lemmaNonCompactToCompact} then allows one to bring an approximating sequence back to the original manifold \(N^n\). 

This approach cannot work for the density of \(C^{\infty}(\overline{Q^m} ; N^n)\) in  the space \((W^{1,p}\cap L^{\infty})(Q^m ; N^n)\). Indeed, there is no guarantee that the manifold \(L^n\) inherits the topological assumption satisfied by \(N^n\). 
For example, by gluing two balls \(\mathbb{B}^{n}\) one gets a manifold which is diffeomorphic to the sphere \(\Sphere^{n}\), while the homotopy group \(\pi_{n}(\Sphere^{n})\) is nontrivial.
Instead, once the density of the class \(R_{m-\floor{p}-1}(Q^m ; N^n)\) is proved, one can proceed along the lines of the proof in the compact setting~\citelist{\cite{Bethuel}\cite{Hang-Lin}\cite{Bousquet-Ponce-VanSchaftingen}}.

\begin{proof}[Proof of Proposition~\ref{proposition_class_R}]
Given \(u\in (W^{1,p}\cap L^{\infty})(Q^m ; N^n)\), the essential range of \(u\) is contained in a compact subset \(K\) of a  smooth compact submanifold  \(N^n_0\subset N^n\) with boundary. 
Let \(L^n\) be a compact smooth submanifold of \(\R^{\nu+1}\) satisfying the properties of Lemma~\ref{lemmaNonCompactToCompact}.  
Then \(u\) belongs to \(W^{1,p}(Q^m ; L^n)\). By \cite{Bousquet-Ponce-VanSchaftingen}*{Theorem 2}, there exists a sequence of maps \((u_j)_{j\in \N}\) in \(R_{m-\floor{p}-1}(Q^m ; L^n)\) which converges to \(u\) in \(W^{1,p}(Q^m ; \R^{\nu + 1})\). 
This implies that the sequence \((P \circ u_j)_{j\in \N}\) in  \(R_{m-\floor{p}-1}(Q^m ; N^n)\) still converges to \(P \circ u=u\) in \(W^{1,p}(Q^m ; \R^{\nu})\). 
Since \(P(L^n)\subset N^n_0\), the sequence \((P \circ u_j)_{j\in \N}\) is also contained in the space \(L^{\infty}(Q^m ; N^n)\). 
This completes the proof of the first part of the proposition.

If we further assume that \(\pi_{\floor{p}}(N^n)\) is trivial, then we can approximate each map \(P(u_j)\in R_{m-\floor{p}-1}(Q^m ; N^n)\)  by a sequence of smooth maps in \(C^{\infty}(\overline{Q^m} ; N^n)\). 
As in the case of compact-target manifolds, one relies here on the existence of a smooth projection from a tubular neighborhood of a compact subset of \(N^{n}\) into \(N^{n}\), see Section~7 and the Claim in Section~9 of~\cite{Bousquet-Ponce-VanSchaftingen}. 
By a diagonal argument, this implies that \(u\) itself belongs to the closure of \(C^{\infty}(\overline{Q^m} ; N^n)\).
\end{proof}

\section{Lack of strong density in $W^{1, n}(Q^m ; N^n)$}%
\label{sectionCounterexample}

In this section we give an example of a complete manifold for which \((W^{1, n} \cap L^\infty)(Q^m; N^n)\) is not strongly dense in \(W^{1, n}(Q^m; N^n)\) with \(n \le m\). 
We state the main result of this section as follows:

\begin{proposition}
\label{propCounterExample}
Let \(\nu \in \N_{*}\), \(m, n \in \N_{*}\) be such that \(m \ge n \ge 2\), and let \(a \in \Sphere^{n}\).
For every smooth embedding \(F : \Sphere^{n}\setminus \{a\} \to \R^{\nu}\) such that \(\lim\limits_{x \to a}{\abs{F(x)}} = +\infty\) and \(DF \in L^{n}(\Sphere^{n} \setminus \{a\})\), the closed submanifold \(N^{n} = F(\Sphere^{n} \setminus \{a\})\) equipped with the Riemannian metric inherited from \(\R^{\nu}\) is complete, but \((W^{1, n} \cap L^{\infty})(Q^m ; N^n)\) is \emph{not} strongly dense in \(W^{1, n}(Q^{m}; N^{n})\).
\end{proposition}

For instance, one may take \(F(x) = \lambda(x) x\), where \(\lambda : \Sphere^{n} \setminus \{a\} \to \R\) is a positive smooth function such that \(\lambda \in W^{1, n}(\Sphere^{n}; \R)\) and \(\lim\limits_{x\to a}\lambda(x)=+\infty\).
This is always possible in dimension \(n \ge 2\), and an example is given by setting
\[
\lambda(x) =  \Big( \log{\frac{1}{\dist_{\Sphere^{n}}(x, a)}} \Big)^{\alpha}
\]
for \(x\) in a neighborhood of \(a\) and any exponent \(0 < \alpha < \frac{n-1}{n}\).

Proposition~\ref{propCounterExample} follows from the fact that \(F\) cannot be approximated in \(W^{1,n}(\Sphere^n ; \R^{n+1})\) by a sequence in \(C^{\infty}(\Sphere^n ; N^n)\). 
In turn, the latter is proved by contradiction: the ranges of  such approximating maps when restricted to a small neighborhood \(V\) of \(a\) in \(\Sphere^n\) would contain a fixed compact subset \(K\) of \(N^n\). 
By the area formula, the possibility of taking \(V\) arbitrarily small contradicts the equi-integrability of the sequence in \(W^{1,n }\).

\begin{proof}[Proof of Proposition~\ref{propCounterExample}]
Given a diffeomorphism \(f : \overline{Q^{n}} \to \Sphere^{n}\) between \(\overline{Q^{n}}\) and a closed neighborhood of \(a\) in \(\Sphere^{n}\), the function 
\begin{equation}
\label{eq-FunctionCounterExample}
u = F \circ f
\end{equation}
belongs to \(W^{1, n} (Q^n; N^{n})\).{}
We first handle the case \(m = n\) by proving that \(u\) cannot be approximated by bounded maps in \(W^{1, n}(Q^{n}; N^{n})\).
We proceed by contradiction:
if \(u\) is in the closure of the set \((W^{1, n} \cap L^{\infty})(Q^n ; N^n)\) then, by density of the set \(C^{\infty}(\overline{Q^{n}}; N^{n})\) in the former space (Proposition~\ref{proposition_smooth_bounded_density}), there exists a sequence of maps \((u_k)_{k \in \N}\) in \(C^\infty (\overline{Q^n}; N^{n})\) converging strongly to \(u\) in \(W^{1, n} (Q^n; \R^{\nu})\).

Without loss of generality, we may assume that \(f(0) = a\).
Given a compact subset \(K \subset N^{n}\) with \(\mathcal{H}^n(K)>0\), since the embedding \(F\) diverges at the point \(a\), there exists  \(0 < \delta < 1\) such that 
\[{}
K \cap u(\overline{Q_{\delta}^{n}}) = \emptyset.{}
\]
By a Fubini-type argument, there exists a subsequence \((u_{k_{j}})_{j \in \N}\)
such that, for almost every \(r \in (0, 1)\), \((u_{k_{j}}|_{\partial Q_{r}^{n}})_{j \in \N}\) converges to \(u|_{\partial Q_{r}^{n}}\) in \(W^{1, n}({\partial Q_{r}^{n}}; \R^{\nu})\), whence also uniformly by the Morrey--Sobolev inequality.
Since for each such \(r \le \delta\) we have \(K \cap u(\partial{Q_{r}^{n}}) = \emptyset\), by uniform convergence of 
\((u_{k_{j}}|_{\partial Q_{r}^{n}})_{j \in \N}\) there exists \(J_{r} \in \N\) such that, for every \(j \ge J_{r}\), 
\[{}
\|u_{k_{j}}-u\|_{L^{\infty}(\partial Q_{r}^{n})} <\dist (K, u(\partial Q_{r}^{n})).
\] 
In particular, \(K \cap u_{k_{j}}(\partial Q_{r}^{n}) = \emptyset\).

We claim that 
\begin{equation}\label{eq491}
K \subset u_{k_j}(Q^{n}_r).
\end{equation}
To prove this, we take a homeomorphism \(g : \overline{Q^{n}_r} \to \Sphere^{n}\setminus f(Q^{n}_r)\) such that \(g|_{\partial Q^{n}_r} = f|_{\partial Q^{n}_r}\). 
Since \(K\cap \big(F\circ f(\overline{Q^n_{\delta}})\big)=\emptyset\), this implies that 
\[
F\circ g(Q^{n}_r) = F(\Sphere^n\setminus f(\overline{Q^{n}_r})) \supset K.
\]
By continuity of the Brouwer degree with respect to the uniform convergence, for every \(y\in K\) and for every \(j\geq J_r\) we have
\[
\deg{(u_{k_j}, Q^{n}_r, y)} = \deg{(F\circ g , Q^{n}_r, y) \ne 0}.
\] 
This implies Claim~\eqref{eq491}.

By monotonicity of the Hausdorff measure and by the area formula, we then have
\[
\mathcal{H}^n (K) 
\le 
\mathcal{H}^n \bigl(u_{k_{j}}(Q^n_r)\bigr)
\le
\int_{Q^{n}_{r}} {\jac u_{k_{j}}} ,
\]
where \(\jac u_{k_j}=\bigl(\det{((Du_{k_j})^*\circ Du_{k_j})}\bigr)^{1/2}\).
Using the pointwise inequality 
\({}
\resetconstant
{\jac  u_{k_{j}}} \le \Cl{cte-1452} \abs{D u_{k_{j}}}^n\), as \(j\) tends to infinity we get
\[
\mathcal{H}^n (K) 
\le 
\Cr{cte-1452} \int_{Q^{n}_{r}} \abs{D u}^n.
\]
Since the right-hand side tends to zero as \(r\) tends to zero, we have a contradiction.
Hence the density of bounded maps fails in the space \(W^{1, n}(Q^{n}; N^{n})\), and there exists \(\varepsilon > 0\) such that, for every \(w \in (W^{1, n} \cap L^{\infty})(Q^{n}; N^{n})\),{}
\begin{equation}
	\label{eqDensityLowerBound}
	\int_{Q^{n}} \abs{Dw - Du}^{n} \ge \varepsilon.
\end{equation}

When \(m > n\), we consider the map \(u \circ P\), where \(u\) is defined by \eqref{eq-FunctionCounterExample} and \(P : Q^m = Q^n \times Q^{m - n} \to Q^n\) is the projection on the first component. 
Given \(v \in (W^{1, n} \cap L^\infty) (Q^m; N^n)\), it follows for almost every \(y'' \in Q^{m - n}\) that \(v(\cdot, y'') \in (W^{1, n} \cap L^\infty) (Q^n; N^n)\).{}
By Fubini's theorem and the lower bound \eqref{eqDensityLowerBound} applied to \(w = v(\cdot, y'')\), we get
\[
\int_{Q^{m}} \abs{D v - D (u\circ P)}^n
=  \int_{Q^{m - n}} \biggl(\; \int_{Q^n} \abs{D v (y', y'') - D u (y')}^n \dif y' \biggr)\dif y''
\ge 2^{m-n} \varepsilon.  
\]
Hence, \(u \circ P\) does not belong to the closure of \((W^{1,n}\cap L^{\infty})(Q^m ; N^n)\).
\end{proof}

An alternative example, this time of an \textit{algebraic} complete manifold \(N^{n}\) for which \((W^{1, n} \cap L^{\infty})(Q^{n} ; N^n)\) is not strongly dense in \(W^{1, n}(Q^{n}; N^{n})\), is
\[
N^{n} = \Bigl\{ y = (y_{1}, y') \in \R^+ \times \R^{n} \st \abs{y'}^2 = \frac{y_{1}}{(1 + y_{1})^{2\beta - 1}}\Bigr\},
\]
where \(\beta > \frac{n}{n - 1}\), see Figure~\ref{figureAlgebraic}.
\begin{figure}
\begin{center}
\begin{tikzpicture}
\begin{scope}[shift={(0,0)},rotate=-90]
\draw [domain=0:2, scale=5, samples=200]
plot ({sqrt(\x/((1+\x)^6))} ,\x);
\draw [domain=0:2, scale=5, samples=200]
plot ({-sqrt(\x/((1+\x)^6))} ,\x);
\draw[dashed][scale=5] 
(-0.125,1) arc (-135 : -45 : {(0.25/sqrt(2))}) ;
\draw[dashed][scale=5] 
(0.125,1) arc (45 : 135 : {(0.25/sqrt(2))}) ;
\end{scope}
\end{tikzpicture}
\end{center}
\caption{The algebraic manifold $N^n$}
\label{figureAlgebraic}
\end{figure}
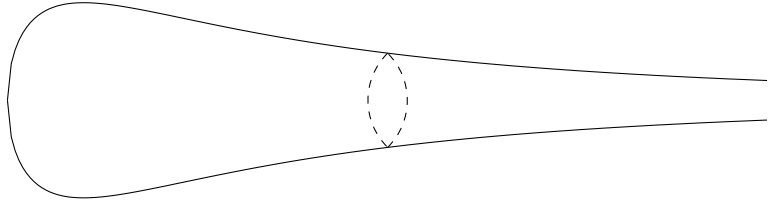
The lack of density can be obtained using the map \(u : Q^{n} \to N^{n}\) defined for \(x \in Q^{n} \setminus \{0\}\) by
\begin{equation}
\label{eqDefBadFunctionStrong}
  u (x) = \biggl(\varphi (\abs{x}), \frac{\sqrt{\varphi (\abs{x})} }{(1 + \varphi (\abs{x}))^{\beta - \frac{1}{2}} } \frac{x}{\abs{x}}\biggr),
\end{equation}
where \(\varphi: (0, \infty) \to \R_{+}\) is a smooth function such that \(\varphi (r) = \abs{\log r}^\gamma\) for  \(r \in (0, 1/3)\), \(\varphi (r) = 0\) for  \(r \in (2/3, \infty)\)  and \(\varphi'/\sqrt{\varphi}\) is bounded on \((1/3, 2/3)\). 
Then, \(u\) belongs to \(W^{1, n}(Q^{n}; N^{n})\) provided that 
\(
  \frac{1}{n(\beta-1)} < \gamma < \frac{n - 1}{n}.
\)
An adaptation of the proof of Proposition~\ref{propCounterExample}  shows that there exists no sequence of maps in \(C^{\infty}(\overline{Q^{n}}; N^{n})\) converging strongly to \(u\) in \(W^{1, n}(Q^{n}; N^{n})\).

Haj\l asz and Schikorra~\cite{Hajlasz-Schikorra}*{Section~3} have provided examples of noncompact manifolds \(N^{n}\) for which Lipschitz maps are not strongly dense in \(W^{1, n} (Q^n; N^{n})\). 
Instead of taking an embedding \(F\) that blows up at some point \(a\) as we do, they construct an embedding that is not proper but strongly oscillates in a neighborhood of the point \(a\).


\section{Main tools}
\label{sectionMainTools}

In this section we explain the main tools used in the proofs of Theorems~\ref{theoremMainNonInteger} and~\ref{theorem_Ap_CNS}.

\subsection{The opening technique}%
\label{sectionOpening}

We recall the technique of opening of maps that has been introduced by Brezis and Li~\cite{Brezis-Li} and pursued in \cite{Bousquet-Ponce-VanSchaftingen}*{Section~2}.  
To illustrate the main idea, we explain this tool in a model situation.
Given a map \(u \in W^{1, p}(\R^{m}; \R^{\nu})\), we wish to construct a smooth map \(\Phi : \R^{m} \to \R^{m}\) such that
\begin{enumerate}[\((a)\)]
	\item \(u \circ \Phi\) is constant in \(Q_{1}^{m}\);{}
	\item \(u \circ \Phi = u\) in \(\R^{m} \setminus Q^{m}_{4}\);{}
	\item \(u \circ \Phi \in W^{1, p}(\R^{m}; \R^{\nu})\) and \(\norm{u \circ \Phi}_{W^{1, p}(\R^{m})} \le C \norm{u}_{W^{1, p}(\R^{m})}\).
\end{enumerate}

The opening construction is based on the following elementary inequality~\cite{Bousquet-Ponce-VanSchaftingen}*{Lemma 2.5}:
\[{}
\fint_{Q_{1}^{m}} \biggl( \int_{Q_{5}^{m}} f \circ \Phi_{z} \biggr) \dif z
\le 6^{m} \int_{\R^{m}} f,
\]
whose proof is based on Fubini's theorem and is valid for every nonnegative function \(f \in L^{1}(\R^{m})\), where \(\Phi_{z} : \R^{m} \to \R^{m}\) is defined by
\[{}
\Phi_{z}(x) = \zeta(x + z) - z
\]
and \(\zeta : \R^{m} \to \R^{m}\) is any smooth function.
Assuming that \(\zeta = 0\) in \(Q_{2}^{m}\) and \(\zeta = \Id\) in \(\R^{m} \setminus Q_{3}^{m}\), we then have for every \(z \in Q_{1}^{m}\) that \(\Phi_{z}\) is constant in \(Q_{1}^{m}\) and \(\Phi_{z} = \Id\) in \(\R^{m} \setminus Q_{4}^{m}\).{}
Formally taking \(f = \abs{u}^{p} + \abs{Du}^{p}\) in the inequality above, one gets
\[{}
\fint_{Q_{1}^{m}}{\norm{u \circ \Phi_{z}}_{W^{1, p}(Q_{5}^{m})}^{p}} \dif z
\le 6^{m} \norm{u}_{W^{1, p}(\R^{m})}^{p},
\]
and then it suffices to take \(z \in Q_{1}^{m}\) such that 
\[{}
\norm{u \circ \Phi_{z}}_{W^{1, p}(Q_{5}^{m})}^{p}
\le 2 \cdot 6^{m} \norm{u}_{W^{1, p}(\R^{m})}^{p}.
\]
This formal argument can be rigorously justified using an approximation of \(u\) by smooth functions in \(W^{1, p}(\R^{m}; \R^{\nu})\), see \cite{Bousquet-Ponce-VanSchaftingen}*{Lemma 2.4} and also \cite{Hang-Lin-III}*{Section~7}.
Such an averaging procedure is reminiscent of the work of Federer and Fleming~\cite{Federer-Fleming} and was adapted to Sobolev functions by Hardt, Kinderlehrer and Lin~\cite{Hardt-Kinderlehrer-Lin}.

More generally, one can open a map around a small neighborhood of \(0\), or along the normals to a planar set.
For example, the singleton \(\{0\}\) may be replaced by a relative open subset of an \(\ell\)-dimensional plane, with \(\ell \in \{1, \dotsc, m-1\}\).{}
In this case, we obtain an opened Sobolev map depending locally on \(\ell\) variables, and constant along \(m-\ell\) normal directions.
The following statement coincides with \cite{Bousquet-Ponce-VanSchaftingen}*{Proposition 2.2} 
and we omit the proof.

\begin{proposition}\label{openingpropSimplex}
Let \(\ell \in \{0, \dotsc, m-1\}\), \(\eta>0\), \(0<\underline{\rho}<\overline{\rho}\) and \(A\subset \R^\ell\) be an open set. 
For every \(u\in W^{1,p}(A\times Q^{m-\ell}_{\overline{\rho}\eta} ; \R^\nu)\), there exists a smooth map \(\zeta : \R^{m-\ell} \to \R^{m-\ell}\) such that
\begin{enumerate}[$(i)$]
\item  \(\zeta\) is constant in \(Q^{m-\ell}_{\underline{\rho}\eta}\),
\item \(\{x\in \R^m : \zeta(x)\not= x\} \subset  Q^{m-\ell}_{\overline{\rho}\eta}\) and \(\zeta(Q^{m-\ell}_{\overline{\rho}\eta}) \subset Q^{m-\ell}_{\overline{\rho}\eta}\),
\item \label{orthosimplex} if \(\Phi : \R^m \to \R^m\) is defined for every \(x=(x',x'')\in \R^\ell \times \R^{m-\ell}\) by
\[
\Phi(x)=\bigl(x', \zeta(x'')\bigr),
\]
then \(u\circ \Phi \in W^{1,p}(A\times Q^{m-\ell}_{\overline{\rho}\eta} ; \R^{\nu})\) and
\[
\norm{D(u \circ \Phi)}_{L^p(A\times Q^{m-\ell}_{\overline{\rho}\eta})}\leq C \norm{Du}_{L^p(A\times Q^{m-\ell}_{\overline{\rho}\eta})},
\] 
for some constant \(C>0\) depending on \(m, p, \underline{\rho}\) and \(\overline{\rho}\).
\end{enumerate}
\end{proposition}

We observe that $(\ref{orthosimplex})$ implies that \(\Phi\) is constant on the \((m-\ell)\)-dimensional cubes of inradius \(\underline{\rho}\eta\) which are orthogonal to \(A\). 
The map \(u\circ \Phi\) thus only depends on \(\ell\) variables in a neighborhood of \(A\).

In order to present the opening technique in the framework of cubications, we first need to introduce some vocabulary. First, given a set \(A\subset \R^m\) and \(\eta>0\), a cubication of \(A\) of inradius \(\eta >0\) is a family of closed cubes \(\mathcal{S}^m\) of inradius \(\eta\) such that
\begin{enumerate}[$(1)$]
\item \(\bigcup\limits_{\sigma^m \in \mathcal{S}^m}\sigma^m =A\),
\item for every \(\sigma^{m}_1,\sigma^{m}_2 \in \mathcal{S}^m\) which are not disjoint, \(\sigma^{m}_1 \cap \sigma^{m}_2\) is a common face  of dimension \(i\in \{0, \dotsc, m\}\). 
\end{enumerate}

For \(\ell \in \{0, \dots , m\}\), the set \(\mathcal{S}^{\ell}\) of all \(\ell\)-dimensional faces of all cubes in \(\mathcal{S}^m\) is called the skeleton of dimension \(\ell\). We then denote by \(S^\ell\) the union of all elements of \(\mathcal{S}^{\ell}\):
\[
S^\ell =\bigcup_{\sigma^\ell\in \mathcal{S}^\ell} \sigma^{\ell}.
\] 
A subskeleton \(\mathcal{E}^\ell\) of \(\mathcal{S}^\ell\) is simply a subfamily of \(\mathcal{S}^\ell\) and the associated subset of \(\R^m\) is 
\[
 E^\ell =\bigcup_{\sigma^\ell\in \mathcal{E}^\ell} \sigma^{\ell}.
\]
Accordingly, given a set \(A\) in \(\R^m\) equipped with a cubication, a subskeleton of \(A\) is a subfamily of the \(\ell\)-dimensional skeleton of the given cubication.

We proceed to state the main result of this section, which is essentially \cite{Bousquet-Ponce-VanSchaftingen}*{Proposition 2.1}.
   
\begin{proposition}\label{openingpropGeneral}
Let \(p\geq 1\), \(\ell \in \{0, \dotsc, m-1\}\), \(\eta > 0\), \(0 < \rho < \frac{1}{2}\), and \(\mathcal{E}^\ell\) be a 
subskeleton of \(\R^m\) of inradius \(\eta\). Then, for every \(u\in W^{1, p}(E^\ell + Q^m_{2\rho\eta}; \R^\nu)\), there exists a smooth 
map \(\Phi : \R^m \to \R^m\) such that
\begin{enumerate}[$(i)$]
\item 
for every \(i \in \{0, \dotsc, \ell\}\) and 
for every \(i\)-dimensional face \(\sigma^i \in \mathcal{E}^i\), \(\Phi\) is constant on the  \((m-i)\)-dimensional cubes of inradius \(\rho\eta\) which are orthogonal to \(\sigma^i\),
\label{itemgenopeningprop1}
\item \(\{x\in \R^m : \Phi(x)\not=x\} \subset E^{\ell}+Q^{m}_{2\rho\eta}\) and, 
for every \(\sigma^\ell \in \mathcal{E}^\ell\), 
\(\Phi(\sigma^\ell + Q^m_{2\rho\eta}) \subset \sigma^\ell + Q^m_{2\rho\eta}\,\),
\label{itemgenopeningprop4}
\item \(u \circ \Phi \in W^{1, p}(E^\ell + Q^m_{2\rho\eta}; \R^{\nu})\),
\label{itemgenopeningprop5}
and
\begin{equation*}
  \norm{D(u \circ \Phi)}_{L^p(E^\ell + Q^m_{2\rho\eta})} \leq C   \norm{D u}_{L^p(E^\ell + Q^m_{2\rho\eta})},
\end{equation*}
for some constant \(C > 0\) depending on \(m\),  \(p\) and \(\rho\),
\item
\label{itemgenopeningprop6}
for every \(\sigma^\ell \in \mathcal{E}^\ell\),
\begin{equation*}
 \norm{D(u \circ \Phi)}_{L^p(\sigma^\ell + Q^m_{2\rho\eta})} \leq C'  \norm{D u}_{L^p(\sigma^\ell + Q^m_{2\rho\eta})},
\end{equation*}
for some constant \(C' > 0\) depending on \(m\), \(p\) and \(\rho\).
\end{enumerate}
\end{proposition}

Here \(Q^m_r (x) = x + r Q^m\) is the cube centered at \(x\) with inradius \(r\).
For the convenience of the reader, and also because Assertion~$(\ref{itemgenopeningprop4})$ in Proposition~\ref{openingpropGeneral} is slightly more precise than the corresponding statement in \cite{Bousquet-Ponce-VanSchaftingen}, we sketch its proof.

\begin{proof}[Proof of Proposition~\ref{openingpropGeneral}] 
Let \(\rho=\rho_\ell<\dots < \rho_i<\dots<\rho_0<2\rho\). 
We define by induction on \(i\in \{0, \dotsc, \ell\}\) a map \(\Phi^i : \R^m \to \R^m\) such that 
\begin{enumerate}[$(a)$]
\item for every \(r \in \{0, \dotsc, i\}\) and every \(\sigma^r \in \mathcal{E}^r\), \(\Phi^{i}\) is constant on the \((m-r)\)-dimensional cubes of inradius \(\rho_i\eta\) which are orthogonal to \(\sigma^r\),
\label{openingrecursive1}
\item \(\{x\in \R^m \st \Phi^i(x)\not=x\}\subset E^i+Q^{m}_{2\rho\eta}\) and, for every \(\sigma^i \in \cE^i\), \(\Phi^i(\sigma^i+Q^{m}_{2\rho\eta})\subset \sigma^i+Q^{m}_{2\rho\eta}\,\), 
\label{openingrecursive2}
\item \(u\circ \Phi^i \in W^{1,p}(E^\ell+Q^{m}_{2\rho\eta}; \R^{\nu})\),
\label{openingrecursive3}
\item for every \(\sigma^i \in \cE^i\),
\label{openingrecursive4}
\begin{equation*}
 \norm{D(u \circ \Phi^i)}_{L^p(\sigma^i+Q^{m}_{2\rho\eta})} \leq C \norm{Du}_{L^p(\sigma^i+Q^{m}_{2\rho\eta})},
\end{equation*}
for some constant \(C > 0\) depending on \(m\), \(p\) and \(\rho\).
\end{enumerate}
The map \(\Phi^\ell\) will satisfy the conclusion of the proposition.

For \(i=0\), \(\cE^0\) is the set of vertices of cubes in \(\cE^m\). The map \(\Phi^0\) is obtained by applying Proposition~\ref{openingpropSimplex} to \(u\) around each \(\sigma^0 \in \cE^0\) with parameters \(\rho_0<2\rho\) and \(\ell=0\).

Assume that the maps \(\Phi^0, \dotsc, \Phi^{i-1}\) have been constructed. We then apply Proposition~\ref{openingpropSimplex} to \(u\circ \Phi^{i-1}\) around each \(\sigma^i \in \cE^i\) with parameters \(\rho_i<\rho_{i-1}\). This gives a smooth map \(\Phi_{\sigma^i} : \R^m \to \R^m\) which is constant on the \((m-i)\)-dimensional cubes of inradius \(\rho_i\eta\) orthogonal to \(\sigma^i\).

Let \(\sigma^i_1, \sigma^i_2 \in \cE^i\) such that 
\[
(\sigma^{i}_1+Q^{m}_{\rho_{i-1}\eta})\cap (\sigma^{i}_2+ Q^{m}_{\rho_{i-1}\eta})\not=\emptyset.
\] 
We claim that for every \(x\) in this set, 
\begin{equation}\label{eq728}
\Phi^{i-1}(\Phi_{\sigma^{i}_1}(x))= \Phi^{i-1}(\Phi_{\sigma^{i}_2}(x))=\Phi^{i-1}(x).
\end{equation}
Indeed, since \(\sigma^i_1\) and \(\sigma^i_2\) are not disjoint, we can take the smallest integer \(r\in \{0, \dotsc,i-1\}\) such that \(x\in \tau^r + Q^{m}_{\rho_{i-1}\eta}\) for some face \(\tau^r\in \cE^r\) with \(\tau^r\subset \sigma^i_1\cap \sigma^i_2\).

By the formula of \(\Phi_{\sigma^{i}_j}\) given in Proposition~\ref{openingpropSimplex}, the points \(\Phi_{\sigma^{i}_1}(x)\), \(\Phi_{\sigma^{i}_2}(x)\) and \(x\) belong to the same \((m-r)\)-dimensional cube of inradius \(\rho_{i-1}\eta\) which is orthogonal to \(\tau^r\).
By induction, \(\Phi^{i-1}\) is constant on the \((m-r)\)-dimensional cubes of inradius \(\rho_{i-1}\eta\) which are orthogonal to \(\tau^r\). 
This proves claim~\eqref{eq728}.

We can thus define the map \(\Phi^i : \R^m \to \R^m\) as follows:
\[
\Phi^i(x)=
\begin{cases}
\Phi^{i-1}(\Phi_{\sigma^i}(x)) &\text{if } x\in \sigma^i + Q^{m}_{\rho_{i-1}\eta},\\
\Phi^{i-1}(x)&\text{otherwise}.
\end{cases}
\]
Assertion~$(\ref{openingrecursive1})$ follows from the above discussion which implies in particular that \(\Phi^i=\Phi^{i-1}\) on \(E^{i-1}+Q^{m}_{\rho_{i-1}\eta}\).

We proceed with the proof of assertion $(\ref{openingrecursive2})$. 
By definition of the map \(\Phi^i\), we have \(\Phi^i=\Phi^{i-1}\) on \(\R^m \setminus (E^i+Q^{m}_{\rho_{i-1}\eta})\). 
By induction, \(\Phi^{i-1}\) agrees with the identity outside \(E^{i-1}+Q^{m}_{2\rho\eta}\). 
Hence,
\[
\{x\in \R^m \st \Phi^{i}(x)\not=x\}\subset E^{i}+Q^{m}_{2\rho\eta}.
\]
Moreover, by induction, for every \(\tau^{i-1}\in \cE^{i-1}\), 
\[\Phi^{i-1}(\tau^{i-1}+Q^{m}_{2\rho\eta})\subset \tau^{i-1}+Q^{m}_{2\rho\eta}.\]
Thus for every \(\sigma^i\in\cE^i\),
\(\Phi^{i-1}(\partial \sigma^i + Q^{m}_{2\rho\eta})\subset \partial \sigma^i + Q^{m}_{2\rho\eta}\). Since \(\Phi^{i-1}(x)=x\)  for \(x\not\in E^{i-1}+Q^{m}_{2\rho\eta}\) and \((E^{i-1}+Q^{m}_{2\rho\eta})\cap (\sigma^i+Q^{m}_{2\rho\eta})\subset \partial \sigma^i + Q^{m}_{2\rho\eta}\,\), it follows that
\begin{equation}\label{eq748}
\Phi^{i-1}(\sigma^i+Q^{m}_{2\rho\eta}) \subset \sigma^i + Q^{m}_{2\rho\eta}.
\end{equation}
Now, let \(x\in \sigma^i+Q^{m}_{2\rho\eta}\). If \(x\not\in E^i+Q^{m}_{\rho_{i-1}\eta}\), then \(\Phi^i(x)=\Phi^{i-1}(x)\). From \eqref{eq748}, it follows that \(\Phi^i(x)\in \sigma^i+Q^{m}_{2\rho\eta}\). 
We now assume that \(x\in E^i+Q^{m}_{\rho_{i-1}\eta}\). 
Let \(\tau^i \in \cE^i\) be such that \(x\in \tau^i + Q^{m}_{\rho_{i-1}\eta}\). 
Then the cube \(\tau^r:=\tau^i\cap \sigma^i\) is not empty, \(x\in \tau^r+Q^{m}_{2\rho\eta}\) and from the form of \(\Phi_{\tau^i}\), we deduce that \(\Phi_{\tau^i}(x)\in \tau^r + Q^{m}_{2\rho\eta}\). 
In particular, \(\Phi_{\tau^i}(x)\in \sigma^i+Q^{m}_{2\rho\eta}\) and thus by \eqref{eq748}, \(\Phi^i(x)=\Phi^{i-1}(\Phi_{\tau^i}(x))\in \sigma^i + Q^{m}_{2\rho\eta}\). This completes the proof of $(\ref{openingrecursive2})$.

The proofs of Assertions $(\ref{openingrecursive3})$ and $(\ref{openingrecursive4})$ are the same as in \cite{Bousquet-Ponce-VanSchaftingen} and we omit them.
\end{proof}

When \(\ell \leq p+1\), the function \(u \circ \Phi\) given by Proposition~\ref{openingpropGeneral} satisfies
\begin{equation}\label{eq2087}
\frac{1}{r^{m-p}} \int_{Q^{m}_r(x)}\abs{D(u \circ \Phi)}^p \leq \frac{C}{\eta^{m-p}} 
\int_{\tau^{\ell-1}+Q^{m}_{\rho\eta}}\abs{D(u\circ \Phi)}^p, 
\end{equation}
for every cube \(Q^{m}_r(x)\subset \tau^{\ell-1}+Q^{m}_{\rho\eta}\) and every face \(\tau^{\ell-1}\in \cE^{\ell-1}\).
This estimate follows from the fact that \(\Phi\) is constant on the \((m-\ell+1)\)-dimensional cubes of inradius \(\rho\eta\) which are orthogonal to \(\tau^{\ell-1}\).

Indeed, without loss of generality, we can assume that \(\tau^{\ell-1}\) has the form \( [-\eta ,\eta]^{\ell-1}\times \{0''\}\), where \(0'' \in \R^{m - \ell + 1}\). 
Accordingly, we write every \(y\in \tau^{\ell-1}+Q^{m}_{\rho\eta}\) as \(y=(y',y'')\in \R^{\ell-1}\times \R^{m-\ell+1}\). 
By construction, for every \(y\in \tau^{\ell-1}+Q^{m}_{\rho\eta}\,\), \((u \circ \Phi)(y)= (u \circ \Phi)(y',0'')\). 
This implies
\[{}
\resetconstant
\begin{split}
\int_{Q^{m}_r(x)}\abs{&D(u \circ \Phi)(y)}^p\dif y \\
&\leq \C r^{m-\ell+1} \int_{Q^{\ell-1}_r(x')}\abs{D(u \circ \Phi)(y',0'')}^p\dif y'\\
&\leq \Cl{cte-1453} \frac{r^{m-\ell+1}}{\eta^{m-\ell+1}} \int_{Q^{\ell-1}_r(x')\times Q^{m-\ell+1}_{\rho\eta}(0'')}
\abs{D(u \circ \Phi)(y',y'')}^p\dif y'\dif y''\\
&\leq \Cr{cte-1453} \frac{r^{m-\ell+1}}{\eta^{m-\ell+1}} \int_{\tau^{\ell-1}+Q^{m}_{\rho\eta}}
\abs{D(u \circ \Phi)(y',y'')}^p\dif y'\dif y''.
\end{split}
\]
In the last line, we have used the fact that \(Q^{\ell-1}_r(x')\times Q^{m-\ell+1}_{\rho\eta}(0'')\subset \tau^{\ell-1}+Q^{m}_{\rho\eta}\) which in turn follows from the assumption \(Q^{m}_r(x)\subset \tau^{\ell-1}+Q^{m}_{\rho\eta}\) and the explicit form of \(\tau^{\ell-1}\).
Hence,
\[
\frac{1}{r^{m-p}} \int_{Q^{m}_r(x)}\abs{D(u \circ \Phi)}^p 
\leq \frac{\C}{\eta^{m-p}}\Big(\frac{r}{\eta}\Big)^{p-\ell+1} 
\int_{\tau^{\ell-1}+Q^{m}_{\rho\eta}}\abs{D(u \circ \Phi)}^p,
\]
which proves estimate~\eqref{eq2087} since \(r\leq \eta\) and \(\ell\leq p+1\).


\subsection{Adaptive smoothing}%
\label{sectionAdaptiveSmoothing}

A second tool is the adaptive smoothing, in which the function is smoothened by mollification at a variable scale.
More precisely, given \(u \in L^{1}_{\mathrm{loc}}(\R^{m}; \R^{\nu})\) and a nonnegative function 
\(\varphi \in C_c^\infty(B_1^m)\) such that \(\int_{B_{1}^{m}}{\varphi} = 1\), define the map \(\varphi_{\psi} \ast u \st \R^{m} \to \R^\nu\)
\begin{equation}
	\label{eqConvolution}
	(\varphi_{\psi} \ast u)(x)
	= \int_{B_1^m} u(x-\psi(x)y)\varphi(y) \dif y,
\end{equation}
where \(\psi \in C^\infty(\Omega)\) is a nonnegative function that plays the role of the variable parameter.
Observe that if \(\psi (x) > 0\), then by an affine change of variable we have
\[
 (\varphi_{\psi} \ast u)(x)= \frac{1}{\psi(x)^m} \int_{\Omega} u(z) \varphi\Big(\frac{x-z}{\psi(x)}\Big) \dif z;
\]
otherwise, \(\psi (x) = 0\) and since \(\int_{B_{1}^{m}}{\varphi} = 1\) we have
\[
 \varphi_{\psi} \ast u (x) = u (x).
\]

The adaptive smoothing has an immediate counterpart for functions \(u \in L^{1}_{\mathrm{loc}}(\Omega; \R^{\nu})\) in an open subset \(\Omega \subset \R^m\).{}
In this case, we define \(\varphi_{\psi} \ast u\) by \eqref{eqConvolution} at points \(x\) in the open set
\begin{equation}
	\label{eqConvolutionDomain}
\omega = \big\{x\in \Omega \st \dist{(x, \partial\Omega)} > \psi(x) \big\}.
\end{equation}
For Sobolev maps \(u\) the following property holds~\cite{Bousquet-Ponce-VanSchaftingen}*{Propositions~3.1 and~3.2}:

\begin{proposition}
\label{lemmaConvolutionEstimates}
If  \(u\in W^{1, p}(\Omega; \R^\nu)\) and \(\norm{D\psi}_{L^{\infty}(\Omega)} < 1\), then \(\varphi_{\psi} \ast u \in W^{1, p}(\omega; \R^\nu)\) where \(\omega\) is given by \eqref{eqConvolutionDomain}.{}
Moreover, the following estimates hold:
\begin{align*}
\norm{\varphi_{\psi} \ast u - u}_{L^p(\omega)}
& \leq \sup_{v \in B_1^m}{ \norm{\tau_{\psi v}( u) -  u}_{L^p(\omega)}},\\ 
\norm{D(\varphi_{\psi} \ast u)}_{L^p(\omega)} 
& \leq \frac{C}{(1 - \norm{D\psi}_{L^\infty(\omega)})^\frac{1}{p}}  \norm{D u}_{L^p(\Omega)},
\end{align*}
and
\begin{multline*}
 \norm{D(\varphi_{\psi} \ast u) - D u}_{L^p(\omega)}\\
\leq \sup_{v \in B_1^m}{ \norm{\tau_{\psi v}(D u) - D u}_{L^p(\omega)}} + \frac{C'}{(1 - \norm{D\psi}_{L^\infty(\omega)})^\frac{1}{p}}  
 \norm{D u}_{L^p(A)},
\end{multline*}
for some  constants \(C, C'>0\) depending on \(m\) and \(p\), where
\(
 A = \bigcup\limits_{x \in  \omega \cap\supp{D\psi}} B_{\psi(x)}^m(x).
\) 
\end{proposition}

In this statement, \(\tau_v(u) : \Omega + v \to \R^\nu\) denotes the translation with respect to the vector \(v \in \R^m\) defined for each \(x \in \Omega + v\) by \(\tau_v(u) (x) = u (x - v)\).

\subsection{Zero-degree homogenization}%
\label{sectionHomogenization}

This tool has been used in problems involving compact-target manifolds~\cites{Bethuel-Zheng, Bethuel, Hang-Lin}, and allows one to extend a So\-bo\-lev map \(u\) defined on the boundary of a star-shaped domain to the whole domain, by preserving the range of \(u\).
We first recall the notion of a Sobolev map on skeletons~\cite{Hang-Lin}:

\begin{definition}
\label{definitionSobolevSkeleton}
Given \(p \ge 1\), \(\ell \in \{0, \dotsc, m\}\), and an \(\ell\)-dimensional skeleton \(\mathcal{S}^\ell\), we say that a map \(u\) 
belongs to \(W^{1,p}(S^{\ell};\R^\nu)\) whenever
\begin{enumerate}[$(i)$]
\item each \(\sigma^\ell \in \mathcal{S}^\ell\), the map \(u|_{\sigma^{\ell}}\) belongs to \(W^{1, p}(\sigma^{\ell}; \R^\nu)\),
\item each \(\sigma_1^\ell, \sigma_2^\ell \in \mathcal{S}^\ell\) such that \(\sigma_1^\ell \cap \sigma_2^\ell \in \mathcal{S}^{\ell - 1}\), we have \(u|_{\sigma_1^{\ell}} = u|_{\sigma_2^{\ell}}\)  on \(\sigma_1^\ell \cap \sigma_2^\ell\) in the sense of traces.
\end{enumerate}
We then denote  
\[
\norm{u}_{W^{1,p}(S^\ell ; \R^\nu)}^p= \sum_{\sigma^\ell\in \mathcal{S}^\ell} \norm{u}_{W^{1,p}(\sigma^\ell ; \R^\nu)}^p.
\] 
\end{definition}

Given \(\ell \in \{1, \dotsc, m\}\), \(\eta > 0\) and \(a\in \R^m\), we may consider the boundary of the cube \( Q^{\ell}_{\eta}(a)\) 
as an \((\ell-1)\)-dimensional skeleton, and then \(W^{1,p}(\partial Q^{\ell}_\eta(a);\R^\nu)\) has a well-defined meaning in the sense of Definition~\ref{definitionSobolevSkeleton}.

Given  \(i \in \N_*\) and \(\eta > 0\), 
the homogenization of degree \(0\) of a map \(u:\partial  Q^{i}_{\eta}(a) \to \R^\nu\) 
is the map \(v: Q^{i}_{\eta}(a) \to \R^\nu\) defined for \(x\in  Q^{i}_{\eta}(a)\) by
\begin{equation}\label{eq446}
v(x) = u\Big(a+\eta\frac{x-a}{|x-a|_{\infty}} \Big),
\end{equation}
where \(|y|_{\infty} = \max{\big\{ \abs{y_{1}}, \dotsc, \abs{y_{i}}\big\}}\) denotes the maximum norm in \(\R^{i}\).
The basic property satisfied by this construction is the following:

\begin{proposition}
	\label{proposition_homogeneisation_0}
If \(1 \le p < i\), then, for every
\(u\in W^{1,p}(\partial Q^{i}_{\eta}(a) ; \R^\nu)\), the map \(v : Q^{i}_{\eta}(a) \to \R^\nu\) defined in \eqref{eq446} belongs to \(W^{1,p}(Q^{i}_{\eta}(a) ; \R^\nu)\) and  
\[
\norm{Dv}_{L^{p}(Q^{i}_{\eta}(a))}
\leq C\eta^{\frac{1}{p}} \norm{Du}_{L^{p}(\partial Q^{i}_{\eta}(a))}.
\]
\end{proposition}

Iterating the zero-degree homogenization described above, one extends Sobolev functions defined on lower-dimensional subskeletons of \(\R^m\) to an \(m\)-dimensional subskeleton. 
We apply this strategy  to prove the following proposition that will be used in the proof of Theorems~\ref{theoremMainNonInteger} and~\ref{theorem_Ap_CNS}:

\begin{proposition}
\label{proposition_homogeneisation_1}
Let \(\ell\in \{0, \dotsc, m-1\}\), \(\eta>0\), \(\cE^m\) be a subskeleton of \(\R^m\) of inradius \(\eta\) and \(\cS^{m-1}\) be a subfamily of 
\(\cE^{m-1}\). 
If \(p < \ell + 1\), then for every continuous function \(u : E^{\ell}\cup S^{m-1} \to \R^{\nu}\) such that
\begin{enumerate}[$(i)$]
\item \(u|_{E^\ell} \in W^{1,p}(E^\ell ; \R^\nu)\),
\item \(u|_{S^i} \in W^{1,p}(S^i ; \R^\nu)\), for every \(i\in \{\ell+1,\dotsc, m-1\}\), 
\end{enumerate}
 there exists \(v\in W^{1,p}(E^m ;\R^\nu)\) such that \(v(E^m) \subset u(E^\ell \cup S^{m-1})\), \(v=u\) on \(S^{m-1}\) in the sense of traces, and
\[
\norm{Dv}_{L^{p}(E^m)} \leq C \bigg( \eta^\frac{m - \ell}{p} \norm{Du}_{L^{p}(E^\ell)}  + \sum_{i=\ell+1}^{m-1} \eta^\frac{m - i}{p}\norm{Du}_{L^{p}(S^i)}\bigg). 
\]
\end{proposition}

\begin{proof}
Let \(v^\ell : E^\ell \to \R^\nu\) be defined by \(v^\ell = u\) in \(E^\ell\).
We define by induction on \(i \in \{\ell + 1, \dotsc, m-1\}\) a map \(v^{i}: E^{i} \to \R^{\nu}\)
as follows:
\begin{enumerate}[\((a)\)]
\item 
	\label{proposition_homogeneisation_1-First}
	for every \(\sigma^i \in \cE^i \setminus \cS^i\), we apply the zero-degree homogenization on the face \(\sigma^i\) (Proposition~\ref{proposition_homogeneisation_0}) to define \(v^i\) on \(\sigma^i\), so that \(v^i = v^{i-1}\) on \(\partial\sigma^i\) in the sense of traces and
\[
\resetconstant
\int_{\sigma^{i}} \abs{Dv^{i}}^p 
\leq 
\C \eta\int_{\partial \sigma^i} \abs{Dv^{i-1}}^p. 
\]
\item 
	\label{proposition_homogeneisation_1-Second} for every \(\sigma^i \in \cS^i\), we take \(v^i = u\) in \(S^i\), whence
\[	
\int_{\sigma^{i}} \abs{Dv^{i}}^p 
= \int_{\sigma^i} \abs{Du}^p. 
\]
\end{enumerate}
With this definition, we have \(v^i \in W^{1, p}(E^i; \R^\nu)\) since for any given \(\sigma_1^i, \sigma_2^i \in \cE^i\) such that \(\sigma_1^i \cap \sigma_2^i \in \cE^{i-1}\) we have \(v^i|_{\sigma_1^i} = v^i|_{\sigma_2^i}\) on \(\sigma_1^i \cap \sigma_2^i\) in the sense of traces.
From \((\ref{proposition_homogeneisation_1-First})\) and \((\ref{proposition_homogeneisation_1-Second})\),
\[
\int_{E^{i}} \abs{Dv^{i}}^p
\le \C \eta \int_{E^{i-1}} \abs{Dv^{i - 1}}^p + \int_{S^{i}} \abs{Du}^p.
\]
Iterating these estimates we get
\[
\int_{E^{m-1}} \abs{Dv^{m-1}}^p
\le \C \bigg(\eta^{m-1-\ell} \int_{E^{\ell}} \abs{Dv^{\ell}}^p + \sum_{i = \ell+1}^{m-1} \eta^{m-1-i} \int_{S^{i}} \abs{Du}^p \bigg).
\]
From the construction of \(v^i\) we also have 
\[
v^{i}(E^{i}) \subset v^{i-1}(E^{i-1}) \cup u(S^{i}).
\]
Iterating these inclusions we deduce that
\[
v^{m-1}(E^{m-1}) \subset v^\ell(E^\ell) \cup \bigcup_{i = \ell+1}^{m-1}{u(S^{i})} \subset u(E^\ell \cup S^{m-1}).
\]

The map \(v^{m-1} : E^{m-1} \to \R^{\nu}\) extends by zero-degree homogenization on each cube \(\sigma^{m} \in \cE^{m}\) to \(v^{m} : E^{m} \to \R^{\nu}\), with \(v^{m}(E^{m}) = v^{m-1}(E^{m-1})\) and
\[{}
\int_{E^{m}} \abs{Dv^{m}}^p
\le \C \, \eta \int_{E^{m-1}} \abs{Dv^{m-1}}^p.
\]
The function \(v^m\) thus satisfies the required properties.
\end{proof}


\section{Trimming property}
\label{sectionExtensionProperty}

The next proposition reformulates the trimming property (Definition~\ref{definitionExtensionProperty}), replacing \emph{smooth} maps by \emph{continuous} Sobolev maps.

\begin{proposition}\label{proposition_continuous_trimming_property}
Let \(p \in \N_{*}\).
The manifold \(N^n\) satisfies the trimming property of dimension \(p\) if and only if
there exists a constant \(C' > 0\) such that, for each map \(u \in W^{1, p}(Q^p; N^n)\) with \(u \in W^{1, p}(\partial Q^p; N^n)\), 
\[
\norm{Dv}_{L^p(Q^p)}
\le C' \norm{Du}_{L^p(Q^p)},
\]
for some \(v \in W^{1, p}(Q^p; N^n) \cap C^0(\overline{Q^p}; N^n)\) such that \(u = v\) on \(\partial Q^p\) in the sense of traces.
\end{proposition}

\begin{proof}
We begin with the direct implication.
For this purpose, let \(u \in W^{1, p}(Q^p; N^n)\) be such that \(u \in W^{1, p}(\partial Q^p; N^n)\). We regularize \(u\) in two different ways: near the boundary of \(Q^p\), this is done by zero-degree homogenization of \(u|_{\partial Q^p}\); far from the boundary, we use the trimming property. We  paste the two different approximations with a mollification and cut-off argument in such a way that the approximating function takes its values in a neighborhood of \(N^n\). This allows us to project it back on \(N^n\).

More precisely, given \(\frac{1}{2} \le \lambda < 1\), we set \(w : Q^p \to N^n\) to be the function defined for \(x \in Q^p\) by
\[
w(x)=
\begin{cases}
	u({x}/{\lambda}) & \text{if } x \in Q_\lambda^p,\\
	u({x}/{\abs{x}_\infty}) & \text{if } x \in Q^p \setminus Q_\lambda^p.
\end{cases} 
\]
Then \(w \in W^{1, p}(Q^p; N^n)\), \(w\) is continuous in \(Q^p \setminus Q_\lambda^p\) by the Morrey--Sobolev inequality and
\[{}
\resetconstant
\int_{Q^p} \abs{Dw}^p
\le \C \bigg( \int_{Q^p} \abs{Du}^p + (1 - \lambda)\int_{\partial Q^p} \abs{Du}^p \bigg).
\]

Take \(\lambda < r_1 < r_2 < 1\).
If \((\varphi_{\varepsilon})_{\varepsilon > 0}\) is a family of smooth mollifiers with \(\supp \varphi_{\varepsilon}\subset B_{\varepsilon}^p\), then the function \(\varphi_\varepsilon * w\) is smooth on \(Q^p_{r_2} \setminus \overline{Q^p_{r_1}}\), for every \(0 < \varepsilon \le \min{\{r_1 - \lambda, 1 - r_2\}}\).
Given \(\theta \in C_c^\infty(Q^p_{r_2} \setminus \overline{Q^p_{r_1}})\), then for \(\varepsilon > 0\) small the function \(v : Q^p \to N^n\) such that
\[
v(x)
= \begin{cases}
w(x)	& \text{if \(x \in Q^p \setminus (Q^p_{r_2} \setminus \overline{Q^p_{r_1}})\),}\\
\Pi\big( (1 - \theta(x)) w(x) + \theta(x) \varphi_\varepsilon*w(x) \big)
& \text{if \(x \in Q^p_{r_2} \setminus \overline{Q^p_{r_1}}\),}
\end{cases}
\]
is well-defined and belongs to \(W^{1, p}(Q^p; N^n)\). 
Remember that \(\Pi\) is the nearest point projection onto the submanifold \(N^n\), which is smooth on a neighborhood \(O \subset \R^\nu\) of \(N^n\) and satisfies \(D\Pi\in L^{\infty}(O)\), see Section~\ref{section_nearest_point_projection}. 
Here, we also use the continuity of \(w\) on \(Q^p \setminus Q^p_\lambda\), which implies that for \(\varepsilon\) small, the set \(\big((1 - \theta) w + \theta \varphi_\varepsilon*w\big)( Q^p_{r_2} \setminus \overline{Q^p_{r_1}})\) is contained in a compact subset of \(O\).
Moreover, by Lemma~\ref{lemma_chain_rule_C1}, we have the estimate
\[
\int_{Q^p} \abs{Dv}^p
\le \C \bigg(\int_{Q^p} \abs{Dw}^p + \norm{D\theta}_{L^\infty(Q^p)} \int_{Q^p_{r_2} \setminus \overline{Q^p_{r_1}}} \, \abs{w - \varphi_\varepsilon * w}^p\bigg).
\]

Setting \(\theta = 1\) on \(\partial Q_r^p\) for some \(r_1 < r < r_2\), then \(v \in C^\infty(\partial Q_r^p; N^n)\).
Applying the trimming property to the map \(v\) on \(Q_r^p\), there exists a map \(\tilde v \in C^\infty(\overline{Q^p_r}; N^n)\) that coincides with \(v\) on \(\partial Q^p_r\) and is such that
\[
\int_{Q_r^p} \abs{D\tilde v}^p
\le \C \int_{Q_r^p} \abs{Dv}^p.
\]
Extending \(\tilde v\) as \(v\) on \(Q^p \setminus \overline{Q^p_r}\), we deduce from the estimates above that
\begin{multline*}
\int_{Q^p} \abs{D\tilde v}^p
\le \C \bigg(\int_{Q^p} \abs{Du}^p + (1 - \lambda)\int_{\partial Q^p} \abs{Du}^p \\
 + \norm{D\theta}_{L^\infty(Q^p)} \int_{Q^p_{r_2} \setminus \overline{Q^p_{r_1}}} \, \abs{w - \varphi_\varepsilon * w}^p\bigg).
\end{multline*}
To conclude the proof we may assume that \(\int_{Q^p} \abs{Du}^p > 0\). 
Choosing \(\lambda\) close to \(1\) and then \(\varepsilon > 0\) small, the second and third terms in the right-hand side can be controlled by \(\int_{Q^p} \abs{Du}^p\) and the direct implication follows.

To prove the converse implication, we take \(f \in C^\infty(\partial Q^p; N^n)\) having an extension \(u \in W^{1, p}(Q^p; N^n)\). By assumption, there exists a map \(v \in W^{1, p}(Q^p; N^n) \cap C^0(\overline{Q^p}; N^n)\) such that
\(v|_{\partial Q^p}=f\) and 
\[
\norm{Dv}_{L^p(Q^p)}
\le C' \norm{Du}_{L^p(Q^p)}.
\]
Once again, the idea of the proof is to smoothen \(v\) in two different ways. Far from the boundary, this is done by mollification and projection. Near the boundary, we work with a smooth extension of \(f\). 

More specifically, given \(0 < \lambda < 1\), we fix a smooth extension \(\tilde f \in C^\infty(\overline{Q^p} \setminus Q^p_\lambda; N^n)\) of \(f\).
Given \(0 < \lambda < r < \overline{r} < 1\), we take \(\theta \in C_c^\infty(\overline{Q^p} \setminus Q^p_r)\) such that \(\theta = 1\) in \(\overline{Q^p} \setminus Q^p_{\overline r}\).
We note that, for \(r\) close to \(1\) and for \(\varepsilon > 0\) small, the function \(\tilde{v} : Q^p \to N^n\) such that
\[
\tilde v(x)
= \begin{cases}
\tilde f(x) & \text{if \(x \in \overline{Q^p} \setminus Q^p_{\overline{r}} \),}\\
\Pi\bigl( (1 - \theta(x)) \varphi_\varepsilon * v(x) + \theta(x) \tilde f(x) \bigr)
& \text{if \(x \in Q^p_{\overline{r}} \setminus \overline{Q^p_{r}}\),}\\
\Pi(\varphi_\varepsilon * v(x))	& \text{if \(x \in \overline{Q^p_{r}}\),}
\end{cases}
\]
is well-defined and satisfies the estimate
\begin{multline*}
\int_{Q^p} \bigabs{D\tilde v}^p
\le \C \bigg(\int_{Q^p} \abs{Dv}^p +
\int_{Q^p \setminus \overline{Q^p_r}} \abs{D \tilde f}^p 
+ \norm{D\theta}_{L^\infty(Q^p)}^p\int_{Q^{p}_{\overline{r}} \setminus \overline{Q^p_r}} \abs{\varphi_{\varepsilon} * v - v}^{p}\\ + 
\norm{D\theta}_{L^\infty(Q^p)}^p \int_{Q^p \setminus \overline{Q^p_r}} \abs{v - \tilde f}^{p} \bigg).
\end{multline*}
Since \(v - \tilde f = 0\) on \(\partial Q^p\), it follows from the Poincaré inequality that
\[
\int_{Q^p \setminus \overline{Q^p_r}} \bigabs{v -\tilde f}^p
\le \C (1 - r)^p \int_{Q^p \setminus \overline{Q^p_r}} \bigabs{D(v -\tilde f)}^p.
\]
Taking \(r < \overline{r} < 1\) and \(\theta\) such that \((1-r)\norm{D\theta}_{L^\infty(Q^p)} \le C\), we get
\[
\int_{Q^p} \bigabs{D\tilde v}^p
\le \C \bigg(\int_{Q^p} \abs{Dv}^p +
\int_{Q^p \setminus \overline{Q^p_r}} \abs{D \tilde f}^p  + (1-r)^{-p}\int_{Q^{p}_{\overline{r}} \setminus \overline{Q^p_r}} \abs{\varphi_{\varepsilon} * v - v}^{p}\bigg).
\]
We now assume that \(\int_{Q^p} \abs{Du}^p > 0\).
The first integral in the right-hand side is by assumption estimated by \(\int_{Q^p} \abs{Du}^p\). 
Taking \(r\) close to \(1\) and then  \(\varepsilon > 0\) small the second and third integrals are also bounded by \(\int_{Q^p} \abs{Du}^p\), and the conclusion follows. 
\end{proof}

We prove the necessity part in Theorem~\ref{theorem_Ap_CNS}.

\begin{proposition}\label{proposition_thm_Ap_CN}
If \(p \in \{2, \dotsc, m\}\) and if the set \((W^{1,p}\cap L^{\infty})(Q^m ; N^n)\)  is  dense in \(W^{1,p}(Q^m ; N^n)\), then  \(N^n\) satisfies the trimming property of dimension \(p\).
\end{proposition}

\begin{proof}
We first consider the case \(p=m\). 
Let \(u \in W^{1, m}(Q^m; N^n)\) be such that \(u \in W^{1, m}(\partial Q^m; N^n)\). 
By the characterization given by Proposition~\ref{proposition_continuous_trimming_property}, it suffices to prove that there exists a map \(v \in W^{1, m}(Q^m; N^n) \cap C^0(\overline{Q^m}; N^n)\) such that \(u = v\) on \(\partial Q^m\) and such that \(\norm{Dv}_{L^m(Q^m)}\le C \norm{Du}_{L^m(Q^m)}\).
The idea of the proof is to rely on the density of bounded maps to replace \(u\) by a smooth approximation of \(u\) in the interior of the cube.
Such an approximation can be made uniform on the boundary of cubes \(\partial Q^{m}_r\) with \(r\) close to \(1\).{}
We can thus create a transition layer using a zero-degree homogenization of \(u\) and an averaging procedure, obtaining a function that is sufficiently close to \(N^{n}\).

More precisely, given \(0 < \lambda < 1\), we introduce the same map \(w\) as in the proof of Proposition~\ref{proposition_continuous_trimming_property}.
By assumption, the map \(w\) is the limit of a sequence of bounded maps in \(W^{1, m}(Q^m ; N^n)\). 
It follows from Proposition~\ref{proposition_smooth_bounded_density} that there exists a sequence \((u_k)_{k \in \N}\) in \(C^{\infty}(\overline{Q^m} ; N^n)\) converging strongly  to \(w\) in \(W^{1,m}(Q^m ; \R^{\nu})\). 
Then, for almost every \(\lambda < r < 1\), there exists a subsequence \((u_{k_{j}}|_{\partial Q^{m}_r})_{j \in \N}\) converging to \(w|_{\partial Q^m_r}\) in \(W^{1, m}(\partial Q^m_r; \R^{\nu})\). 
By the Morrey--Sobolev embedding, this convergence is also uniform on the set \(\partial Q^m_r\). 
Since \(w(\partial Q_r^m)\) is a compact subset of \(N^n\),  
there exist a compact set \(K\) in the domain \(O\) of the nearest point projection \(\Pi\) and an integer \(J \in \N\) such that for every \(j \geq J\), every \(z \in \partial Q^m_r\) and every \(t\in [0, 1]\), 
\[
tu_{k_{j}}(z)+(1-t)w(z) \in K. 
\]
We also introduce a cut-off function \(\theta\in C^{\infty}_c(Q^m)\) such that \(0\leq\theta\leq 1\) in \(Q^{m}\) and \(\theta=1\) in \(\overline{Q^m_{r}}\), with \((1-r)\|D\theta\|_{L^{\infty}(Q^m)} \leq C\).
For every \(j\geq J\), the map \(v_j : \overline{Q^m} \to N^n\) defined by
\[
v_j(x)
=
\begin{cases}
\Pi\bigl(\theta(x) u_{k_{j}}(r{x}/{\abs{x}_{\infty}}) + (1-\theta(x))w(r{x}/{\abs{x}_{\infty}}) \bigr)  
		& \text{if } x \in Q^{m} \setminus Q^m_r,\\
u_{k_{j}}(x) 	& \text{if } x \in Q^{m}_r,
\end{cases}
\]
is such that \(v_j \in W^{1, m}(Q^m ; N^n)\cap C^{0}(\overline{Q^m} ; N^n)\), \(v_j =u\) on \(\partial Q^m\), and
\begin{multline*}
\resetconstant
\int_{Q^m} \abs{Dv_j}^m
\leq \C \bigg( \int_{Q^m_r} \abs{Du_{k_{j}}}^m 
 + \norm{D\theta}_{L^{\infty}(Q^m)} (1 - r) \int_{\partial Q^{m}_{r}} \abs{u_{k_j} - w }^m \\
 + (1-r)\int_{\partial Q^m_r} (\abs{Du_{k_j}}^m + \abs{Dw}^m) \bigg).
\end{multline*}
Without loss of generality, we can assume that \(\int_{Q^m} \abs{Du}^m > 0\). 
We take \(j\geq J\) large enough so that
the second term in the right-hand side is bounded from above by \(\int_{Q^m} \abs{Du}^m \).
By  convergence of the sequence \((u_{k_j})_{j \in \N}\) we may also assume that 
\[
\int_{Q^m_r} \abs{Du_{k_j}}^m
\le 2 \int_{Q^m_r} \abs{Dw}^m
\quad \text{and} \quad
\int_{\partial Q^m_r} \abs{Du_{k_j}}^m
\le 2 \int_{\partial Q^m_r} \abs{Dw}^m.
\]
In view of the definition of \(w\) in terms of \(u\) we deduce from the estimates above that
\[
\int_{Q^m} \abs{Dv_j}^m  
\leq \C \bigg( \int_{Q^m} \abs{Du}^m + (1-\lambda) \int_{\partial Q^m} \abs{Du}^m \bigg).
\]
To conclude the case \(p = m\), it suffices to choose \(\lambda\) sufficiently close to \(1\) so that the second term in the right-hand side is bounded from above by \(\int_{Q^m} \abs{Du}^m\).

We now consider the case where \(p < m\). 
Under the assumption of the proposition, we claim that \((W^{1,p}\cap L^{\infty})(Q^p ; N^n)\) is also dense in \(W^{1,p}(Q^p ; N^n)\).
We are thus led to the first situation where \(p\) equals the dimension of the domain, and we conclude that the manifold \(N^n\) satisfies the trimming property of dimension \(p\).
It thus suffices to prove the claim.
For this purpose, take \(u\in W^{1,p}(Q^p; N^n)\) and define the function \(v : Q^m \to N^n\) for \(x=(x',x'')\in Q^p \times Q^{m-p} \) by
\[
v(x)=u(x').
\]
By assumption, there exists a sequence of maps \((v_k)_{k \in \N}\) in \((W^{1, p} \cap L^{\infty})(Q^m;N^n)\) converging to \(v\) in \(W^{1, p}(Q^m; N^n)\). 
Hence, there exist a subsequence \((v_{k_{j}})_{j \in \N}\) and \(a \in Q^{m-p}\) such that  \((v_{k_{j}}(\cdot, a))_{j \in \N}\) converges to \(u\) in \(W^{1, p}(Q^p; N^n)\).  
This proves the claim, and concludes the proof of the proposition.
\end{proof}

\begin{definition}\label{definition_bounded_geometry}
A Riemannian manifold  \(N^n\) has \emph{uniform Lipschitz geometry} (or  \(N^n\) is \emph{uniformly Lipschitz}) whenever there exist \(\kappa, \kappa' > 0\) and \(C > 0\) such that, for every \(\xi \in N^n\), 
\[{}
\norm{D\Psi}_{L^{\infty}(B_{N^n} (\xi; \kappa))} + \norm{D\Psi^{-1}}_{L^{\infty}(B_{\R^{n}}(\Psi(\xi); \kappa'))} \le C,
\]
for some local chart \(\Psi : B_{N^n} (\xi; \kappa) \to \R^n\) with \(B_{\R^{n}}(\Psi (\xi); \kappa') \subset \Psi (B_{N^n}(\xi; \kappa))\).
\end{definition}

Here, for \(\xi\in N^n\) and \(\kappa\geq 0\), we have denoted by \(B_{N^n}(\xi; \kappa)\) the geodesic ball in \(N^n\) of center \(\xi\) and radius \(\kappa\). 
A natural candidate for \(\Psi\) is the inverse of the exponential map when the manifold \(N^n\) has a positive global injectivity radius and the exponential and its inverse are uniformly Lipschitz maps on balls of a fixed radius.
If the injectivity radius of \(N^n\) is uniformly bounded from below and the Riemann curvature of \(N^n\) is uniformly bounded, then \(N^n\) has uniform Lipschitz geometry.
By relying on harmonic coordinates instead of the normal coordinates given by the exponential maps, it can be proved that it is sufficient to bound the Ricci curvature instead of the Riemann curvature \cite{Anderson-1990}.

\begin{proposition}\label{propositionboundedgeometry}
If \(N^n\) is uniformly Lipschitz, then  \(N^n\) satisfies the trimming property of any dimension \(p \in \N_{*}\).
\end{proposition}

The proof of Proposition~\ref{propositionboundedgeometry} is based on the following lemma that reduces the problem to a trimming property for maps with small energy on the boundary \(\norm{Du}_{L^{p}(\partial Q^p)}\).
By the Morrey--Sobolev embedding, the range of \(u|_{\partial Q^{p}}\) is then contained on a small geodesic ball, and one can perform the extension in a suitable local chart for manifolds having uniform Lipschitz geometry.

\begin{lemma} \label{lemmalocaltrimming}
Let \(p \in \N_{*}\) and \(\alpha > 0\). 
Assume that for every map \(u \in W^{1, p}(Q^p; N^n)\) satisfying
\(u|_{\partial Q^p}\in W^{1,p}(\partial Q^p ; N^n)\) and 
\[{}
\norm{Du}_{L^{p}(Q^p)} + \norm{Du}_{L^{p}(\partial Q^p)} \leq \alpha,{}
\] 
there exists \(v\in (W^{1, p}\cap L^{\infty})(Q^p;N^n)\) such that
\(v= u\) on \(\partial Q^p\) and 
\[
\norm{Dv}_{L^{p}(Q^p)} \leq C'' \big(\norm{Du}_{L^{p}(Q^p)} + \norm{Du}_{L^{p}(\partial Q^p)} \big) 
\]
for some constant \(C''>0\) independent \(u\). 
Then,
\(N^n\) satisfies the trimming property of dimension \(p\).
\end{lemma}

\begin{proof}[Proof of Lemma~\ref{lemmalocaltrimming}]
Let \(u\in W^{1,p}( Q^p; N^n)\) be a map such that \(u|_{\partial Q^p} \in W^{1,p}(\partial Q^p; N^n)\).  
The idea of the proof is to subdivide the domain \(Q^{p}\) into smaller cubes, and to apply the opening technique on each cube.
The resulting map also has small (rescaled) Sobolev energy on the boundaries of the small cubes, and so we can locally apply the small-energy trimming property to obtain a bounded extension of \(u|_{\partial Q^{p}}\).{}
We then use an approximation argument by convolution to get a continuous extension of \(u|_{\partial Q^{p}}\).
 
More precisely, for \(\frac{1}{2} < \lambda < 1\), we introduce the map \(w\) as in the proof of Proposition~\ref{proposition_continuous_trimming_property}. 
Then, \(w|_{\partial Q^p}= u|_{\partial Q^p}\), \(w\) is bounded on \(Q^p\setminus Q^{p}_{\lambda}\) 
and 
\[{}
\resetconstant
 \norm{Dw}_{L^{p}(Q^p)} 
 \leq \Cl{cte-0604} \bigl( \norm{Du}_{L^{p}(Q^p)}+(1-\lambda)^{\frac{1}{p}} \norm{Du}_{L^{p}(\partial Q^p)} \bigr).
\]
Without loss of generality, we can assume that \(\norm{Du}_{L^{p}(Q^p)}>0\). We take \(\lambda>0\) such that
\[
(1-\lambda)^{\frac{1}{p}} \norm{Du}_{L^{p}(\partial Q^p)}\leq \norm{Du}_{L^{p}(Q^p)}.
\]
This implies
\begin{equation}\label{eq470}
\norm{Dw}_{L^{p}(Q^p)} \leq 
2\Cr{cte-0604} \norm{Du}_{L^{p}(Q^p)}.
\end{equation}

We fix \(0 < \rho < \frac{1}{2}\). 
For every  \(0 < \mu < 1\) sufficiently small, we consider a cubication \(\cK^{p}_\mu\) of inradius \(\mu\) such that
\[
Q^{p}_{\lambda+2\rho\mu}\subset K^{p}_\mu \subset K^{p}_\mu + Q^{p}_{2\rho\mu}\subset Q^p. 
\]
We open the map \(w\) around \(\cK^{p-1}_\mu\). 
More precisely, denoting by \(\Phi^\mathrm{op} : \R^m \to \R^m\) the smooth map
 given by Proposition~\ref{openingpropGeneral} above, we consider
\[
w^\mathrm{op} = w \circ \Phi^\mathrm{op}.
\]
In particular, \(w^\mathrm{op} \in W^{1, p}(Q^p; N^n)\), \(w^\mathrm{op} = w\) outside \(K^{p-1}_\mu+Q^{p}_{2\rho\mu}\) 
and, for every \(\sigma^p \in \cK^{p}_\mu\), we have
\[
\norm{D w^\mathrm{op}}_{L^{p}(\partial \sigma^p+Q^{p}_{2\rho\mu})} 
\leq \C  \norm{D w}_{L^{p}(\sigma^p+Q^{p}_{2\rho\mu})}.
\]
This implies that \(w^\mathrm{op}|_{\partial Q^p} = u|_{\partial Q^p}\) and, for every \(\sigma^p \in \cK_{\mu}^p\),
\begin{equation}
\label{eqTrimmingLocalcell}
 \norm{D w^\mathrm{op}}_{L^{p}(\sigma^p + Q^{p}_{2\rho\mu})}
 \leq \Cl{cte-1598} \norm{Dw}_{L^{p}(\sigma^p+Q^{p}_{2\rho\mu})} .
\end{equation}
Raising both sides to the power \(p\) and summing over all \(\sigma^p \in \cK^p\), we also get
\begin{equation*}
\label{eqTrimmingLocal}
 \norm{D w^\mathrm{op}}_{L^{p}(Q^p)}
 \leq \C \norm{Dw}_{L^{p}(Q^p)} .
\end{equation*}

We  also  need the fact that the opening construction preserves the ranges of the maps. More precisely, for every \(\sigma^{p-1}\in \cK_{\mu}^{p-1}\), we have
\[
w^\mathrm{op}(\sigma^{p-1} + Q^{p}_{2\rho\mu}) \subset w(\sigma^{p-1} + Q^{p}_{2\rho\mu}).
\]
We apply this remark to every \(\sigma^{p-1}\subset \partial K^p_\mu\) to get
\[
 w^\mathrm{op}(\partial K^p_\mu + Q^{p}_{2\rho\mu}) \subset w(\partial K^{p}_\mu + Q^{p}_{2\rho\mu}).
\]
Together with the fact that  \(w\) is bounded on \(Q^p\setminus Q^{p}_{\lambda}\supset \partial K_{\mu}^p + Q^{p}_{2\rho\mu}\), 
this proves that \(w^\mathrm{op} \) is bounded on \(Q^p\setminus  K^{p}_\mu\).

Since \( w^\mathrm{op}\) is \((p-1)\)-dimensional on \(\partial \sigma^p+Q^{p}_{\rho\mu}\) for every \(\sigma^p \in \cK^{p}_\mu\), we have
\begin{equation}
\label{eqlemmatrimming1}
\norm{D w^\mathrm{op}}_{L^{p}(\partial \sigma^p)} 
\leq \frac{\C}{\mu^{\frac{1}{p}}} \norm{D w^\mathrm{op}}_{L^{p}(\partial \sigma^p+Q^{p}_{2\rho\mu})} \leq  \frac{\Cl{cte-1455}}{\mu^{\frac{1}{p}}} \norm{D w}_{L^{p}( \sigma^p+Q^{p}_{2\rho\mu})}.
\end{equation}
We take \(\mu>0\) such that, for every \(\sigma^p\in \cK^{p}_\mu\), we have
\[
(\Cr{cte-1598} + \Cr{cte-1455}) \norm{D w}_{L^{p}( \sigma^p+Q^{p}_{2\rho\mu})} \leq \alpha;
\]
this is possible by equi-integrability of the summable function \(\abs{Dw}^{p}\).
Then, by estimates~\eqref{eqTrimmingLocalcell} and \eqref{eqlemmatrimming1} we have
\[
\norm{D w^\mathrm{op}}_{L^{p}(\sigma^p)} + \mu^{\frac{1}{p}}\norm{D w^\mathrm{op}}_{L^{p}(\partial \sigma^p)} \leq \alpha.
\]
By the small-energy trimming assumption applied to \(w^{\mathrm{op}}|_{\sigma^p}\) for every \(\sigma^p \in \cK^p_{\mu}\) and by a scalling argument, there exists a map \(w_{\sigma^p}\in (W^{1,p}\cap L^{\infty})(\sigma^p ; N^n)\) which agrees with \(w^\mathrm{op}\) on \(\partial \sigma^p\) and is such that 
\begin{equation}
\label{eqlemmatrimming2}
\norm{Dw_{\sigma^p}}_{L^{p}(\sigma^p)} 
\leq C'' \big(\norm{Dw^{\mathrm{op}}}_{L^{p}(\sigma^p)} 
+ \mu^{\frac{1}{p}}\norm{Dw^{\mathrm{op}}}_{L^{p}(\partial \sigma^p)} \big). 
\end{equation}

We then define the map \(\widetilde{w}\) by 
\[
\widetilde{w}(x)=w_{\sigma^p}(x) \text{ when } x\in \sigma^p \text{ and } \sigma^p \in \cK^{p}_\mu 
\]
and we extend \(\widetilde{w}\) by  \( w^\mathrm{op}\) outside \(K^{p}_\mu\). 
Then, \(\widetilde{w}\in (W^{1,p}\cap L^{\infty})(Q^p ; N^n)\) and \(\widetilde{w}|_{\partial Q^p}=u|_{\partial Q^p}\).{}
By additivity of the integral and by estimates \eqref{eqlemmatrimming1} and \eqref{eqlemmatrimming2}, we also have
\[
\begin{split}
 \norm{D\widetilde{w}}_{L^{p}(Q^p)}^p & = \sum_{\sigma^p\in \cK^{p}_\mu}\norm{Dw_{\sigma^p}}_{L^{p}(\sigma^p)}^p 
 + \norm{Dw^\mathrm{op}}_{L^{p}(Q^p\setminus K^{p}_\mu)}^p\\
& \leq \sum_{\sigma^p\in \cK^{p}_\mu} 2^{p-1}(C'')^{p} \bigl( \norm{Dw^{\mathrm{op}}}_{L^{p}(\sigma^p)}^p 
+ \mu\norm{Dw^{\mathrm{op}}}_{L^{p}(\partial \sigma^p)}^p \bigr) \\
&\hspace{5cm}+ \norm{Dw^\mathrm{op}}_{L^{p}(Q^p\setminus K^{p}_\mu)}^p\\
& \leq \Cl{cte-0605} \norm{Dw}_{L^{p}(Q^p)}^p.
\end{split}
\]
Applying estimate \eqref{eq470}, we deduce that
\begin{equation}
	\label{eqLocalTrimming-1}
	 \norm{D\widetilde{w}}_{L^{p}(Q^p)}
	\le 2\Cr{cte-0604}(\Cr{cte-0605})^{\frac{1}{p}} \norm{Du}_{L^{p}(Q^p)}.
\end{equation}

The map \(\widetilde{w}\) is continuous on  \(Q^p\setminus (K^{p}_\mu+Q^{p}_{2\rho\mu})\) since it agrees with the map \(w\) there. We introduce a cut-off function \(\theta\in C^{\infty}_c(Q^p)\) such that \(0 \le \theta \le 1\) in \(Q^{m}\) and \(\theta = 1\) on a neighborhood of \(K^{p}_\mu+Q^{p}_{2\rho\mu}\).{}
Given a family of mollifiers \((\varphi_{\varepsilon})_{\varepsilon>0}\), 
as a consequence of the VMO-property of \(\widetilde{w}\) in the critical-integrability case (see the proof of Proposition~\ref{proposition_smooth_bounded_density}) the Poincaré-Wirtinger inequality implies that there exists \(\overline{\varepsilon}>0\) such that, for every \(0<\varepsilon \leq \overline{\varepsilon}\), the set \((\varphi_{\varepsilon}*\widetilde{w})(\supp \theta)\) is contained in the neighborhood \(O\) where the nearest point projection \(\Pi\) is defined.
Since \(\widetilde{w}\) is continuous on \(Q^p\setminus (K^{p}_\mu+Q^{p}_{2\rho\mu})\), we can define
\[
v=\Pi\bigl(\theta (\varphi_{\varepsilon} * \widetilde{w})+ (1-\theta)\widetilde{w}\bigr)
\] 
for \(\varepsilon\) sufficiently small. 
This map \(v\) is an  extension of \(u|_{\partial Q^{p}}\) in the space \(W^{1,p}(Q^p ; N^n)\cap C^0(\overline{Q^p}; N^n)\). By the same calculation as in the proof of Proposition~\ref{proposition_continuous_trimming_property}, one has the estimate
\begin{equation}
	\label{eqLocalTrimming-2}
\norm{Dv}_{L^p(Q^p)}\leq \Cl{cte-0608} \norm{D \widetilde{w}}_{L^{p}(Q^p)}
\end{equation}
for \(\varepsilon\) small enough.
By estimates~\eqref{eqLocalTrimming-1} and~\eqref{eqLocalTrimming-2}, we have constructed a continuous extension of \(u|_{\partial Q^{p}}\) such that
\[{}
\norm{Dv}_{L^p(Q^p)}
\le 2 \Cr{cte-0604}(\Cr{cte-0605})^{\frac{1}{p}}\Cr{cte-0608}\norm{Du}_{L^{p}(Q^p)}.
\]
Proposition~\ref{proposition_continuous_trimming_property} now yields the conclusion.
\end{proof}

We now apply Lemma~\ref{lemmalocaltrimming} to prove that manifolds with uniform Lipschitz geometry satisfy the trimming property.
Another application of such a lemma in connection with the problem of weak sequential density of bounded Sobolev maps is investigated in~\cite{Bousquet-Ponce-VanSchaftingen-2017}.

\begin{proof}[Proof of Proposition~\ref{propositionboundedgeometry}]
Let \(u\in W^{1,p}(Q^p;N^n)\) be such that \(u|_{\partial Q^p}\in W^{1,p}(\partial Q^p ; N^n)\). 
Take \(0 < \kappa'' < \min{\{\kappa, \kappa'/C\}}\), where throughout the proof we refer to the notation of Definition~\ref{definition_bounded_geometry}.
By the Morrey--Sobolev embedding, there exists \(\alpha>0\) such that if \(\norm{Du}_{L^p(\partial Q^p)}\leq \alpha\), then there exists \(x \in \partial Q^p\) such that, for almost every \(y \in \partial Q^p\), we have
\[{}
\dist_{N^{n}}(u(y), u(x)) \le \kappa''.{}
\] 
Given a local chart \(\Psi\) on \(B_{N^n} (u (x); \kappa)\) as in Definition~\ref{definition_bounded_geometry},
the function \(\Psi \circ u\) belongs to \(W^{1, p}(\partial Q^{p}; \R^{n})\).
By the classical extension property of Sobolev functions, there exists \(w\in W^{1,p}(Q^{p} ; \R^n)\) such that \(w = \Psi\circ u\) on \(\partial Q^{p}\) in the sense of traces, and the following estimate holds 
\[{}
\resetconstant
\int_{Q^p} \abs{D w}^p 
\le \Cl{cte-0850} \int_{\partial Q^p}\abs{D(\Psi \circ u)}^p 
\le \Cr{cte-0850}C^{p} \int_{\partial Q^p}\abs{Du}^p.
\]

Observe that \(\Psi\circ u(\partial Q^p)\subset B_{\R^{n}}(\Psi (u(x)); \kappa')\).{}
Indeed, by the mean value inequality and the choice of \(\kappa''\), for almost every \(y \in \partial Q^p\) we have
\[{}
\dist_{\R^{n}}{(\Psi(u(y)), \Psi(u(x)))}
\le C \dist_{N^{n}}(u(y), u(x)) 
\le C \kappa''{}
< \kappa'.
\]
Thus, truncating \(w\) with a retraction on the ball \(B_{\R^{n}}(\Psi (u(x)); \kappa')\) if necessary, we may further assume that the image of the extension \(w\) satisfies \(w (Q^p) \subset B_{\R^{n}}(\Psi (u(x)); \kappa')\), since this does not modify the values of \(\Psi \circ u|_{\partial Q^p}\).
Defining the map \(v = \Psi^{-1} \circ w\), by composition of Sobolev maps with smooth functions it follows that \(v\in W^{1,p}(Q^p; N^{n})\) 
and
\[
\int_{Q^p}\abs{Dv}^p 
\le C^{p} \int_{Q^p} \abs{D w}^p 
\leq \C \int_{\partial Q^p} \abs{Du}^p.
\]
In view of Lemma~\ref{lemmalocaltrimming}, the proof is complete.
\end{proof}


\section{Proofs of Theorems~\ref{theoremMainNonInteger} and~\ref{theorem_Ap_CNS}}
\label{sectionProofs}

Let \(1\leq p\leq m\) and \(u \in W^{1, p}(Q^m; N^n)\).{}
We begin by extending \(u\) in a neighborhood of \(Q^{m}\) and then by taking a cubication that contains \(Q^{m}\). 
More precisely, by using reflexions across the boundary of \(Q^{m}\), we can extend \(u\) as a map in \(W^{1, p}(Q_{1 + 2\gamma}^m; N^n)\) for some \(\gamma > 0\). 
We also fix  \(0 < \rho < \frac{1}{2}\). 
Let \(\cK^m_\eta\) be a cubication of \(Q_{1+\gamma}^m\) of inradius \(0 < \eta \le \gamma\) such that
\[
2\rho\eta \le \gamma.
\]
For almost every  \(x, y\in Q^m_{1 + 2 \gamma}\), the function \(t\mapsto u(tx+(1-t)y)\) is an absolutely continuous path in \(N^n\) between \(u(x)\) and \(u(y)\). Hence, the geodesic distance \(\dist_{N^{n}}{(u(x),u(y))}\) between \(u(x)\) and \(u(y)\) can be estimated as follows:
\[
\dist_{N^{n}}{(u(x),u(y))}\leq \int_{0}^{1}\abs{Du(tx+(1-t)y)[x-y]}\dif t.
\]
As in the proof of the Poincar\'e-Wirtinger inequality, this implies that
\[{}
\resetconstant
\int_{Q^m_{1 + 2 \gamma}} \int_{Q^m_{1 + 2 \gamma}} \dist_{N^n}(u(x), u(y)) \dif x \dif y
\leq \C \int_{Q^{m}_{1+2\gamma}}\abs{Du}.
\]
It follows that, for almost every \(y\in Q^{m}_{1+2\gamma}\), 
\[
\int_{Q^m_{1 + 2 \gamma}} \dist_{N^n}(u(x), u(y)) \dif x <\infty.
\]
This implies that, for every \(a\in N^n\), the function \(x\mapsto \dist_{N^n}(u(x), a)\) is summable on \(Q^m_{1 + 2 \gamma}\).

\newcommand{\commentary}[1]{\textit{#1}}

\commentary{
We now distinguish the cubes in the cubication \(\cK^m_\eta\) in terms of \emph{good cubes} and\/ \emph{bad cubes}.
In a good cube, most of the values of the function \(u\) lie in a geodesic ball centered at some fixed point \(a \in N^n\), and \(u\) does not oscillate too much;
the latter is quantified in terms of the rescaled \(L^{p}\) norm of \(Du\).
}

We  fix a point \(a\in N^n\).
For every \(R>0\) and \(\lambda >0\),
we define the subskeleton \(\cG^m_\eta\) of \(\cK^{m}_\eta\) as the set of good cubes \(\sigma^m \in \cK^{m}_\eta\) in the sense that
\[
\fint_{\sigma^{m}+Q^{m}_{2\rho\eta}}\dist_{N^{n}}{(u(x),a)} \dif x \le R \quad \text{and} \quad
\frac{1}{\eta^{\frac{m-p}{p}}} \norm{Du}_{L^{p}(\sigma^m+Q^{m}_{2\rho\eta})} \le \lambda.
\]
We also introduce the subskeleton of bad cubes \(\cE^{m}_\eta\) defined as the complement 
of \(\cG^{m}_\eta\)  in \(\cK^{m}_\eta\).
Thus, by definition of \(\cE^m_{\eta}\), for every \(\sigma^m \in \cE^m_{\eta}\) we have
\[
R <  \fint_{\sigma^m + Q_{2\rho\eta}^m} \dist_{N^n}{(u(x), a)} \dif x
\quad
\text{or}
\quad
\lambda < \frac{1}{\eta^{\frac{m-p}{p}}} \norm{Du}_{L^{p}(\sigma^m+Q^{m}_{2\rho\eta})}.
\]
In the proof, we do not explicitly indicate the dependence of \(\cG^{m}_{\eta}\) and \(\cE^{m}_{\eta}\) on the parameters \(R\) and \(\lambda\).

\commentary{
On the bad cubes \(\cE^{m}_{\eta}\), we wish to replace the function \(u\) by some nicer, bounded, function.
To reach this goal, we would like to use the values of \(u\) on the lower-dimensional skeleton \(\cE^{\ell}_{\eta}\), where \(\ell=\floor{p}\), assuming that  \(u \in W^{1, p}(E^{\ell}_{\eta}; N^{n})\).{}
Indeed, for \(p\) noninteger we have \(p > \ell\) and then the Morrey--Sobolev embedding implies that \(u\) is continuous on \(E^{\ell}_{\eta}\).{}
We can thus propagate the values of \(u\) inside \(E^{m}_{\eta}\) by zero-degree homogenization.{}
When \(p\) is integer, we have \(p = \ell\) and we cannot rely on the Morrey--Sobolev embedding.
In this case, we apply the trimming property of dimension \(p\) to modify the function \(u\) on \(E^{\ell}_{\eta}\), keeping its values on the lower-dimensional skeleton \(E^{\ell - 1}_{\eta}\).{}
}

We first quantify the total volume of bad cubes.
More precisely, the Lebesgue measure of the set \(E^m_{\eta} + Q^m_{2\rho\eta}\) can be made as small as we want by a suitable choice of parameters \(R\) and \(\eta\).
This is a consequence of the following estimate:

\begin{claim}\label{claim_measure_bad}
The Lebesgue measure of the set \(E^m_{\eta} + Q^m_{2\rho\eta}\) satisfies
\[
	\bigabs{E^m_{\eta} + Q^m_{2\rho\eta}} 
	\le C \bigg( \frac{1}{R} \int_{Q^m_{1 + 2 \gamma}} \dist_{N^n}(u(x), a) \dif x + \frac{\eta^p}{\lambda^p} \int_{Q^m_{1 + 2 \gamma}} \abs{Du}^p \bigg).
\]
\end{claim}

\begin{proofclaim}
By finite subadditivity of the Lebesgue measure, we have 
\[
\bigabs{E^m_{\eta} + Q^m_{2\rho\eta}} 
\le \sum_{\sigma^{m} \in \cE^{m}_{\eta}} \bigabs{\sigma^m + Q^m_{2\rho\eta}}
 \le \C \eta^{m} \, (\#\cE^m_{\eta}). 
\]
From the definition of \(\cE^{m}_{\eta}\), we estimate the number \(\#\cE^m_{\eta} \) of bad cubes as follows:
\[
\begin{split}
\#\cE^m_{\eta} 
& \le  \sum_{\sigma^m \in \cE^m_{\eta}} \biggl(\frac{1}{|\sigma^m+Q^{m}_{2\rho\eta}| R} \int_{\sigma^m + Q_{2\rho\eta}^m} \dist_{N^n}(u(x), a) \dif x \\
 &\hspace{15em}+ \frac{1}{\eta^{m-p} \lambda^{p}} \int_{\sigma^m + Q_{2\rho\eta}^m} \abs{Du}^p \biggr)\\
& \le 
\frac{\C}{\eta^m}\bigg( \frac{1}{R} \int_{Q^m_{1 + 2 \gamma}} \dist_{N^n}(u(x), a) \dif x 
+ \frac{\eta^p}{\lambda^p} \int_{Q^m_{1 + 2 \gamma}}  \abs{Du}^p \bigg).
\end{split}
\]
Combining both estimates, we get the conclusion.
\end{proofclaim}

\commentary{
Since the cubication \(\cK^m_\eta\) is prescribed independently of \(u\), the \(L^{p}\) norm of \(Du\) on the skeleton \(E^{\ell}_{\eta}\) could be very large.{}
We thus begin by opening \(u\) in a neighborhood of \(E^{\ell}_{\eta}\), which provides a new function \(u^{\mathrm{op}}_{\eta}\) depending on, at most, \(\ell\) components around each face of \(E^{\ell}_{\eta}\).
}

Throughout the proof, we denote by 
\[{}
\ell=\floor{p}
\]
the integer part of \(p\).
We begin by opening the map \(u\) in a neighborhood of \(E^{\ell}_\eta\) if \(p<m\) and in a neighborhood of \(E^{m-1}_{\eta}\) if \(p=m\). 
More precisely, if \(\Phi^\mathrm{op} : \R^m \to \R^m\) is the smooth map given by Proposition~\ref{openingpropGeneral} with the parameter \(\rho\), we consider the opened map
\[
u^\mathrm{op}_\eta = u \circ \Phi^\mathrm{op}.
\]
When \(p<m\), we have that \(u^\mathrm{op}_\eta \in W^{1, p}(Q^m_{1+2\gamma}; N^n)\) and \(u^\mathrm{op}_\eta = u\) in 
the complement of \(E^{\ell}_\eta + Q^m_{2\rho\eta}\). 
Moreover, there exists \(C>0\) such that, for every \(\sigma^\ell \in \cE^{\ell}_\eta\), we have
\begin{equation}\label{eq1561}
\norm{D u^\mathrm{op}_\eta}_{L^p(\sigma^\ell +Q^{m}_{2\rho\eta})} 
\le C \norm{D u}_{L^p(\sigma^{\ell} + Q^m_{2\rho\eta})},
\end{equation}
and also
\begin{equation}
\label{inequalityMainOpening-weak-pni}
\norm{D u^\mathrm{op}_\eta - D u}_{L^p(Q^m_{1+2\gamma})} 
\le C \norm{D u}_{L^p(E^{\ell}_\eta + Q^m_{2\rho\eta})}.
\end{equation}
When \(p=m\), the integer \(\ell\) must be replaced by \(m-1\) in the above estimates.

\commentary{
We now consider a convolution of the opened map \(u_{\eta}^{\mathrm{op}}\).
The convolution parameter is not constant: it is small on the good cubes and quickly becomes zero as we enter the bad cubes.
Such a transition is made in a region having width of order \(\eta\).
}

More precisely, given \(0<\underline{\rho}<\rho\), we  consider a smooth function \(\psi_\eta \in C^\infty(Q^m_{1+2\gamma})\) such that 
\begin{enumerate}[$(a)$]
\item\label{2141} \(0 \leq  \psi_\eta < (\rho-\underline{\rho}) \eta\),
\item\label{2142} \(\psi_{\eta}=t\eta\) on \(G^{m}_{\eta}\), for some parameter \(0 < t < \rho-\underline{\rho}\) ,
\item\label{2143} \(\supp{\psi_{\eta}} \subset G^{m}_{\eta}+Q^{m}_{\underline{\rho}\eta}\) ,
\item\label{2144} \(\norm{D\psi_\eta}_{L^\infty(Q^m_{1+2\gamma})} <1\),
\end{enumerate}
The parameter \(t\) is fixed throughout the proof and is independent of \(\eta\), \(R\) and \(\lambda\).{}
Condition~$(\ref{2142})$ gives an upper bound on \(t\), while Condition~$(\ref{2144})$ imposes \(t\) to be typically smaller than \(\underline{\rho}\) and this can be achieved independently of the geometry of the cubication \(\cG_{\eta}\).

Given a mollifier \(\varphi \in C_c^\infty(B_1^m)\),
for every \(x\in Q^{m}_{1+\gamma}\) let
\[
u^\mathrm{sm}_\eta(x) 
= (\varphi_{\psi_\eta} \ast u^\mathrm{op}_\eta)(x)
= \int_{B_1^m} u^\mathrm{op}_\eta \bigl(x-\psi_{\eta}(x)y\bigr)\varphi(y) \dif y .
\]
Since \(0 < \psi_\eta \le \rho\eta < \gamma\), the smoothened map \(u^\mathrm{sm}_\eta : Q_{1 + \gamma}^m \to \R^\nu\) is well-defined.

\begin{claim}
\label{claimEstimateDirectDistance}
	The map \(u_{\eta}^{\mathrm{sm}}\) satisfies the estimates%
\begin{gather}
\label{eq1590}
\norm{u^\mathrm{sm}_\eta - u}_{L^p(Q^{m}_{1+\gamma})}
 \leq \sup_{v \in B_1^m}{\norm{\tau_{\psi_\eta v}(u) -  u}_{L^p(Q^{m}_{1+\gamma})}} + C \norm{u^\mathrm{op}_\eta - u}_{L^p(Q^{m}_{1+2\gamma})},
\\
\label{eq1594}
\begin{multlined}
	\norm{D u^\mathrm{sm}_\eta - D u}_{L^p(Q^{m}_{1+\gamma})}
 \leq \sup_{v \in B_1^m}{\norm{\tau_{\psi_\eta v}(Du) - D u}_{L^p(Q^{m}_{1+\gamma})}} \\
 + C \norm{D u}_{L^p(E^m_\eta + Q^m_{2\rho\eta})}.
\end{multlined}
\end{gather}
\end{claim}

\begin{proofclaim}
By Proposition~\ref{lemmaConvolutionEstimates} with \(\omega = Q^m_{1+\gamma}\), we have
\[
\norm{u^\mathrm{sm}_\eta - u^\mathrm{op}_\eta}_{L^p(Q^{m}_{1+\gamma})}
\leq \sup_{v \in B_1^m}{\norm{\tau_{\psi_\eta v}(u^\mathrm{op}_\eta) - u^\mathrm{op}_\eta}_{L^p(Q^{m}_{1+\gamma})}}. 
\]
We also observe that, for every \(v \in B_{1}^{m}\), we have
\begin{align*}
\norm{\tau_{\psi_\eta v}(u^\mathrm{op}_\eta) &- u^\mathrm{op}_\eta}_{L^p(Q^{m}_{1+\gamma})} \\
& \leq \norm{\tau_{\psi_\eta v}(u^\mathrm{op}_\eta) - \tau_{\psi_\eta v}(u)}_{L^p(Q^{m}_{1+\gamma})}\\
& \qquad + \norm{\tau_{\psi_\eta v}(u) - u}_{L^p(Q^{m}_{1+\gamma})} +\norm{u^\mathrm{op}_\eta - u}_{L^p(Q^{m}_{1+\gamma})}\\ 
& \le  \norm{\tau_{\psi_\eta v}(u) - u}_{L^p(Q^{m}_{1+\gamma})} + C\norm{u^\mathrm{op}_\eta - u}_{L^p(Q^{m}_{1+2\gamma})},
\end{align*}
and this proves \eqref{eq1590}.

We now consider the second estimate.
Since \(\norm{D\psi_\eta}_{L^\infty(Q^m_{1+2\gamma})} <1\), it also follows from Proposition~\ref{lemmaConvolutionEstimates} that
\begin{multline}
\label{inequalityMainSmoothing-weak-pni}
\resetconstant
\norm{D u^\mathrm{sm}_\eta - D u^\mathrm{op}_\eta}_{L^p(Q^{m}_{1+\gamma})}\\
\leq \sup_{v \in B_1^m}{\norm{\tau_{\psi_\eta v}(Du^\mathrm{op}_\eta) - Du^\mathrm{op}_\eta}_{L^p(Q^{m}_{1+\gamma})}} 
+ \C  \norm{D u^\mathrm{op}_\eta}_{L^p(A)},
\end{multline}
where
\(
A = \bigcup\limits_{x \in Q^{m}_{1+\gamma} \cap \supp{D\psi_\eta}}B_{\psi_\eta(x)}^m(x).
\)
From Property~$(\ref{2142})$, 
we have
\[
\supp{D\psi_\eta} \cap Q^{m}_{1+\gamma} \subset Q^{m}_{1+\gamma}\setminus G^{m}_\eta \subset E^{m}_{\eta}
\]
and, since \(\psi_\eta \le \rho\eta\), we deduce that \(A \subset E^m_\eta + Q^m_{\rho\eta}\). 
By Proposition~\ref{openingpropGeneral}, we then get
\begin{equation}
\label{inequalityOpeningSmallSet-weak-pni}
\norm{D u^\mathrm{op}_\eta}_{L^p(A)} 
\le \C \norm{Du}_{L^p(E^m_\eta + Q^m_{2\rho\eta})}.
\end{equation}
As in the proof of the first estimate, for every \(v \in B_{1}^{m}\) we also have 
\begin{multline}
\norm{\tau_{\psi_\eta v}(Du^\mathrm{op}_\eta) - Du^\mathrm{op}_\eta}_{L^p(Q^{m}_{1+\gamma})}
\leq  \norm{\tau_{\psi_\eta v}(Du) - Du}_{L^p(Q^{m}_{1+\gamma})}\\ 
+ C \norm{Du^\mathrm{op}_\eta - Du}_{L^p(Q^{m}_{1+2\gamma})}. \label{eq1202-pni}
\end{multline}
Combining estimates \eqref{inequalityMainSmoothing-weak-pni}--\eqref{eq1202-pni} and \eqref{inequalityMainOpening-weak-pni}, we complete the proof of \eqref{eq1594}.
\end{proofclaim}

\commentary{
Although the smoothened map \(u^{\textrm{sm}}_{\eta}\) need not lie on the manifold \(N^{n}\), we now quantify how far the set \(u^{\mathrm{sm}}_\eta(G^{m}_\eta)\) is with respect to some large geodesic ball \(B_{N^{n}}(a; \overline{R})\) with \(\overline R > R\).{}
Since there are many points of \(u(G^{m}_{\eta})\) on the geodesic ball \(B_{N^{n}}(a; {R})\), we can apply the Poincaré-Wirtinger inequality to establish such an estimate.
By choosing the parameter \(\lambda\) sufficiently small, we will later on be able to project back \(u^{\mathrm{sm}}_{\eta}\) to \(N^{n}\), at least on the good cubes \(\cG^{m}_{\eta}\).{}
}

\begin{claim}
\label{claimDistance}
There exists \(\overline{R} > R\) such that, for every \(\eta > 0\) and \(\lambda > 0\),
the directed Hausdorff distance to the geodesic ball \(B_{N^n}(a; \overline{R})\) satisfies
\[
\Dist_{B_{N^n}(a; \overline{R})}(u^{\mathrm{sm}}_\eta(G^{m}_\eta))
 \leq \frac{C'}{\eta^{\frac{m-p}{p}}}\max_{\sigma^m \in \cG^{m}_\eta} \norm{Du}_{L^{p}(\sigma^m+Q^{m}_{2\rho\eta})},
\]
for some constant \(C'>0\) depending on \(m\) and \(p\).
\end{claim}
Here, the directed Hausdorff distance from a set \(S \subset \R^\nu\) to the geodesic ball \(B_{N^n}(a; \overline{R})\) is defined as
\[
\Dist_{B_{N^n}(a; \overline{R})}{(S)} = \sup{\Big\{ \dist_{\R^{\nu}}{\bigl(x, B_{N^n}(a; \overline{R})\bigr)} \st x \in S \Big\}},
\]
where \(\dist_{\R^{\nu}}\) denotes the Euclidean distance in \(\R^{\nu}\).

\begin{proofclaim}[Proof of the claim]
Given \(\sigma^m \in \cG_{\eta}^m\) and \(\overline{R} > 0\), we consider the sets
\begin{align*}
W_{\overline R}
& = \Big\{z \in \sigma^m + Q^m_{2 \rho \eta} \st \dist_{N^n}(u(z), a) < \overline{R} \Big\},\\
Z_{\overline R}
& = \Big\{z \in \sigma^m + Q^m_{2 \rho \eta} \st \dist_{N^n}(u(z), a) \ge \overline{R} \Big\},\intertext{and their counterparts for the map \(u^\mathrm{op}_\eta\) obtained by the opening construction,}
W_{\overline R}^\mathrm{op}
& = \Big\{z \in \sigma^m + Q^m_{2 \rho \eta} \st \dist_{N^n}(u^\mathrm{op}_\eta(z), a) < \overline{R} \Big\},\\
Z_{\overline R}^\mathrm{op}
& = \Big\{z \in \sigma^m + Q^m_{2 \rho \eta} \st \dist_{N^n}(u^\mathrm{op}_\eta(z), a) \ge \overline{R} \Big\}.
\end{align*}
Observe that by definition \(u^\mathrm{op}_\eta(z) \in B_{N^n}(a; \overline{R})\) for every \(z \in W_{\overline R}^\mathrm{op}\).
Assuming that \(\abs{W_{\overline R}^\mathrm{op}} > 0\), then for every \(x \in \sigma^m\) we may estimate 
the distance from \(u^\mathrm{sm}_\eta(x)\) to \(B_{N^n}(a; \overline{R})\) in terms of an average integral as follows
\[
\dist_{\R^{\nu}}{\bigl(u^\mathrm{sm}_\eta(x), B_{N^n}(a; \overline{R})\bigr)} 
\le \fint_{W_{\overline R}^\mathrm{op}}  \abs{u^\mathrm{sm}_\eta(x) - u^\mathrm{op}_\eta(z)} \dif z.
\]
For every \(x \in \sigma^m \), we then have
\begin{equation*}
\resetconstant
\dist_{\R^{\nu}}{\bigl(u^\mathrm{sm}_\eta(x), B_{N^n}(a; \overline{R})\bigr)} 
\le 
\norm{\varphi}_{L^{\infty}(B_{1}^{m})} \fint_{W_{\overline R}^\mathrm{op}} \fint_{B_{\psi_\eta(x)}^m(x)}\abs{u^\mathrm{op}_\eta(y)-u^\mathrm{op}_\eta(z)}\dif y\dif z,
\end{equation*}
where \(\varphi\) is the mollifier used in the definition of \(u^\mathrm{sm}_\eta\).
Since both sets \(W_{\overline R}^\mathrm{op}\) and \(B_{\psi_\eta(x)}^m(x)\) are contained in \(\sigma^m + Q^m_{2 \rho \eta}\),
 by the Poincar\'e--Wirtinger inequality we deduce that
\begin{equation}
\label{eq_1971}
\dist_{\R^{\nu}}{(u^\mathrm{sm}_\eta(x), B_{N^n}(a; \overline{R}) )} 
\le 
\frac{\C  \eta^{2m}}{\abs{W_{\overline R}^\mathrm{op}}\,
\abs{B_{\psi_\eta(x)}^m(x)}}\frac{1}{\eta^{\frac{m-p}{p}}}\norm{Du^\mathrm{op}_\eta}_{L^p(\sigma^m + Q^m_{2 \rho \eta})}.
\end{equation}
Since \(\psi_\eta= t\eta\) on \(G^{m}_\eta\), for every \(x \in \sigma^m\) we have
\begin{equation}\label{eq_estim_Q_psi}
\abs{B_{\psi_\eta(x)}^m(x)} 
\geq \C \eta^m.
\end{equation}

We now estimate from below the quantity \(\abs{W_{\overline R}^\mathrm{op}}\).
Since \(\sigma^m \in \cG^m_\eta\), then by definition of \(\cG^m_\eta\) the average integral satisfies
\[
\fint_{\sigma^m + Q^m_{2 \rho \eta}} \dist_{N^n}(u(x), a) \dif x 
\le R,
\]
hence by the Chebyshev inequality we have
\begin{equation}\label{eq1093}
\frac{\abs{Z_{\overline{R}}}}{\abs{\sigma^m + Q^m_{2 \rho \eta}}} \, \overline{R}  
\le R.
\end{equation}
We now proceed with the choice of \(\overline{R}\).
Taking any \(\overline{R} > R\) such that
\begin{equation}\label{choice_R_1}
\abs{\sigma^m + Q^m_{2 \rho \eta}} \, R
\le  \frac{\bigabs{(\sigma^m + Q^m_{2 \rho \eta}) \setminus (\partial\sigma^m + Q^m_{2 \rho \eta})}}{2} \,\overline{R},
\end{equation}
we have
\[
\abs{Z_{\overline{R}}}
\le \frac{\bigabs{(\sigma^m + Q^m_{2 \rho \eta}) \setminus (\partial\sigma^m + Q^m_{2 \rho \eta})}}{2}.
\]
Since \(\sigma^m\) is a cube of inradius \(\eta\), by a scaling argument with respect to \(\eta\) this choice of \(\overline{R}\) is independent of \(\eta\).
Since the maps \(u^\mathrm{op}_\eta\) and \(u\) coincide in \((\sigma^m + Q^m_{2 \rho \eta}) \setminus (\partial\sigma^m + Q^m_{2 \rho \eta})\), we have
\[
Z_{\overline{R}}^{\mathrm{op}} 
\subset Z_{\overline{R}} \cup (\partial\sigma^m + Q^m_{2 \rho \eta}).
\]
By subadditivity of the Lebesgue measure and by the choice of \(\overline{R}\) we get
\[
\abs{Z^\mathrm{op}_{\overline{R}}}
\le \frac{\bigabs{(\sigma^m + Q^m_{2 \rho \eta}) \setminus (\partial\sigma^m + Q^m_{2 \rho \eta})}}{2} 
+ \bigabs{\partial\sigma^m + Q^m_{2 \rho \eta}},
\]
hence the measure of the complement set \(W^\mathrm{op}_{\overline{R}}\) satisfies
\begin{equation}\label{eq_estim_W}
\abs{W^\mathrm{op}_{\overline{R}}} \ge \frac{\bigabs{(\sigma^m + Q^m_{2 \rho \eta}) \setminus (\partial\sigma^m + Q^m_{2 \rho \eta})}}{2} 
= 2^{m-1} (\eta - 2\rho \eta)^m
= \C \eta^{m}.
\end{equation}
By estimates \eqref{eq_1971}, \eqref{eq_estim_Q_psi} and \eqref{eq_estim_W} for every \(x \in \sigma^m\) we deduce that with the above choice of \(\overline{R}\) we have 
\begin{equation}\label{eq1743}
\dist_{\R^{\nu}}{\bigl(u^\mathrm{sm}_\eta(x), B_{N^n}(a,\overline{R})\bigr)} 
\le  \frac{\C}{\eta^{\frac{m-p}{p}}}\norm{D u^\mathrm{op}_\eta}_{L^p(\sigma^m + Q^m_{2 \rho \eta})}.
\end{equation}
By subadditivity of the Lebesgue measure and by the properties of the opening construction, when \(p<m\) we have
\begin{equation}\label{eq_estim_cube_op}
\begin{split}
\norm{Du^{\mathrm{op}}_\eta}^{p}_{L^{p}(\sigma^m +Q^{m}_{2\rho\eta})} & \leq \norm{Du^{\mathrm{op}}_\eta}^{p}_{L^{p}((\sigma^m +Q^{m}_{2\rho\eta})\setminus (E^\ell_\eta + Q^{m}_{2\rho\eta}))} \\
&\hspace{3cm}+ \sum_{\substack{\sigma^\ell \in \cE^{\ell}_\eta\\ \sigma^\ell \subset \sigma^m}} \norm{Du^{\mathrm{op}}_\eta}^{p}_{L^{p}(\sigma^\ell +Q^{m}_{2\rho\eta})}\\
& \leq \C \norm{Du}^{p}_{L^{p}(\sigma^m +Q^{m}_{2\rho\eta})}. 
\end{split}
\end{equation}
When \(p=m\), then \(\ell\) must be replaced by \(m-1\) in the above inequality.
Together with \eqref{eq1743}, this implies the estimate we claimed.
\end{proofclaim}

\commentary{
We now quantify how far the smoothened map \(u^{\mathrm{sm}}_\eta\) is from the large geodesic ball \(B_{N^n}(a; \overline{R})\) on the part of the bad set \(E^{\ell}_\eta\) that lies in the transition between good and bad cubes.
The estimate uses the fact that the opened map \(u^{\mathrm{op}}\) depends on \(\ell\) components nearby \(E^{\ell}_\eta\) and that the convolution parameter is chosen very small in this region.
}

\begin{claim}\label{lemma_distance_bis}
There exists \(\overline{R} > R\) such that, for every \(\eta > 0\) and \(\lambda > 0\),
the directed Hausdorff distance to the geodesic ball \(B_{N^n}(a; \overline{R})\) satisfies
\[
\Dist_{B_{N^n}(a; \overline{R})}\bigl(u^{\mathrm{sm}}_\eta (E^{\ell}_\eta \cap \supp{\psi_\eta})\bigr)
 \leq \frac{C''}{\eta^{\frac{m-p}{p}}}\max_{\sigma^m \in \cG^{m}_\eta} \norm{Du}_{L^{p}(\sigma^m+Q^{m}_{2\rho\eta})},
\]
for some constant \(C'' > 0\) depending on \(m\) and \(p\).
\end{claim}
\begin{proofclaim}
Using Property~\eqref{2143} satisfied by the function \(\psi_\eta\) (see page~\pageref{2143}), one gets
\[
E^{\ell}_\eta \cap \supp \psi_\eta 
\subset (E^{\ell}_{\eta} \cap G^{\ell}_\eta) \cup \big( ( E^{\ell-1}_\eta\cap G^{\ell-1}_\eta) 
+ Q^{m}_{\underline{\rho}\eta} \big). 
\]
By Claim~\ref{claimDistance} above it thus suffices to prove that, for every \(\tau^{\ell-1}\in \cE^{\ell-1}_\eta \cap \cG^{\ell-1}_\eta \), we have 
\begin{equation}\label{eq1142}
\resetconstant
\Dist_{B_{N^n}(a; \overline{R})} \bigl(u^{\mathrm{sm}}_\eta(\tau^{\ell-1}+Q^{m}_{\underline{\rho}\eta} )\bigr)
\leq \frac{\Cl{cte-1502}}{\eta^{\frac{m-p}{p}}}\max_{\sigma^m \in \cG^{m}_\eta} \norm{Du}_{L^{p}(\sigma^m+Q^{m}_{2\rho\eta})}.
\end{equation}
For this purpose, we observe that there exists \(\overline{R} > R\) such that the map \(u^{\mathrm{op}}_{\eta}\) can be constructed with the following additional property: for every \(\tau^{\ell-1}\in \cE^{\ell - 1}_\eta\cap \cG^{\ell -1}_\eta\),
\begin{equation}\label{eq1171}
\Dist_{B_{N^n}(a; \overline{R})} \big(u^{\mathrm{op}}_\eta(\tau^{\ell - 1} +Q^{m}_{\rho\eta})\big)
\leq \frac{\C}{\eta^{\frac{m-p}{p}}}\norm{Du^\mathrm{op}_\eta}_{L^p(\tau^{\ell-1} + Q^m_{\rho\eta})}.
\end{equation}
Indeed, for every \(\sigma^m\in \cG^m_\eta\) and for every \(\overline{R} > R\) such that
\begin{equation}\label{choice_R_2}
\abs{\sigma^m + Q^m_{2 \rho \eta}} \, R  
\le  \frac{\abs{Q^m_{\rho\eta}}}{2} \, \overline{R}, 
\end{equation}
we have, by \eqref{eq1093},
\[
\abs{Z_{\overline{R}}}
\le \frac{\abs{Q^m_{\rho\eta}}}{2}.
\]
Again by a scaling argument with respect to \(\eta\), this choice of \(\overline{R}\) is independent of \(\eta\).
For each vertex \(v\) of the cube \(\sigma^m\) and for at least half of the points \(x\) of the cube \(Q_{\rho\eta}^m(v)\), 
we thus have \(u(x) \in B_{N^n}(a; \overline{R})\).
Since the opening construction is based on a Fubini-type argument (see the explanation preceding  Proposition~\ref{openingpropSimplex}), we may thus assume that for each vertex \(v\) of \(\partial\sigma^m\cap E^0_\eta\),
the common value of \(u^\mathrm{op}_\eta\) in \(Q_{\rho\eta}^m(v)\) belongs to \(B_{N^n}(a; \overline{R})\).

Consider an \((\ell - 1)\)-dimensional face \(\tau^{\ell - 1}\in \cE^{\ell - 1}_\eta\cap \cG^{\ell - 1}_\eta\).
Since \(p > \ell - 1\), by the Morrey--Sobolev inequality we have, for every \(y, z \in \tau^{\ell-1}\),
\[
\dist_{\R^\nu}{\bigl(u^\mathrm{op}_\eta(y), u^\mathrm{op}_\eta(z)\bigr)}
\le \Cl{cte-1029} {\eta^{1 - \frac{\ell - 1}{p}}} \norm{Du^\mathrm{op}_\eta}_{L^p(\tau^{\ell-1})}.
\]
On the other hand, since  the map \(u^{\mathrm{op}}_\eta\) is, by construction, an \((\ell - 1)\)-dimensional map in \(\tau^{\ell -1} + Q^{m}_{\rho\eta}\), we have \(u_{\eta}^{\mathrm{op}}(\tau^{\ell-1}+Q^{m}_{\rho\eta})=u_{\eta}^{\mathrm{op}}(\tau^{\ell-1})\) and also
\[
\norm{Du^\mathrm{op}_\eta}_{L^p(\tau^{\ell-1})}
\le \frac{\Cl{cte-1030}}{\eta^{\frac{m - (\ell - 1)}{p}}} \norm{Du^\mathrm{op}_\eta}_{L^p(\tau^{\ell-1} + Q^m_{\rho\eta})}.
\]
This implies that, for every \(y, z\in \tau^{\ell-1}+Q^{m}_{\rho\eta}\),
\[
\dist_{\R^\nu}{\bigl(u^\mathrm{op}_\eta(y), u^\mathrm{op}_\eta(z)\bigr)}
\le \frac{\Cr{cte-1029} \Cr{cte-1030}}{\eta^{\frac{m-p}{p}}}\norm{Du^\mathrm{op}_\eta}_{L^p(\tau^{\ell-1} + Q^m_{\rho\eta})}.\label{eq2051}
\]
Taking as \(z\) any vertex of \(\tau^{\ell - 1}\) in \(\cG^{0}_\eta\),  we thus obtain
estimate \eqref{eq1171}. 

We now complete the proof of \eqref{eq1142}.
Recall that the map \(u^{\mathrm{sm}}_\eta\) is obtained from \(u^{\mathrm{op}}_\eta\) by convolution with parameter  \(\psi_\eta\). 
Hence, for every \(\tau^{\ell - 1}\in \cG^{\ell-1}_{\eta}\cap \cE^{\ell-1}_\eta\) and for every \(x \in \tau^{\ell - 1}+Q^{m}_{\underline{\rho}\eta}\) such that \(\psi_\eta(x)\not=0\), by the triangle inequality we have
\begin{multline*}
\dist_{\R^\nu}{\bigl(u^\mathrm{sm}_\eta(x), B_{N^n}(a, \overline{R})\bigr)}\\
\le  \C
\fint_{Q^{m}_{\psi_\eta(x)}(x)} \fint_{Q^{m}_{\psi_\eta(x)}(x)}|u^{\mathrm{op}}_\eta(z)-u^{\mathrm{op}}_\eta(y)|\dif y\dif z \\
+ \fint_{Q^{m}_{\psi_\eta(x)}(x)} \dist_{\R^{\nu}}{\bigl(u^{\mathrm{op}}_\eta(y), {B_{N^n}(a, \overline{R})}\bigr)} \dif y.
\end{multline*}
Since \(x\in \tau^{\ell - 1}+Q^m_{\underline{\rho}\eta}\) and \(\psi_\eta(x)< (\rho-\underline{\rho})\eta\), we have \(Q^{m}_{\psi_\eta(x)}(x)\subset \tau^{\ell - 1} + Q^{m}_{\rho\eta}\). 
Together with \eqref{eq1171}, this implies that, for every \(y\in Q^{m}_{\psi_\eta(x)}(x)\),
\[
\dist_{\R^{\nu}}{\bigl(u^{\mathrm{op}}_\eta(y), {B_{N^n}(a, \overline{R})}\bigr)}
\leq \frac{\C}{\eta^{\frac{m-p}{p}}}\norm{Du^\mathrm{op}_\eta}_{L^p(\tau^{\ell-1} + Q^m_{\rho\eta})}.
\]  
By the Poincar\'e--Wirtinger inequality, we deduce that
\begin{multline}
	\label{eq2145}
\dist_{\R^\nu}{(u^\mathrm{sm}_\eta(x), B_{N^n}(a, \overline{R}))} 
\le \C\bigg(\frac{1}{\psi
_\eta(x)^{\frac{m-p}{p}}}\norm{Du^{\mathrm{op}}_\eta}_{L^p(Q^{m}_{\psi_\eta(x)}(x))}\\
+\frac{1}{\eta^{\frac{m-p}{p}}}\norm{Du^\mathrm{op}_\eta}_{L^p(\tau^{\ell-1} + Q^m_{\rho\eta})}
\bigg).
\end{multline}

Next, from the opening construction, for every cube \(Q^{m}_r(x)\subset \tau^{\ell-1}+Q^{m}_{\rho\eta}\) we have
\begin{equation}
	\label{eq2154}
\frac{1}{r^{m-p}} \int_{Q^{m}_r(x)}\abs{Du^{\mathrm{op}}_\eta}^p \leq \frac{\C}{\eta^{m-p}} 
\int_{\tau^{\ell-1}+Q^{m}_{\rho\eta}}\abs{Du^{\mathrm{op}}_\eta}^p. 
\end{equation}
Indeed, this follows directly from~\eqref{eq2087} when \(p < m\). 
When \(p = m\), one can proceed along the lines of the proof of estimate \eqref{eq2087} with \(\ell\) replaced by \(m\).

Combining inequalities \eqref{eq2145} and \eqref{eq2154} with \(r = \psi_{\eta}(x)\), we get
\[
\dist_{\R^\nu}{(u^\mathrm{sm}_\eta(x), B_{N^n}(a, \overline{R}))} 
\le \frac{\C}{\eta^{\frac{m-p}{p}}} \norm{Du^{\mathrm{op}}_\eta}_{L^p(\tau^{\ell-1}+Q^{m}_{\rho\eta})}.
\]
In view of the estimates satisfied by the opening construction and the fact that \(\tau^{\ell-1}\in \cG^{\ell-1}_\eta\), for every \(x\in \tau^{\ell-1}+Q^{m}_{\underline{\rho}\eta}\) such that \(\psi_\eta(x)\not=0\) we have
\[
 \dist_{\R^\nu}{\bigl(u^\mathrm{sm}_\eta(x), B_{N^n}(a, \overline{R})\bigr)} 
\le \frac{\C}{\eta^{\frac{m-p}{p}}}\max_{\sigma^m\in \cG^{m}_\eta}\norm{Du}_{L^p(\sigma^m+Q^{m}_{2\rho\eta})},
\]
from which \eqref{eq1142} follows.
If \(\psi_\eta(x)=0\), then \(u_{\eta}^{\mathrm{sm}}(x)=u_{\eta}^{\mathrm{op}}(x)\), and the above inequality remains true by \eqref{eq1171}.
This completes the proof of the claim.
\end{proofclaim}

\commentary{
Up to now, the parameters \(\lambda\) and \(\eta\) were arbitrary. 
In the following, they will be subject to some restrictions.
Our aim is to make \(u^{\mathrm{sm}}_{\eta}\) sufficiently close to \(N^{n}\) on the set \(E_{\eta}^{\ell} \cup G_{\eta}^{m}\), so that we can project \(u^{\mathrm{sm}}_{\eta}\) back to the manifold \(N^{n}\).{}
We then extend the projected map to \(E_{\eta}^{m}\) using the zero-degree homogenization.
}

For a given \(R>0\), we take  \(\overline{R} > R\) satisfying the conclusions of  Claims~\ref{claimDistance} and~\ref{lemma_distance_bis}. 
For any such \(\overline{R}\), let \(\iota_{\overline{R}} > 0\) be such that   
\[
\overline{B_{N^n}(a; \overline{R}) + B^\nu_{\iota_{\overline{R}}}}\subset O.
\]
Remember that the geodesic ball \(B_{N^n}(a; \overline{R})\) is a relatively compact subset of \(N^{n}\) and that \(O\) is an open neighborhood of \(N^n\) in \(\R^\nu\) on which can be defined a smooth retraction \(\Pi : O \to N^{n}\) such that \(D\Pi\in L^{\infty}(O)\); see Section~\ref{section_nearest_point_projection}.
We also take \(\lambda>0\) depending on \(\overline{R} >0\), whence also on \(R>0\), such that
\begin{equation}\label{eq1300}
\lambda 
\le \frac{\iota_{\overline{R}}}{\max{\{C',C''\}}},
\end{equation}
where \(C', C'' > 0\) are the constants given  by  Claims~\ref{claimDistance} and~\ref{lemma_distance_bis}, respectively. 
On the one hand,  for every good cube \(\sigma^m\in \cG^m_\eta\) we have
\[
\frac{1}{\eta^{\frac{m-p}{p}}}\norm{Du}_{L^{p}(\sigma^m+Q^{m}_{2\rho\eta})}
\leq  \frac{\iota_{\overline{R}}}{\max{\{C',C''\}}}. 
\]
By the estimate from Claim~\ref{claimDistance}, this implies that 
\[
u^{\mathrm{sm}}_\eta(G^m_\eta)\subset \overline{B_{N^n}(a; \overline{R}) + B^\nu_{\iota_{\overline{R}}}}\subset O.
\]
On the other hand, Claim~\ref{lemma_distance_bis} implies that 
\[
u^{\mathrm{sm}}_\eta(E^{\ell}_\eta \cap \supp{\psi_\eta} ) 
\subset \overline{B_{N^n}(a; \overline{R}) + B^\nu_{\iota_{\overline{R}}}}\subset O.
\]
On \(K^m_\eta\setminus \supp\psi_\eta\), we have \(u_{\eta}^{\mathrm{sm}}= u_{\eta}^{\mathrm{op}}\). 
In particular,
\[
u_{\eta}^{\mathrm{sm}}(E^\ell_\eta \setminus \supp{\psi_\eta})
\subset N^n.
\]
This proves  that 
\begin{equation}\label{eq_range_badset}
u^{\mathrm{sm}}_\eta(E^{\ell}_\eta) \subset N^n \cup  \overline{B_{N^n}(a; \overline{R}) + B^\nu_{\iota_{\overline{R}}}}\subset O.
\end{equation}

We now define the projected map \(u^{\mathrm{pr}}_\eta : G^{m}_\eta \cup E^{\ell}_\eta \to N^n \) by setting 
\[
u^{\mathrm{pr}}_\eta
=\Pi \circ u^{\mathrm{sm}}_\eta.
\]
On \(G^{m}_\eta\), the map \(u^{\mathrm{pr}}_\eta\) is smooth and we have:

\begin{claim}
\label{claimEstimateProjectionGoodCubes}
The map \(u_{\eta}^{\mathrm{pr}}\) satisfies
\begin{multline*}
\norm{Du^{\mathrm{pr}}_\eta - Du}_{L^p(G^{m}_\eta)}\\
\leq \norm{D\Pi}_{L^{\infty}(O)} \norm{Du^{\mathrm{sm}}_\eta-Du}_{L^{p}(G^{m}_\eta)}
 + \bignorm{\abs{D\Pi(u^{\mathrm{sm}}_\eta)-D\Pi(u)}\,\abs{Du}}_{L^{p}(G^{m}_\eta)}.
\end{multline*}
\end{claim}
\begin{proofclaim}
Since \(u^{\mathrm{pr}}_\eta =\Pi \circ u^{\mathrm{sm}}_\eta\), \(u = \Pi \circ u\) and \(u^{\mathrm{sm}}_\eta(G^{m}_\eta)\) is contained in a compact subset of \(O\), Lemma~\ref{lemma_chain_rule_C1} and the triangle inequality imply
\begin{multline*}
 \norm{D u^\mathrm{pr}_\eta - Du}_{L^p(G^{m}_\eta)}
\leq \norm{D\Pi(u^{\mathrm{sm}}_\eta)}_{L^{\infty}(G^{m}_\eta)}\norm{Du^{\mathrm{sm}}_\eta-Du}_{L^{p}(G^{m}_\eta)} \\
+\bignorm{\abs{D\Pi(u^{\mathrm{sm}}_\eta)-D\Pi(u)}\,\abs{Du}}_{L^{p}(G^{m}_\eta)}.
\qedhere 
\end{multline*}
\end{proofclaim}

\commentary{
We now make sure that the zero-degree homogenization can be successfully performed on \(E^{m}_\eta\) using the values of \(u^{\mathrm{pr}}_\eta\) on \(E_{\eta}^{\ell}\).{}
For this purpose, we need some local control of the \(L^{p}\) norm of \(D u^{\mathrm{pr}}_\eta\) in terms of the rescaled \(L^{p}\) norm of \(Du\).{}
This is enough to get the strong convergence in \(W^{1, p}\) since the bad set \(E^{m}_\eta\) is small.
}

\begin{claim}\label{lemma_usm_estimate}
The map \(u^{\mathrm{pr}}_\eta|_{E^{\ell}_\eta}\) belongs to \(W^{1, p}(E^{\ell}_\eta; N^n)\) and, 
for every \(\tau^{\ell} \in \cE_{\eta}^{\ell}\) , we have \(u^{\mathrm{pr}}_\eta|_{\partial\tau^{\ell}} \in W^{1, p}(\partial\tau^{\ell}; N^n)\) and also
\begin{equation}
\label{eq2260}
\norm{Du_{\eta}^{\mathrm{pr}}}_{L^{p}(\tau^{\ell})}
\leq \frac{C}{\eta^{\frac{m-\ell}{p}}}\norm{Du}_{L^{p}(\tau^{\ell}+Q^{m}_{2\rho\eta})},
\end{equation}
for some constant \(C > 0\) depending on \(m\) and \(p\).
\end{claim}

\begin{proofclaim}
By Lemma~\ref{lemma_chain_rule_C1}, the inclusion \eqref{eq_range_badset} and the fact that \(u_{\eta}^{\mathrm{pr}}=\Pi \circ u_{\eta}^{\mathrm{sm}}\), it is enough to prove the claim for \(u_{\eta}^{\mathrm{sm}}\) instead of \(u^{\mathrm{pr}}_\eta\). 
We first prove that, for every \(\tau^\ell \in \cE^\ell\), the restriction \(u_{\eta}^{\mathrm{sm}}|_{\tau^\ell}\) belongs to \(W^{1,p}(\tau^\ell ; \R^\nu)\) and satisfies the estimate above.
For this purpose, we can assume that 
\[
\tau^{\ell}=(-\eta, \eta)^{\ell} \times \{0''\},
\]
where \(0'' \in \R^{m - \ell}\).
Accordingly, we write every vector \(y \in \R^{m}\) as \(y = (y',y'')\in \R^\ell \times \R^{m-\ell}\).
Since, for \(y\in \tau^{\ell}+Q^{m}_{\rho\eta}\),
\[
u_{\eta}^{\mathrm{op}}(y) = u_{\eta}^{\mathrm{op}}(y', 0''),
\]
and since \(D\psi_\eta\) is uniformly bounded with respect to \(\eta\),
for every \(x' \in \tau^{\ell}\) we have
\[{}
\resetconstant
\begin{split}
\abs{Du^{\mathrm{sm}}_\eta(x', 0'')}^p 
& \leq \C \int_{B^{m}_1} \abs{Du^{\mathrm{op}}_\eta(x'-\psi_{\eta}(x',0'')y', 0'')}^p\dif y\\
& \leq \Cl{cte-1504} \int_{B^{\ell}_1}\abs{Du^{\mathrm{op}}_\eta(x'-\psi_{\eta}(x',0'')y', 0'')}^p \dif y'. 
\end{split}
\]
Hence, by Fubini's theorem,
\[
\int_{\tau^{\ell}}\abs{Du_{\eta}^{\mathrm{sm}}(x', 0'')}^p\dif x'
\leq \Cr{cte-1504}  
\int_{B^{\ell}_1}\int_{\tau^{\ell}}\abs{Du^{\mathrm{op}}_\eta(x'-\psi_{\eta}(x',0'')y', 0'')}^p\dif x'\dif y'.
\]
Using the change of variables \(z'= x'- \psi_{\eta}(x',0'')y'\) with respect to the variable \(x'\), we get
\begin{multline*}
\int_{\tau^{\ell}}\abs{Du_{\eta}^{\mathrm{sm}}(x', 0'')}^p\dif x'\\
 \begin{aligned}
& \leq \frac{\C}{1-\norm{D\psi_\eta}_{L^{\infty}(K^{m}_\eta)}} 
 \int_{B^{\ell}_1}\dif y'\int\limits_{(-(1+\rho)\eta, (1+\rho)\eta)^{\ell}}\abs{Du_{\eta}^{\mathrm{op}}(z', 0'')}^p\dif z'\\
& \leq {\C} \int\limits_{(-(1+\rho)\eta, (1+\rho)\eta)^{\ell}}\abs{Du_{\eta}^{\mathrm{op}}(z', 0'')}^p\dif z'.
 \end{aligned}
\end{multline*}
We observe that
\[
 \int\limits_{(-(1+\rho)\eta, (1+\rho)\eta)^{\ell}}\abs{Du_{\eta}^{\mathrm{op}}(z',0'')}^p\dif z'
\leq  \frac{\C}{\eta^{m-\ell}}\int_{\tau^{\ell}+Q^{m}_{\rho\eta}}\abs{Du_{\eta}^{\mathrm{op}}}^p.
\]
Combining both inequalities, we get
\[
 \int_{\tau^{\ell}}\abs{Du_{\eta}^{\mathrm{sm}}}^p \leq 
\frac{\C}{\eta^{m-\ell}}\int_{\tau^{\ell}+Q^{m}_{\rho\eta}}\abs{Du_{\eta}^{\mathrm{op}}}^p.
\]
When \(p < m\), we deduce that \(u_{\eta}^{\mathrm{sm}}|_{\tau^\ell}\) belongs to \(W^{1,p}(\tau^\ell ; \R^\nu)\) and \eqref{eq2260} holds by the estimate in Assertion~$(\ref{itemgenopeningprop6})$ of Proposition~\ref{openingpropGeneral} satisfied by the opened map \(u_{\eta}^{\mathrm{op}}\).{}
When \(p = m\), we rely instead on the estimate~\eqref{eq_estim_cube_op}, which holds with \(\ell\) replaced by \(m-1\).{}

We can prove in a similar way that \(u_{\eta}^{\mathrm{sm}}|_{\tau^{\ell - 1}}\) belongs to \(W^{1,p}(\tau^{\ell - 1} ; \R^\nu)\) for every \(\tau^{\ell - 1} \in \cE_{\eta}^{\ell - 1}\).{}
Moreover, the map \(u_{\eta}^{\mathrm{op}}\) is \((\ell-1)\)-dimensional on \(E^{\ell-1}_{\eta}+Q^{m}_{\rho\eta}\) and thus continuous by the Morrey--Sobolev embedding. 
This implies that \(u_{\eta}^{\mathrm{sm}}\) is continuous on a neighborhood of \(E^{\ell-1}_\eta\). 
Hence, the first part of the claim also follows.
\end{proofclaim}

\commentary{
The zero-degree homogenization of \(u_{\eta}^{\mathrm{pr}}\) should be performed inside \(E^{m}_{\eta}\) so as to preserve the values of \(u_{\eta}^{\mathrm{pr}}\) on the (higher-dimensional) common faces with good cubes.
Such a construction naturally yields a Sobolev map on the entire domain \(E_{\eta}^{m} \cup G_{\eta}^{m}\).
}

By construction, the map \(u_{\eta}^{\mathrm{sm}}\) is smooth on a neighborhood of \(E^{i}_\eta \cap G^{i}_\eta\) for every \(i\in \{\ell, \dots, m-1\}\). In particular, \(u_{\eta}^{\mathrm{pr}}|_{E^{i}_\eta \cap G^{i}_\eta}\) belongs to \(W^{1, p}(G^{i}_\eta \cap E^i_\eta; \R^{\nu})\), and we now estimate the \(L^p\) norm of \(D(u_{\eta}^{\mathrm{pr}}|_{E^{i}_\eta\cap G^{i}_\eta}) \) when \(p<m\) (and thus \(\ell \leq m-1 \)):

\begin{claim}\label{lemma_sm_good}
Assume that \(p<m\).
For every \(i\in \{\ell, \dotsc, m-1\}\), we have
\[
\norm{Du^{\mathrm{pr}}_\eta}_{L^{p}(E^i_\eta \cap G^{i}_\eta)}\leq \frac{C}{\eta^{\frac{m-i}{p}}} \norm{Du}_{L^{p}(E^i_\eta+Q^{m}_{2\rho\eta})},
\]
for some constant \(C>0\) depending on \(m\) and \(p\).
\end{claim}
\begin{proofclaim} Once again, we only need to prove the estimate with \(u^{\mathrm{sm}}_\eta\) instead of \(u^{\mathrm{pr}}_\eta\).
Fix  \(i\in \{\ell, \dotsc, m-1\}\). 
For every \(x \in E^i_\eta \cap G^{i}_\eta\), we have \(\psi_\eta(x) = t\eta\) and thus
\[
u^{\mathrm{sm}}_{\eta}(x)=\int_{B^m_1} u^{\mathrm{op}}_{\eta}(x-t\eta y )\varphi(y)\dif y. 
\]
Hence, by Jensen's inequality and a change of variable,
\[{}
\resetconstant
\abs{Du^{\mathrm{sm}}_{\eta}(x)}^p
\leq {\Cl{cte-1505}} \int_{B^m_1} \abs{Du^{\mathrm{op}}_{\eta}(x-t\eta y )}^p\dif y
= \frac{\Cr{cte-1505}}{(t\eta)^m} \int_{B^m_{t\eta} (x)} \abs{Du^{\mathrm{op}}_{\eta}}^p. 
\]
Since the parameter \(t < \rho\) is fixed, we can incorporate it in the constant.
Integrating both members with respect to the \(i\)-dimensional Hausdorff measure over \(E^i_\eta \cap G^{i}_\eta\), by Fubini's theorem we get
\[{}
\int_{E^i_\eta \cap G^{i}_\eta} \abs{Du^{\mathrm{sm}}_{\eta}}^p 
\leq \frac{{\C}}{\eta^{m-i}} \int_{E^i_\eta+Q^{m}_{2\rho\eta}} \abs{Du^{\mathrm{op}}_{\eta}}^p.
\]
By construction of \(u^{\mathrm{op}}_{\eta}\), we also have
\[
\begin{split}
 \int_{E^i_\eta+Q^{m}_{2\rho\eta}} \abs{Du^{\mathrm{op}}_{\eta}}^p 
&\leq \int_{(E^i_\eta+Q^{m}_{2\rho\eta}) \setminus (E^{\ell}_\eta + Q^{m}_{2\rho\eta})}\abs{Du}^p + \sum_{\tau^\ell\in \cE^\ell_\eta} \;	 \int_{\tau^\ell +Q^{m}_{2\rho\eta}}\abs{Du_{\eta}^{\mathrm{op}}}^p\\
& \leq \C \int_{E^{i}_\eta+Q^{m}_{2\rho\eta}} \abs{Du}^p,  
\end{split}
\]
and the conclusion follows.
\end{proofclaim}

\commentary{
We now proceed to construct a bounded extension \(u^{\mathrm{be}}_\eta\) to \(E_{\eta}^{m}\) of the map \(u^{\mathrm{pr}}_\eta|_{E^{m}_\eta \cap G^{m}_{\eta}}\).{}
It follows from Claim~\ref{lemma_usm_estimate} and the Morrey--Sobolev embedding that, when \(\ell < p\), the projected map \(u^{\mathrm{pr}}_{\eta}\) is bounded on \(E^{\ell}_\eta\).{}
In this case, we apply the zero-degree homogenization to extend \(u^{\mathrm{pr}}_\eta\) inside \(E^{m}_{\eta}\) using its values on \(E^{\ell}_{\eta}\).
When \(\ell = p\), we can only infer that \(u^{\mathrm{pr}}_{\eta}\) is bounded on the lower-dimensional skeleton \(E_{\eta}^{\ell - 1}\).{}
We then apply the trimming property of dimension \(p\), restated in terms of Sobolev maps by Proposition~\ref{proposition_continuous_trimming_property}, to obtain a new, bounded and continuous, function on \(E_{\eta}^{\ell}\).{}
The quantitative character of the trimming property ensures that this new function satisfies the same energy bounds. 
}

\begin{claim}
\label{claimEstimateProjectionBadCubes}
If \(\ell < p\) or if \(\ell = p\) and \(N^{n}\) satisfies the trimming property of dimension \(p\), then there exists a map \(u^{\mathrm{be}}_\eta \in (W^{1,p} \cap L^{\infty})(E^m_\eta ; N^n)\) such that
 \(u^{\mathrm{be}}_\eta = u^{\mathrm{pr}}_\eta \) on \(E^{m}_\eta \cap G^m_\eta\) and
\[
\norm{Du^{\mathrm{be}}_\eta}_{L^{p}(E_{\eta}^{m})} 
\leq C\bigg( \eta^{\frac{m - \ell}{p}} \norm{Du^{\mathrm{pr}}_\eta}_{L^{p}(E_{\eta}^{\ell})} + \sum_{i = \ell+1 }^{m-1} \eta^{\frac{m - i}{p}} \norm{Du^{\mathrm{pr}}_\eta}_{L^{p}(E^i_\eta \cap G^{i}_\eta)} \bigg),
\]
for some constant \(C>0\) depending on \(m\), \(p\) and \(N^n\).
\end{claim}

\begin{proofclaim}
We first define the extension \(u^{\mathrm{be}}_\eta\) on the subskeleton \(E_{\eta}^{\ell}\).{}
On every face \(\tau^{\ell}\in \cE^{\ell}_\eta \cap \cG^{\ell}_\eta\), the map \(u^{\mathrm{pr}}_{\eta} \)  is smooth, and we set \(u^{\mathrm{be}}_{\eta} =u^{\mathrm{pr}}_{\eta} \) in \(\tau^{\ell}\).{}
When \(\ell < p\), the map \(u^{\mathrm{pr}}_{\eta}\) is continuous on \(E^{\ell}_\eta\), and we also set \(u^{\mathrm{be}}_{\eta} = u^{\mathrm{pr}}_{\eta}\) on \(\tau^{\ell}\in \cE^{\ell}_\eta \setminus \cG^{\ell}_\eta\).
When \(\ell = p\), we assume that the trimming property holds, whence by Proposition \ref{proposition_continuous_trimming_property} 
we may replace \(u^{\mathrm{pr}}_\eta\) on each face \(\tau^{\ell} \in \cE^{\ell}_\eta \setminus \cG^{\ell}_\eta\) 
without changing its trace on \(E_{\eta}^{\ell - 1}\) to get a continuous map 
\(u^{\mathrm{be}}_{\eta} \in W^{1, p}(\tau^{\ell}; N^n)\)
such that 
\begin{equation}\label{eq2189}
\resetconstant
\norm{Du^{\mathrm{be}}_\eta}_{L^{p}(\tau^{\ell})} 
\leq {\C} \norm{Du^{\mathrm{pr}}_\eta}_{L^{p}(\tau^{\ell})}.
\end{equation}
The map \(u^{\mathrm{be}}_{\eta}\) thus defined in \(E_{\eta}^{\ell}\) is continuous and belongs to \(W^{1,p}(E^\ell_\eta ; N^n)\).

The definition of the extension \(u^{\mathrm{be}}_\eta\) on \(E_{\eta}^{m}\) is complete when \(p=m\). 
We now proceed assuming that \(p < m\); thus \(\ell \le m - 1\).
Let  \(\cS^{m-1}_\eta = \cE^{m-1}_\eta \cap \cG^{m-1}_\eta\), and we extend \(u^{\mathrm{be}}_{\eta}\) to \(S^{m-1}_\eta\) as a continuous Sobolev map by \(u^{\mathrm{pr}}_{\eta}\). 
This is possible since 
\(u^{\mathrm{pr}}_{\eta}=u^{\mathrm{be}}_{\eta}\) on \((E^{\ell}_\eta \cap S^{m-1}_\eta)\subset (E^{\ell}_\eta \cap G^{\ell}_\eta)\).
We now apply Proposition~\ref{proposition_homogeneisation_1} to   the map \(u^{\mathrm{be}}_{\eta} : E^\ell_\eta\cup S^{m-1}_\eta\to N^n\).
The resulting map, that we still denote by \(u^{\mathrm{be}}_{\eta}\), belongs to \(W^{1,p}(E^m_\eta; N^n)\), agrees with \(u^{\mathrm{pr}}_{\eta}\) on \( E^{m}_\eta\cap G^{m}_{\eta}=S^{m-1}_\eta\) and satisfies
\[
u^{\mathrm{be}}_{\eta}(E^m_\eta) \subset u^{\mathrm{be}}_{\eta}(E^\ell_\eta\cup S^{m-1}_\eta).
\]
In particular, we have \(u^{\mathrm{be}}_{\eta}\in L^{\infty}(E^{m}_\eta ; N^n)\). 
Finally, 
\[
\int_{E^m_\eta} |Du^{\mathrm{be}}_{\eta}|^p \leq {\C} \bigg( \eta^{m - \ell} \int_{E^\ell_\eta} |Du^{\mathrm{be}}_{\eta}|^p  + \sum_{i=\ell+1}^{m-1} \eta^{m - i}\int_{S^i_\eta} |Du^{\mathrm{pr}}_{\eta}|^p\bigg). 
\]
Since \(S^{i}_{\eta}\subset E^{i}_\eta \cap G^{i}_\eta\), the required estimate follows from the above inequality and \eqref{eq2189}.	
\end{proofclaim}

We deduce from Claims~\ref{lemma_usm_estimate}--\ref{claimEstimateProjectionBadCubes}  that
\begin{equation}\label{eq1350}
\resetconstant
\norm{Du^{\mathrm{be}}_\eta}_{L^{p}(E^{m}_\eta)} \leq \C \norm{Du}_{L^{p}(E^{m}_\eta + Q^{m}_{2\rho\eta})}. 
\end{equation}
We now complete the proof of the theorem.
For this purpose, let \((R_i)_{i \in \N}\) be a sequence of positive numbers diverging to infinity. 
Accordingly, Claims~\ref{claimDistance} and~\ref{lemma_distance_bis} yield a sequence \((\overline{R_i})_{i \in \N}\) from which we define 
a sequence of positive numbers \((\lambda_{R_i})_{i\in \N}\) satisfying \eqref{eq1300}. 
Finally, we take a sequence of positive numbers \((\eta_i)_{i \in \N}\) converging to zero such that
\[
\lim_{i\to \infty}{\frac{\eta_i}{\lambda_{R_i}}}=0.
\]

By Claim~\ref{claim_measure_bad}, we have
\begin{equation}\label{eq2533}
\lim_{i \to \infty}{\abs{E^m_{\eta_i}  + Q^m_{2\rho\eta_i}}} = 0.
\end{equation}
We proceed to prove that
\begin{equation}\label{eq2585}
\lim_{i \to \infty}{\norm{D u_{\eta_{i}}^{\mathrm{pr}} - Du}_{L^{p}(G_{\eta_i}^{m})}}
= 0.
\end{equation}
Indeed, from estimate \eqref{eq1590} in Claim~\ref{claimEstimateDirectDistance}, we have
\[
\norm{u^\mathrm{sm}_{\eta_i} - u}_{L^p(Q^{m}_{1+\gamma})}
\leq \sup_{v \in B_1^m}{\norm{\tau_{\psi_{\eta_i} v}(u) -  u}_{L^p(Q^{m}_{1+\gamma})}} + C \norm{u^\mathrm{op}_{\eta_i} - u}_{L^p(Q^{m}_{1+2\gamma})}.
\]
Since \(u= u^\mathrm{op}_{\eta_i}\) outside \(E^{m}_{\eta_i}+Q^{m}_{2\rho\eta_i}\), by the Poincar\'e inequality for functions vanishing on a set of positive measure and by property \eqref{eq2533} we have
\[
\norm{u^\mathrm{op}_{\eta_i}-u}_{L^p(Q^{m}_{1+2\gamma})} 
 \leq \C \norm{Du^\mathrm{op}_{\eta_i} - Du}_{L^p(Q^{m}_{1+2\gamma})}
 \leq \C \norm{Du}_{L^p(E^{m}_{\eta_i}+Q^{m}_{2\rho\eta_i})}.
\]
Hence, 
\[
\norm{u^\mathrm{sm}_{\eta_i} - u}_{L^p(Q^{m}_{1+\gamma})}
\leq \sup_{v \in B_1^m}{\norm{\tau_{\psi_{\eta_i} v}(u) -  u}_{L^p(Q^{m}_{1+\gamma})}} + \C \norm{ Du}_{L^p(E^{m}_{\eta_i}+Q^{m}_{2\rho\eta_i})},
\]
which proves that
\begin{equation}\label{eq2232}
\lim_{i\to \infty} \norm{u^{\mathrm{sm}}_{\eta_i}-u}_{L^p(Q^{m}_{1+\gamma})}=0.
\end{equation}
By the dominated convergence theorem, we thus get
\[
\lim_{i\to \infty}{\norm{\abs{D\Pi(u^{\mathrm{sm}}_{\eta_i})-D\Pi(u)}\,\abs{Du}}_{L^p(G^{m}_{\eta_i})}}
=0.
\]
By estimate \eqref{eq1594} in Claim~\ref{claimEstimateDirectDistance}, we also have
\[
\lim_{i\to \infty}{\norm{Du^{\mathrm{sm}}_{\eta_i} - Du}_{L^p(G^{m}_{\eta_i})}}
=0.
\]
Both limits
and Claim~\ref{claimEstimateProjectionGoodCubes}  imply \eqref{eq2585}.

Since \(u_{\eta_{i}}^{\mathrm{pr}} = u_{\eta_{i}}^{\mathrm{be}}\) on \(G_{\eta_i}^{m} \cap E_{\eta_i}^{m}\), the function obtained by juxtaposing \(u_{\eta_{i}}^{\mathrm{pr}}\) and \(u_{\eta_{i}}^{\mathrm{be}}\) defined by
\[
u_{\eta_{i}}^{\mathrm{jx}}(x)
= 
\begin{cases}
	u_{\eta_{i}}^{\mathrm{pr}}(x)	& \text{if \(x \in G_{\eta_i}^{m}\),}\\
	u_{\eta_{i}}^{\mathrm{be}}(x)	& \text{if \(x \in E_{\eta_i}^{m}\),}
\end{cases}
\] 
belongs to \((W^{1, p} \cap L^{\infty})(Q^{m}_{1 + \gamma}; N^{n})\). 
Moreover, by \eqref{eq1350}--\eqref{eq2585} we have
\begin{equation}\label{eq2261}
\lim_{i\to \infty} \norm{Du_{\eta_{i}}^{\mathrm{jx}}-Du}_{L^{p}(Q^{m}_{1+\gamma})}=0.
\end{equation}

Finally, we establish that
\begin{equation}\label{eq2262}
\lim_{i\to \infty} \norm{u_{\eta_{i}}^{\mathrm{jx}}-u}_{L^{p}(Q^{m}_{1+\gamma})}=0.
\end{equation}
To this aim, we introduce the auxiliary sequence \((v_i)_{i\in \N}\) in the space \(L^{p}(Q^{m}_{1+\gamma};\R^\nu)\) defined by
\[
v_i(x)
= 
\begin{cases}
	u_{\eta_{i}}^{\mathrm{sm}}(x)	& \text{if \(x \in G_{\eta_i}^{m}\),}\\
	u_{\eta_{i}}^{\mathrm{be}}(x)	& \text{if \(x \in E_{\eta_i}^{m}\).}
\end{cases}
\] 
Observe that \(u_{\eta_{i}}^{\mathrm{jx}}=\Pi\circ v_i\). 

As a consequence of \eqref{eq2232}, the sequence  \((u_{\eta_{i}}^{\mathrm{sm}})_{i\in \N}\) converges to \(u\) in measure on \(Q^{m}_{1+\gamma}\). In view of \eqref{eq2533}, this is also true for \((v_i)_{i\in \N}\). Since \(\Pi\) is continuous, \((u_{\eta_{i}}^{\mathrm{jx}})_{i\in \N}\) converges in measure to \(\Pi\circ u=u\) on  
\(Q^{m}_{1+\gamma}\). Together with \eqref{eq2261}, this implies \eqref{eq2262} as in the end of the proof of Proposition~\ref{propositionIntrinsicSobolev}.
This  completes the proofs  of Theorem~\ref{theoremMainNonInteger} and the sufficiency part of Theorem~\ref{theorem_Ap_CNS}; the necessity part follows from Proposition~\ref{proposition_thm_Ap_CN} above.
\qed



\begin{proof}[Proofs of Corollaries~\ref{corollaryDensitySmoothNonInteger} and~\ref{corollaryDensitySmoothInteger}]
\((\Longrightarrow)\) Let \(1\leq p \leq m\), and assume that the set \(C^{\infty}(\overline{Q^m} ; N^n)\) is dense in \(W^{1,p}(Q^m ; N^n)\). 
When \(p<m\), this implies that \(\pi_{\floor{p}}(N^n)\) is trivial as in the case when \(N^n\) is compact \cites{Schoen-Uhlenbeck, Bethuel-Zheng}, with the same proof.
When \(p\in \{2,\dotsc,m\}\), the set \((W^{1,p}\cap L^{\infty})(Q^m ; N^n)\) is then dense in \(W^{1,p}(Q^m ; N^n)\), and it follows from Proposition~\ref{proposition_thm_Ap_CN} that \(N^n\) satisfies the trimming property of dimension \(p\).  

\((\Longleftarrow)\) Conversely, if \(1\leq p \leq m\) is not an integer and \(\pi_{\floor{p}}(N^n)\) is trivial, then Proposition~\ref{proposition_class_R} implies that \(C^{\infty}(\overline{Q^m} ; N^n)\) is dense in \((W^{1,p}\cap L^{\infty})(Q^m ; N^n)\). It also follows from Theorem~\ref{theoremMainNonInteger}
that the set \((W^{1,p}\cap L^{\infty})(Q^m ; N^n)\) is dense in \(W^{1,p}(Q^m ; N^n)\). Hence, \(C^{\infty}(\overline{Q^m} ; N^n)\) is dense in \(W^{1,p}(Q^m ; N^n)\). This completes the proof of Corollary~\ref{corollaryDensitySmoothNonInteger}.

Finally, the sufficiency part of Corollary~\ref{corollaryDensitySmoothInteger} follows from Theorem~\ref{theorem_Ap_CNS} and Proposition~\ref{proposition_class_R} when $p\in \{1, \dotsc, m-1\}$, and from Theorem~\ref{theorem_Ap_CNS} and Proposition~\ref{proposition_smooth_bounded_density} when $p=m$.
This completes the proof of Corollary~\ref{corollaryDensitySmoothInteger}.
\end{proof}


\section*{Acknowledgments}
The authors would like to thank the referee for his or her detailed reading and comments that have improved the presentation of the paper and H.~Brezis for calling their attention to Marcus and Mizel's paper~\cite{Marcus-Mizel}.
Part of this work was carried out while the first author (PB) was visiting IRMP with support from UCL.
The second (ACP) and third (JVS) authors were supported by the Fonds de la Recherche scientifique---FNRS; ACP: Crédit de Recherche (CDR) J.0026.15 and JVS: Mandat d'Impulsion scientifique (MIS) F.452317.

\end{document}